\documentclass[reqno]{amsart}
\usepackage{diagrams}
\usepackage{amssymb}
\usepackage{mathrsfs}
\usepackage[colorlinks=true,allcolors = blue]{hyperref}
\usepackage{tikz-cd}
\usepackage{array}

\DeclareMathOperator\C{\mathbb C}
\DeclareMathOperator\Z{\mathbb Z}
\DeclareMathOperator\R{\mathbb R}

\newtheorem{theorem}{Theorem}[section]
\newtheorem{lemma}[theorem]{Lemma}
\newtheorem{cor}[theorem]{Corollary}
\newtheorem{prop}[theorem]{Proposition}
\theoremstyle{definition}
\newtheorem{definition}[theorem]{Definition}

\theoremstyle{remark}
\newtheorem{remark}[theorem]{Remark}

\newcommand{\dontprint}[1]\relax

\newcommand{\bs}{{\bf s}}
\newcommand{\bt}{{\bf t}}
\newcommand{\bv}{{\bf v}}
\newcommand{\bphi}{{\bf \varphi}}

\newcommand{\rk}{\operatorname{rk}}
\newcommand{\Cone}{\operatorname{Cone}}
\newcommand{\pa}{\partial}

\newcommand{\const}{\operatorname{const}}
\newcommand{\can}{\operatorname{can}}
\newcommand{\ev}{\operatorname{ev}}
\newcommand{\Th}{\Theta}
\newcommand{\De}{\Delta}
\newcommand{\Aut}{\operatorname{Aut}}
\newcommand{\und}{\underline}
\newcommand{\Pic}{\operatorname{Pic}}
\newcommand{\hra}{\hookrightarrow}
\newcommand{\we}{\wedge}

\newcommand{\bos}{\operatorname{bos}}
\newcommand{\Det}{\operatorname{Det}}
\renewcommand{\P}{{\mathbb P}}
\newcommand{\A}{{\mathbb A}}
\renewcommand{\AA}{{\mathcal A}}
\newcommand{\wt}{\widetilde}
\newcommand{\ot}{\otimes}

\newcommand{\ber}{\operatorname{ber}}

\newcommand{\Hom}{\operatorname{Hom}}

\newcommand{\per}{\operatorname{per}}
\newcommand{\Om}{\Omega}
\newcommand{\La}{\Lambda}
\newcommand{\eps}{\epsilon}
\newcommand{\HH}{{\mathcal H}}
\newcommand{\TT}{{\mathcal T}}
\newcommand{\NN}{{\mathcal N}}
\newcommand{\XX}{{\mathcal X}}
\newcommand{\WW}{{\mathcal W}}

\newcommand{\VV}{{\mathcal V}}
\newcommand{\DD}{{\mathcal D}}
\newcommand{\CC}{{\mathcal C}}

\renewcommand{\SS}{{\mathcal S}}
\newcommand{\FF}{{\mathcal F}}

\newcommand{\II}{{\mathcal I}}
\newcommand{\LL}{{\mathcal L}}
\newcommand{\MM}{{\mathcal M}}
\newcommand{\OO}{{\mathcal O}}

\newcommand{\UU}{{\mathcal U}}
\newcommand{\si}{\sigma}
\newcommand{\de}{\delta}
\newcommand{\sub}{\subset}
\newcommand{\Spec}{\operatorname{Spec}}
\newcommand{\Spf}{\operatorname{Spf}}
\newcommand{\Res}{\operatorname{Res}}
\newcommand{\ov}{\overline}
\newcommand{\im}{\operatorname{im}}

\newcommand{\om}{\omega}
\newcommand{\la}{\lambda}
\renewcommand{\a}{\alpha}
\renewcommand{\b}{\beta}

\newcommand{\tr}{\operatorname{tr}}
\newcommand{\id}{\operatorname{id}}

\newcommand{\G}{{\mathbb G}}
\renewcommand{\th}{\theta}
\newcommand{\ga}{\gamma}
\newcommand{\Ga}{\Gamma}
\newcommand{\lan}{\langle}
\newcommand{\ran}{\rangle}

\newcommand{\Disc}{{\mathbb D}}

\newcommand{\norm}{{\operatorname{norm}}}
\newcommand{\Ber}{{\operatorname{Ber}}}
\newcommand{\End}{{\operatorname{End}}}
\newcommand{\vol}{\mathfrak{vol}}
\newcommand{\ssl}{\mathfrak{sl}}
\newcommand{\SL}{{\operatorname{SL}}}
\newcommand{\gr}{{\operatorname{gr}}}
\newcommand{\Conf}{{\operatorname{Conf}}}

\numberwithin{equation}{section}

\title[Superperiods and superstring measure]{Superperiods and superstring measure near the boundary of the moduli space of supercurves}

\author{Giovanni Felder}
\address{Department of mathematics,
ETH Zurich, 8092 Zurich, Switzerland}
\email{giovanni.felder@math.ethz.ch}
\author{David Kazhdan}
\address{Einstein Institute of Mathematics,
The Hebrew University of Jerusalem,
Jerusalem 91904, Israel}
\email{kazhdan@math.huji.ac.il}
\author{Alexander Polishchuk}
\address{
    Department of Mathematics, 
    University of Oregon, 
    Eugene, OR 97403, USA; National Research University Higher School of Economics, Moscow, Russia
  }
  \email{apolish@uoregon.edu}

\begin{document}

\begin{abstract} We study the behavior of the superperiod map near the boundary of the moduli space of stable supercurves
and prove that it is similar to the behavior of periods of classical curves. We consider two applications to the geometry of this moduli space in genus $2$, denoted as $\ov{\SS}_2$.
First, we characterize the canonical projection of $\SS_2$ in terms of its behavior near the boundary, proving in particular that $\ov{\SS}_2$ is not projected.
Secondly, we combine the information on superperiods with the explicit calculation of genus $2$ Mumford isomorphism, due to Witten, 
to study the expansion of the superstring measure for genus $2$ near the boundary. We also present the proof, due to Deligne, of regularity of the superstring measure on $\SS_g$
for any genus.
\end{abstract}

\maketitle

\tableofcontents

\section{Introduction}

\subsection{Some background and motivation}

We refer to \cite{Manin} and \cite{BHRP} for basics on superschemes.
Recall that a smooth supercurve (aka SUSY curve) over a superscheme $S$ is a smooth map $X\to S$ of relative dimension $1|1$, together with
a distribution $\DD\sub \TT_{X/S}$ of rank $0|1$, such that the map $\DD\ot_{\OO_X} \DD\to \TT_{X/S}/\DD$, induced by the Lie bracket of vector fields, is an isomorphism.

A supercurve $X$ over a point is given by the usual spin curve $(C,L,\kappa\colon
L^2\stackrel{\sim}\longrightarrow \om_C)$, so that $\OO_X=\OO_C\oplus L$.
Recall that the spin structure $L$ is called {\it even} or {\it odd} depending on the parity of $h^0(C,L)$ (and this parity is constant in families).
We denote by $\SS_g$ the moduli stack of supercurves of genus $g$ with even underlying spin structures.
This moduli space has been studied in many papers, including \cite{CraneRabin1988, Deligne-letter, DRS, LeBrunRothstein1988, DW, CV, FKP-per}.

We denote by $\om_{\SS_g}$ the canonical line bundle on $\SS_g$ (i.e., the Berezinian of the cotangent bundle).
The superanalog of Mumford's isomorphism is an isomorphism 
$$\Psi_g:\Ber_1^5\rTo{\sim}\om_{\SS_g},$$ 
where $\Ber_1$ is obtained as the Berezinian of $R\pi_*(\OO_X)$, where $\pi:X\to \SS_g$
is the universal curve (see \cite{Voronov}, \cite{RoslySchwarzVoronov1989}). 
The {\it string supermeasure} $\mu$ 
is a meromorphic section of the Berezinian on $\SS_g\times \SS_g^c$, where $\SS_g^c$ denotes the complex conjugate of $\SS_g$, 
defined near the {\it quasidiagonal}
(i.e., pairs of spin-curves $(C_1,L_2)$, $(C_2,L_2)$ with $C_1\simeq C_2$) using the Mumford form $\Psi_g$ and a natural hermitian form on $\pi_*\om_{X/\SS_g}$.
We will recall the  definition of $\mu$ in Sec.\ \ref{string-sec} (see also \cite{RoslySchwarzVoronov1989}, \cite{Witten}). 

The hermitian form on $\pi_*\om_{X/\SS_g}$ is closely related to the 
superperiod map, which is a map from an unramified covering 
(corresponding to a choice of a symplectic basis in $H^1(C,\C)$)
of an open substack of $\SS_g$ to the Lagrangian Grassmannian $LG(g,2g)$ given by the subbundle
$\pi_*\om_{X/\SS_g}$ in $R^1\pi_*\C_{X/\SS_g}$, where $\pi:X\to \SS_g$ is the universal curve, $\om_{X/\SS_g}$ is the relative Berezinian.
More precisely, the superperiod map is only defined away from the theta-null divisor, on which the underlying spin structure $L$ satisfies $H^0(C,L)\neq 0$.
As a consequence, a priori one only knows regularity of
the supermeasure $\mu$ away from the theta-null divisor.
In \cite{FKP-per} we proved that $\mu$ extends regularly across this divisor for $g\le 11$.
Recently, Pierre Deligne gave a simple proof of the regularity of $\mu$ for any genus $g$. We reproduce his argument in Section \ref{regularity-sec}.

The contribution of genus $g$ to the vacuum amplitude of type II perturbative superstring theory is supposed to be given as the integral
of the supermeasure $\mu$ over a suitable cycle in $\SS_g\times \SS_g^c$. To make sense of this it is important to study the behavior of $\mu$ at infinity.
More precisely, we use the natural compactification $\ov{\SS}_g$ of $\SS_g$ given by the moduli space of {\it stable supercurves}
(see \cite{Deligne-letter}, \cite{FKP-supercurves} and \cite{MZ} for basic definitions and results concerning this). 
A stable supercurve over a point (more generally, over an even base) is the same as a usual stable curve $C$ equipped with a (generalized) spin structure,
i.e., a torsion-free coherent sheaf $L$ with an isomorphism $L\rTo{\sim}\und{\Hom}(L,\om_C)$, where $\om_C$ is the dualizing sheaf on $C$.

Recall that for a spin structure $L$ on a stable curve $C$ there are two possible behaviors of $L$ at a node $q\in C$. Namely, $q$ is called a {\it Ramond} node if $L$ is locally free near $q$,
and it is called {\it Neveu-Schwarz (NS)} otherwise. In \cite{FKP-supercurves} we define natural normal crossing 
Cartier divisors $\De$, $\De_{NS}$ and $\De_R$ in $\ov{\SS}_g$,
such that $\De=\De_{NS}+\De_R$; $\De_{NS}$ is supported on stable supercurves with at least one NS node and $\De_R$ is supported on stable supercurves with
at least one Ramond node.

We showed in  \cite{FKP-supercurves} that the Mumford isomorphism $\Psi_g$ extends to an isomorphism
$$\om_{\ov{\SS}_g}\simeq \Ber_1^5(-2\De_{NS}-\De_R).$$
Equivalently, the Mumford form $\Psi_g\in \om_{\SS_g}\ot \Ber_1^{-5}$ which is non-vanishing everywhere on $\SS_g$, acquires poles of order $2$ at $\De_{NS}$ and
poles of order $1$ at $\De_R$.
The goal of this paper is to study in more details the behavior of the superperiods, of the Mumford form and of the superstring measure near the boundary of
the moduli space $\ov{\SS}_g$ (mostly restricting to the case $g=2$). In particular, we will establish rigorously some of the results of \cite{Witten}.



\subsection{Our results}


\subsubsection{Superperiods near the boundary}



We show that like in the case of usual curves, the leading entries of the superperiod matrix $\Om$ grow logarithmically near the non-separating node
components of the boundary of $\ov{\SS}_g$ and that $\det(\Om-\ov{\Om})$ grows logarithmically. 
We show this mimicking the classical case, using theory of $D$-modules on supervarieties developed by Penkov \cite{Penkov}
and extending to the super case Deligne's construction of the canonical extension of a connection (see Sec.\ \ref{superperiod-sec}).



\subsubsection{Gluing coordinates}

The rest of our results are specific for the case $g=2$. We consider one of the connected components of $\De_{NS}$, the separating node $(+,+)$ boundary divisor $D_0$.
It corresponds to nodal curves with two irreducible components of arithmetic genus $1$, where each component is equipped with an even spin structure.
Similarly to the classical case (see \cite{P-bu}), we introduce {\it gluing coordinates} in a formal neighborhood of $D_0$ 
(see Sec.\ \ref{super-gluing-sec}). These are formal coordinates of special type along $D_0$, respecting geometric structures near $D_0$ (in particular, $D_0$ is
given by $t=0$ where $t$ is one of the coordinates).
We calculate  the first few terms of the expansion of the entries of the superperiod matrix $\Om$ in terms of these coordinates (see Sec.\ \ref{ss-2.2}).

Note that there exists a gluing construction of curves in a formal neighborhood of $D_0$ in higher genus but it depends on extra choices. Namely, if we start with
a pair of supercurves $(X_1,p_1)$, $(X_2,p_2)$ with smooth marked points, over the same base $B$, one can glue $X_1$ and $X_2$ nodally along $p_1$ and $p_2$ 
into a stable supercurve $X_0$ over $B$ with an NS node. If moreover, $X_i$ are equipped with formal (superconformal) relative coordinates along $p_i$, then we can extend
$X_0$ to a family $X$ over the product of $B$ with the formal (even) disk. What makes the case of genus $2$ special, is that in this case $X_1$ and $X_2$ have genus $1$,
so there exists canonical formal parameters depending on a choice of trivializations of the spin structures at the marked points, so the gluing can be controlled by the 
relevant $\G_m$-torsors.


\subsubsection{Canonical projection for genus $2$ near the boundary}

We apply our results on superperiods
to the study of the canonical projection $\pi^{\can}:\SS_2\to \SS_{2,\bos}$ induced by the superperiod map (see Sec.\ \ref{can-proj-sec}).
Note that there exists no projection on $\SS_g$ for $g\ge 5$ (see \cite{DW}), and there seems to be no canonical projection for $g>2$ (for $g=3$ the
projection given by the superperiod map is only defined away from the hyperelliptic locus).

We show that $\pi^{\can}$ extends to a regular projection on $\ov{\SS}_2$ 
away from the $(-,-)$ separating node divisor $D^{-,-}$, i.e., the divisor corresponding to reducible curves with a pair of odd spin structures on its components
(see Theorem \ref{projection-nonsep-prop}). In fact, we show that $\pi^{\can}$ is the unique projection
regular at a generic point of each boundary divisor corresponding to a non-separating node. As a consequence, we derive that $\ov{\SS}_2$ is not projected.

More precisely, we have the following behavior of $\pi^{\can}$ near the boundary in $\ov{\SS}_2$:
\begin{itemize}
\item near the non-separating node divisors (both NS and Ramond): regular and compatible with the divisor;
\item near $D_0$: regular but not compatible with the divisor.
\item near $D^{-,-}$: not regular.
\end{itemize}
Here compatibility with a component of the boundary divisor means that this component coincides with the pullback of its bosonization under the projection.
Note that compatibility of projections with boundary divisors of a compactification comes up naturally when integrating densities with noncompact support on supermanifolds.

In addition, we compute the first few terms of the expansion of $\pi^{\can}$ in terms of the gluing coordinates  in a formal neighborhood of $D_0$.
We also study how the superperiods and hence $\pi^{\can}$ behave near the boundary component $D^{-,-}$ (see Corollary \ref{--cor}).


\subsubsection{Mumford form for genus $2$ near the boundary}



We justify Witten's assumptions used in the calculation in \cite{Witten}, which determines the Mumford's form $\Psi_2$ in the hyperelliptic model of genus $2$ spin-curves.
Recall that a genus $2$ curve can be realized as a double covering of $\P^1$ ramified at $6$ points. A choice of an even spin structure corresponds to splitting the ramification
locus into two groups of $3$ points. The Mumford form $\Psi_2$ is determined by its restriction to $\SS_{2,\bos}$ and its push-forward $\pi^{\can}_*\Psi_2$. Both of these 
objects live on the moduli space of spin curves $\SS_{2,\bos}$ and Witten calculates them in terms of $6$ ramification points of
the hyperelliptic covering $C\to \P^1$ (which are split into two groups of $3$ points). The main idea of the calculation is to
use $\SL_2$-invariance with respect to the action of $\SL_2$ on the configurations of $6$ points in $\P^1$ (this corresponds to the fact that our objects live on the moduli space $\SS_{2,\bos}$) 
and the behavior near infinity, where two of the ramification points merge. We show that the needed behavior at infinity follows from our results on the poles of $\Psi_2$ from
\cite{FKP-supercurves}, as well as from the compatibility of the canonical projection $\pi^{\can}$ with the non-separating node boundary components (see Proposition  \ref{p-Witten}). 
We then use Witten's formulas to find the first terms of the expansion of the Mumford form $\Psi_2$ near the $(+,+)$ separating node boundary divisor $D_0$ in $\ov{\SS}_2$ 
(see Sec.\ \ref{hyper-Mum-sec}).





\subsubsection{Superstring measure for genus $2$ near the boundary}


Recall that the Mumford isomorphism gives an isomorphism of the holomorphic Berezinian of $\SS_g$ with the line bundle $\Ber_1^5$.
One can get rid of $\Ber_1^5$ using a canonical hermitian form on $\pi_*\om_{X/\SS_g}$, which leads to a definition of the superstring measure $\mu$ (see Sec.\ \ref{string-sec}).
Combining the results of previous calculations we find the first terms of the expansion of the superstring measure $\mu$ near the divisor $D_0$ in
terms of gluing coordinates (see Cor.\ \ref{measure-polar-term-cor}). 
These gluing coordinates define a projection defined on the formal neighborhood of $D_0$ (which is different from $\pi^{\can}$),
and the push-forward of $\mu$ under this projection has a pole of order $1$ along $D_0$, with the residue that can be expressed in terms of genus $1$ data.

We plan to use our results in a subsequent work to get a rigorous definition of the integral of $\mu$, using a regularization procedure and a cancelation of second order poles
due to summation over different spin structures. 

\subsection*{Conventions} 
We work with algebraic superschemes over $\C$. When discussing superperiods we use classical topology. 
For a superscheme $S$ we denote by $S_{\bos}$ the scheme with the same underlying topological space and with the sheaf of functions $\OO_S/\NN_S$,
where $\NN_S$ is the ideal sheaf locally generated by odd functions.
By a projection $S\to S_{\bos}$ we mean a morphism, inducing the identity morphism on $S_{\bos}$.
By a vector bundle on a superscheme we mean a ($\Z_2$-graded) locally free $\OO_X$-module of finite rank. 
We denote by $F\mapsto \Pi F$ the functor of change of parity on such bundles.

\subsection*{Acknowledgements}
G.F. was supported in part by the National Centre
of Competence in Research SwissMAP (grant number 205607) of the Swiss
National Science Foundation. 
D.K. is partially supported by the ERC grant 101142781.
A.P. was partially supported by the NSF grants DMS-2001224, 
NSF grant DMS-2349388, by the Simons Travel grant MPS-TSM-00002745,
and within the framework of the HSE University Basic Research Program.
G.F and A.P. also wish to thank the Institut des
Hautes Etudes Scientifiques, where part of this work was done, for
hospitality.

\section{Regularity of the superstring measure near the theta-null divisor}\label{regularity-sec}

\subsection{Superstring measure}\label{string-sec}

Here we briefly review the definition of the superstring measure following \cite[Sec.\ 5]{FKP-per}.

Recall (see \cite[Sec.\ 5.1]{FKP-per}) that with a complex supermanifold $X=(|X|,\OO_X)$, one can associate 
the complex conjugate supermanifold $X^c$ whose underlying topological space is $|X|$, such
that $\OO_{X^c}=\ov{\OO_X}$, which is the same sheaf $\OO_X$ but with the $\C$-algebra structure differing by the complex conjugation.
This operation preserves \'etale maps, so it passes
to complex orbifolds. Note that there is a natural functor from the category of $\OO_X$-modules to that of $\OO_{X^c}$-modules, which is $\C$-antilinear on morphisms.

Let $\SS_g$ denote the moduli stack of smooth supercurves of genus $g$ (more precisely, the component corresponding to even spin structures), 
$\pi:X\to \SS_g$ the universal supercurve, $\SS_g^c$ the complex conjugate. 

We define the quasidiagonal $\De^q_{\bos}$ as the fibered product of real orbifolds
$$\De^q_{\bos}:=\SS_{g,\bos}\times_{\MM_g} \SS_{g,\bos}^c,$$ 
with respect to the natural map $\SS_{g,\bos}\to \MM_g$ forgetting the spin structure.
We can think of $\De^q_{\bos}$ as classifying pairs of spin curves with the same underlying curve.

The superstring measure $\mu$ is a meromorphic
section of the holomorphic Berezinian $\om_{\SS_g}\boxtimes \om_{\SS_g^c}$ on $\SS_g\times \SS_g^c$
defined in a neighborhood of the image of $\De^q_{\bos}$ in $\SS_g\times \SS_g^c$,
regular away from the locus where one of the spin structures has a global section. 
 
The holomorphic ingredient for $\mu$ is the {\it Mumford form} $\Psi=\Psi_g$ which is a section of $\Ber_1^{-5}\ot \om_{\SS_g}$,
corresponding to the canonical isomorphism
\begin{equation}\label{hol-Mumf-isom-eq}
\Ber_1^5\rTo{\sim} \om_{\SS_g},
\end{equation}
where $\Ber_1:=\Ber R\pi_*(\om_{X/\SS_g})$.
We denote by $\wt{\Psi}\in H^0(\wt{\Ber}_1^{-5}\ot \om_{\SS_g^c})$ the complex conjugate of $\Psi$ on $\SS_g^c$.

The superstring measure, which is a meromorphic section of $\om_{\SS_g\times\SS_g^c}$ is given by 
\begin{equation}\label{mu-Psi-h-eq}
\mu=\Psi\cdot \wt{\Psi}\cdot h^5.
\end{equation}
where $h$ is a section of $\Ber_1\boxtimes\wt{\Ber}_1$ defined as follows\footnote{In \cite{FKP-per} this section $h$ is denoted as $s$}.
Let $\UU\sub \SS_g$ denote the open substack where the underlying spin structure has no nonzero global sections.
Let $V$ denote the local system $R^1\pi_*(\R_X)$ on $\SS_g$. We have $R^1\pi_*(\C_{X/\SS_g})\simeq V\ot_{\R}\OO_{\SS_g}$, where $\C_{X/\SS_g}=\pi^{-1}\C_S$.
The exact sequence
\begin{equation}\label{de-sh-ex-seq}
0\to \C_{X/\SS_g}\to \OO_X\rTo{\de} \om_{X/\SS_g}\to 0
\end{equation}
induces a morphism $\pi_*\om_{X/\SS_g}\to V\ot_{\R}\OO_{\SS_g}$.
The restriction of this morphism to $\UU$, 
$$\La_\UU:=\pi_*\om_{X/\SS_g}|_\UU\to V\ot_{\R}\OO_{\UU},$$
in an embedding of a subbundle.
We have the conjugate morphism $\wt{\La}_\UU\to V\ot_{\R}\OO_{\UU^c}$ on $\UU^c\sub \SS_g^c$.

On the other hand, we have a natural symplectic pairing $V\ot V\to \R_{\SS_g}$, and
an identification $p_1^{-1}V\simeq p_2^{-1}V$ near the quasidiagonal in $\SS_{g,\bos}\times \SS_{g,\bos}^c$.
Thus, in a neighborhood of the quasidiagonal we have a pairing
$$p_1^{-1}V\ot p_2^{-1}V\to \und{\R},$$
which leads to a nondegenerate pairing in a neighborhood of the quasidiagonal
$$p_1^*\La_\UU\ot p_2^*\wt{\La}_\UU\to p_1^{-1}V\ot p_2^{-1}V\ot \OO_{\UU\times\UU^c}\to \OO_{\UU\times\UU^c}.$$
Taking the determinant of the corresponding morphism
$$p_1^*\La_\UU\to p_2^*\wt{\La}^\vee_\UU$$
we get a map $p_1^*\Ber_1\to p_2^*\wt{\Ber}_1^{-1}$, 
or equivalently a section of $(\Ber_1\boxtimes \wt{\Ber}_1)^{-1}$, defined in a neighborhood of the quasidiagonal in $\UU\times \UU^c$.
We define $h$ to be the inverse of this section.

Locally we can choose a Lagrangian splitting $V=W\oplus W'$ and represent the image of $\La_\UU$ as the graph of a symmetric morphism
$$\Om:W'\ot \OO\to W\ot \OO$$
($\Om$ is essentially what is called the {\it superperiod matrix}).
This leads to a trivialization $s\in \Det(W)\ot \Ber_1$, and the complex conjugate $\wt{s}\in \Det(W)\ot \wt{\Ber}_1$. 
We have (see \cite[Eq.\ (5.4)]{FKP-per}\footnote{in \cite{FKP-per} we use $\tau$ instead of $\Om$.})
\begin{equation}\label{h-def-eq}
h=\det(p_1^*\Om-p_2^*\wt{\Om})^{-1}\cdot p_1^*s\ot p_2^*\wt{s},
\end{equation}
where we use the identification $p_1^{-1}W\simeq p_2^{-1}W$ near the quasidiagonal and the duality $W'\simeq W^\vee$, so
that $\det(p_1^*\Om-p_2^*\wt{\Om})^{-1}$ is viewed as a section of 
$$p_1^*(\Det(W)^{-1}\ot \Det(W'))\simeq p_1^*\Det(W)^{-1}\ot p_2^*\Det(W)^{-1}.$$


\subsection{Regularity on the smooth locus (after Deligne)}

The following theorem has been communicated to the authors by Pierre Deligne.
Below we are presenting his proof with some details added.

\begin{theorem}\cite{D24}\label{reg-thm}
For any $g\ge 2$, the section $h$ of $\Ber_1\boxtimes \wt{\Ber}_1$ is regular on
a neighborhood of the quasidiagonal in $\SS_g\times \SS_g^c$.
Hence, the superstring measure $\mu$ is regular on such a neighborhood.
\end{theorem}

\begin{proof}
Since $\Psi$ is an isomorphism, the last statement statement is equivalent to the regularity of the section $h^5$. 
Hence, it is enough to prove that $h$ itself is regular.
First, let us rewrite the definition of $h$ slightly. Let us fix an open neighborhood $\WW$ of the quasidiagonal in $\SS_g\times \SS_g^c$,
and an isomorphism $p_1^{-1}V\simeq p_2^{-1}V$ over $\WW$,
and set $\VV:=\UU\times \UU^c\cap \WW$.
First, we notice that $h=\Ber(H_\VV[-1])$, where $H_\VV$ is
the following composition of morphisms on $\VV$:
\begin{equation}\label{H-UU-comp-eq}
H_\VV:p_1^*\La_\UU\rTo{\a_\UU} p_1^{-1}V\ot \OO_{\VV}\rTo{\nu} p_2^{-1}V^\vee\ot \OO_{\VV}\rTo{\wt{\a}_\UU^\vee}
p_2^*\wt{\La}^\vee_\UU,
\end{equation}
where $\nu$ corresponds to the pairing on $V$.

Next, we want to extend $\a_\UU$ to a morphism in the derived category of coherent sheaves on $\SS_g$. 
The exact sequence \eqref{de-sh-ex-seq} gives a morphism in the derived category,
$$\wt{\a}:R\pi_*\om_{X/\SS_g}\to R\pi_*\C_{X/\SS_g}[1].$$
Let us define the object $\La$ in the derived category of $\SS_g$
from the exact triangle
$$\La\to R\pi_*\om_{X/\SS_g}\to (\tau_{\ge 2}R\pi_*\C_{X/\SS_g})[1]\to\ldots.$$
where the second arrow is the composition of $\wt{\a}$ with the truncation $R\pi_*\C_{X/\SS_g}[1]\to \tau_{\ge 2}R\pi_*\C_{X/\SS_g}[1]$.
Since $\tau_{\ge 2}R\pi_*\C_{X/\SS_g}\simeq \OO_{\SS_g}[-2]$, we have an identification
$$\Ber(\La)\simeq \Ber(R\pi_*\om_{X/\SS_g})\simeq \Ber_1.$$

The composition of $\wt{\a}$ with the truncation $R\pi_*\C_{X/\SS_g}[1]\to \tau_{\ge 1}R\pi_*\C_{X/\SS_g}[1]$
induces a morphism
$$\a:\La\to R^1\pi_*\C_{X/\SS_g}=V\ot \OO_{\SS_g},$$
whose restriction to $\UU$ is $\a_\UU$.
Now we can extend the map \eqref{H-UU-comp-eq} to a map in the derived category of $\WW$,
defined as the composition
\begin{equation}\label{H-comp-eq}
H:p_1^*\La|_{\WW}\rTo{\a} p_1^{-1}V\ot \OO_{\WW}\rTo{\nu} p_2^{-1}V^\vee\ot \OO_{\WW}\rTo{\wt{\a}^\vee}
p_2^*\wt{\La}^\vee|_{\WW}.
\end{equation}
Let us consider the object $\Cone(H)[-1]$ in the derived category of $\WW$, fitting into an exact triangle
$$\Cone(H)[-1]\to p_1^*\La|_{\WW}\rTo{H} p_2^*\wt{\La}^\vee|_{\WW}\to \ldots$$
We have an identification
$$\Ber(\Cone(H)[-1])\simeq \Ber_1\boxtimes \wt{\Ber}_1.$$
The main point of the proof is the following

\medskip

\noindent
{\bf Claim}. Locally there exist a morphism $A\to B$ of vector bundles of rank $0|m$ such that $\Cone(H)[-1]$ is quasi-isomorphic to the complex
$[A\to B]$ concentrated in degrees $0$ and $1$.

\medskip

Let us show how this claim implies the result. The complex $\Cone(H)[-1]|_{\VV}$ is acyclic, so we have a canonical isomorphism 
$$\ber:\OO_\VV\rTo{\sim} \Ber(\Cone(H)[-1])|_\VV\simeq \Ber_1\boxtimes \wt{\Ber}_1|_\VV.$$
Note that we have an identification of $\Cone(H)[-1]|_{\VV}$ with the cone of the isomorphism of vector bundles in cohomological degree $1$,
\eqref{H-UU-comp-eq}, hence $\ber=h$, the rational section we are interested in.
On the other hand, a local presentation of $\Cone(H)[-1]$ with $[A\rTo{\phi} B]$ as in the claim gives an isomorphism
\begin{equation}\label{Ber-cone-AB-eq}
\Ber(\Cone(H)[-1])\simeq \Ber(A)\ot \Ber(B)^{-1}\simeq \det(\Pi B)\ot \det(\Pi A)^{-1},
\end{equation} 
and an identification of $\Cone(H)[-1]|_{\VV}$ with an isomorphism of vector bundles
$A|_\VV\to B|_\VV$, which shows that under the isomorphism \eqref{Ber-cone-AB-eq}, the element $\ber$ corresponds
to the determinant of a morphism of even vector bundles 
$$\det(\Pi\phi:\Pi A\to \Pi B)\in \det(\Pi B)\ot \det(\Pi A)^{-1}.$$ 
Since the latter element is regular, the assertion follows.

It remains to prove the claim. To this end, we calculate the cohomology of the derived restriction $\Cone(H)[-1]|_s$, where $s\in S$ is a closed point.
The formation of $H$ is compatible with base changes, so we can start with a pair of smooth supercurves $X_s$, $\wt{X}_s$ given by pairs $(C,L)$, $(\wt{C},\wt{L})$, 
where $C$ (resp., $\wt{C}$) is a curve and $L$ (resp., $\wt{L}$) is a spin structure, such that $\wt{C}$ is complex conjugate to $C$.

We have $\om_{X_s}\simeq \om_C\oplus L$ and the map $H^1(C,\om_{X_s})\to H^2(C,\C)$ is an isomorphism of even parts. Hence, we get
$$H^0(\La)\simeq H^0(C,\om_C)\oplus H^0(C,L), \ \ H^1(\La)\simeq H^1(C,L),$$
and similarly 
$$H^{-1}(\wt{\La}^\vee)\simeq H^1(\wt{C},\wt{L})^\vee, \ \ H^0(\wt{\La}^\vee)\simeq H^0(\wt{C},\om_{\wt{C}})^\vee\oplus H^0(\wt{C},\wt{L})^\vee.$$
Hence, we get an exact sequence for the cohomology of $\Cone(H_s)[-1]$
\begin{align*}
  0&\to H^1(\wt{C},\wt{L})^\vee\to H^0(\Cone(H_s)[-1])\to H^0(C,\om_C)\oplus H_0(C,L)
  \\
  &\rTo{\nu} H^0(\wt{C},\om_{\wt{C}})^\vee\oplus H^0(\wt{C},\wt{L})^\vee
\to H^1(\Cone(H_s)[-1])\to H^1(C,L)\to 0,
\end{align*}
where the even part of $\nu$, $\nu^+:H^0(C,\om_C)\to H^0(\wt{C},\om_{\wt{C}})$ is an isomorphism corresponding to the natural hermitian pairing on 
$H^0(C,\om_C)$. From this it follows that $H^i(\Cone(H_s)[-1])=0$ for $i\neq 0,1$ and that $H^*(\Cone(H_s)[-1])$ is purely odd. This implies that locally $\Cone(H)[-1]$ can be
represented by a complex of the form $[A\to B]$ where $A$ and $B$ are vector bundle of purely odd rank. Since we know that the restriction of $\Cone(H)[-1]$
to $\VV$ is zero, we see that $A$ and $B$ are of the same rank.
\end{proof}

\begin{remark}\label{theta-null-rem}
Theorem \ref{reg-thm} improves \cite[Thm.\ 5.2]{FKP-per} where we proved the regularity of $h^5$ (or equivalently of $\mu$) for $g\le 11$. 
In fact, the proof of \cite[Thm.\ 5.2]{FKP-per} shows that locally near some point $s$ of  the complement $\SS_g\setminus \UU$,
$h^5$ (and hence $\mu$) is divisible by $(f\wt{f})^{11-g}$, where $f$ is a regular holomorphic function on a neighborhood of $s$ in $\SS_g$ defined as follows.
For every choice of a Largangian subbundle 
$\La\sub R^1\pi_*\C_{X/\SS_g}$ near $s$,
such that $\La_s$ is transversal to $H^0(\om_{C_s})\sub H^1(C_s,\C)$,
one has a section
$$\th_{\La}:=\th(\pi_*\om_{X/\SS_g},\La)\in \Gamma(\UU,\Det(\La)^{-1}\ot \Ber_1^{-1})$$
defined as the Berezinian of the morphism 
$$\pi_*\om_{X/\SS_g}|_{\UU}\to R^1\pi_*\C_{X/\SS_g}/\La|_{\UU}.$$
Trivializing $\La$ locally, we can think of $\th_{\La}$ as a meromorphic section of $\Ber_1^{-1}$ near $s$.
By \cite[Thm.\ 4.14]{FKP-per}, $\th_{\La}^{-1}$ is regular near $s$, and in fact, has form $\th_{\La}^{-1}=f^2\cdot \bt$
for some regular even function $f=f_{\La}$, where $\bt$ is a local trivialization of $\Ber_1$.
\end{remark}

\section{Superperiod map near the non-separating node boundary}\label{superperiod-sec}

In this section we study the behavior of the superperiod map near the non-separating node boundary divisor.
For this we use the theory of $D$-modules on superschemes. We generalize to the super case Deligne's theory of
canonical extensions of local systems. Then, similarly to the classical case, we identify explicitly the canonical extension
for the Gauss-Manin connection associated with a degenerating family of stable supercurves.

\subsection{$D$-modules on superschemes}

Penkov in \cite{Penkov} defined for $D$-modules on superschemes the analogs of pull-back and push-forward, and showed
that the pull-back and push-forward with respect to the embedding $i:S_{\bos}\to S$ are mutually inverse equivalences
of the categories of $D$-modules.

Note that $Li^*\OO_S=\OO_{S_{\bos}}$ as $D$-modules, so the push-forward of $D$-modules, $i_+$, satisfies
$$Ri_+(\OO_{S_{\bos}})\simeq \OO_S.$$

Thus, if $f:X\to S$ is a smooth proper morphism of superschemes then $Rf_+(\OO_X)$ is the $D$-module $i_+Rf_{red,+}(\OO_{X_{\bos}})$,
where $f_{\bos}:X_{\bos}\to S_{\bos}$ is the induced morphism of usual schemes.
In other words, for some $i\ge 0$, we have 
$$\VV:=R^if_+(\OO_X)\simeq \OO_S\ot_{\C} R^i_{f_{red *}}(\C_{X_{\bos}}),$$
the $D$-module associated with the local system, $R^i_{f_{red *}}(\C_{X_{\bos}})$

Here is an explicit local recipe for computing $i_+$. Assume $S$ is split: $\OO_S=\OO_{S_{\bos}}\ot_{\C}\bigwedge(E)$, where $E$ is a finite dimensional vector space.
Then we have an isomorphism of sheaves of algebras
$$D_S=D_{S_{\bos}}\ot_{\C}\End(\bigwedge(E)),$$
and an isomorphism of $D_S$-modules
$$i_+(M)=M\ot_{\C}\bigwedge(E),$$
for any $D_{S_{\bos}}$-module $M$.

\begin{remark} If we have a splitting $\OO_S=\OO_{S_{\bos}}\ot_{\C}\bigwedge(E)$, it makes sense to talk about
{\it flat connection in odd directions} on a sheaf of $\OO_S$-modules $\FF$, $\nabla_E:\FF\to \FF\ot E$.
It is easy to see that this is equivalent to having a structure of $\End(\bigwedge(E))$-module on $\FF$, so
there is an equivalence $\FF_{\bos}\to \FF_{\bos}\ot_{\C}\bigwedge(E)$ between the category of $\OO_{S_{\bos}}$-modules
and the category of $\OO_S$-modules with a flat connection in odd directions. Note that the connection in odd directions
on $\FF_{\bos}\ot \bigwedge(E)$ is induced by the natural map 
$$\kappa_E:\bigwedge(E)\to E\ot \bigwedge(E),$$
dual to the wedge multiplication.
\end{remark}

\subsection{Canonical extensions}


Assume that $S=\ov{S}\setminus B$, where $B$ is a normal crossing divisor on a smooth superscheme $\ov{S}$,
and let $(\VV,\nabla)$ be a vector bundle with a flat connection on $S$.

\begin{definition} We say that a vector bundle $\ov{\VV}$ on $\ov{S}$ is a {\it canonical extension} of $\VV$ with respect to $B$, 
if $\nabla$ extends to a connection 
$$\ov{\nabla}:\ov{\VV}\to \Om^1_{\ov{S}}(\log B)\ot \ov{\VV}$$
with logarithmic poles along $B$,
such that the residues of the corresponding logarithmic connection on the branches of $B$ (evaluated at any point) have eigenvalues in $[0,1)$.
\end{definition}

\begin{definition} Let $f$ be an even function on $U\cap S$, where $U$ is a neighborhood of a point $p\sub B$.
We say that $f$ is {\it bounded} near $p$ if there exists a local system of coordinates near $p$, $x_1,\ldots,x_n,\eta_1,\ldots,\eta_m$,
such that $f$ is a polynomial in $\eta_1,\ldots,\eta_m$ with coefficients given by bounded functions in $x_1,\ldots,x_n$.
We say that $f$ has logarithmic growth near $p$ if for a local equation $t=0$ of $B$, there exists $N$ such that
the multivalued function $f/(\log t)^N$ is bounded near $p$. 
\end{definition}

A {\it period of a section} $s$ of $\VV$ with respect to a horizontal (possibly multivalued) section $\phi$ of $\VV^\vee$
is defined as $\lan s,\phi\ran$ (this is an even function on the base).

\begin{lemma}\label{log-growth-lem}
(i) Locally there exists a unique up to an isomorphism canonical extension $\ov{\VV}$ 
of $\VV$. 

\noindent
(ii) If a section $s$ of $\VV$ extends to a regular section of $\ov{\VV}$ then its periods $\lan s,\phi\ran$, with respect
to a horizontal section $\phi$ of $\VV^\vee$, have at most logarithmic growth.
\end{lemma}

\begin{proof}
(i) Locally there exists a splitting $\OO_{\ov{S}}=\OO_{\ov{S}_{\bos}}\ot_{\C}\bigwedge(E)$, such that $B$ is the pull-back of
a normal crossing divisor $B_{\bos}\sub \ov{S}_{\bos}$. By Penkov's theorem, we have
$$\VV\simeq \VV_{\bos}\ot\bigwedge(E),$$
as a bundle with connection.
By \cite[Prop.\ 5.2, Prop.\ 5.4]{Deligne-eq-diff}, there exists a canonical extension $(\ov{\VV}_{\bos},\nabla^{\log}_{\bos})$ of $\VV_{\bos}$.
Hence, we can take 
\begin{equation}\label{adm-ext-splitting-constr}
\ov{\VV}:=\ov{\VV}_{\bos}\ot_{\C}\bigwedge(E)
\end{equation}
as an extension of $\ov{\VV}$ to $\ov{S}$.
We have a splitting
$$\Om^1_{\ov{S}}(\log B)=\bigwedge(E)\ot_{\C}\Om^1_{\ov{S}_{\bos}}(\log B_{\bos})\oplus E\ot \OO_{\ov{S}},$$
and the logarithmic connection on $\ov{\VV}$ is given by
\begin{align*}
\nabla^{\log}:\ov{\VV}=\bigwedge(E)\ot_{\C}\ov{\VV}_{\bos}\rTo{(\id\ot \nabla^{\log}_{\bos},\kappa_E\ot\id)}
                 &\bigwedge(E)\ot_{\C}\Om^1_{\ov{S}_{\bos}}(\log B_{\bos})\ot \ov{\VV}_{\bos}
  \\
  &\oplus E\ot \bigwedge(E)\ot_{\C}\ov{\VV}_{\bos}\\
&\simeq \Om^1_{\ov{S}}(\log B)\ot_{\OO_{\ov{S}}} \ov{\VV}.
\end{align*}
Near the locus where $B$ is smooth, we have $\OO_B=\OO_{B_{\bos}}\ot_{\C}\bigwedge(E)$, and
$\Res_B(\nabla^{\log})$ is given by 
$$\id\ot\Res_{B_{\bos}}(\nabla^{\log}_{\bos}):\bigwedge(E)\ot_{\C}\ov{\VV}_{\bos}|_{B_{\bos}}\to \bigwedge(E)\ot_{\C}\ov{\VV}_{\bos}|_{B_{\bos}},$$
hence, the operator $\Res_B(\nabla^{\log})|_p$, for any point $p\in B$ is equal to $\Res_{B_{\bos}}(\nabla_{\bos}^{\log})|_p$,
so it has eigenvalues in $[0,1)$.

Now let us prove the uniqueness of the canonical extension.
Given a canonical extension $(\ov{\VV},\nabla^{\log})$, it induces a flat connection in odd directions on $\ov{\VV}$, which restricts
to the given connection in odd directions on $\VV$. Hence, we have an isomorphism of bundles with flat connections in odd directions,
$$\ov{\VV}\simeq \bigwedge(E)\ot_{\C}\ov{\VV}_{\bos},$$
where $\ov{\VV}_{\bos}$ is identified with the subsheaf of $\ov{\VV}$ consisting of sections annihilated by $\nabla_{e^*}$ for all
local sections $e^*$ of $E^\vee$. Now for every logarithmic vector field $v$ on $\ov{\SS}_{\bos}$
the operator $\nabla^{\log}_v$ commutes with all $\nabla_{e^*}$, hence, it preserves $\ov{\VV}_{\bos}\sub \ov{\VV}$ and
induces a logarithmic connection $\nabla^{\log}_{\bos}$ on $\ov{\VV}_{\bos}$ such that 
$$\nabla^{\log}=\id\ot\nabla^{\log}_{\bos}+\kappa_E\ot\id$$ 
as in the above construction. Thus, the eigenvalues of $\Res_B(\nabla^{\log})$ are the same as of
$\Res_{B_{\bos}}(\nabla_{\bos}^{\log})$. It remains to use the uniqueness of the canonical extension on $\ov{\SS}_{\bos}$.

\noindent
(ii) By uniqueness, we can assume that $\ov{\VV}$ is of the form \eqref{adm-ext-splitting-constr}, for some splitting of $\ov{S}$
such that $B$ is the pull-back of $B_{\bos}\sub \ov{S}_{\bos}$. Let $e_1,\ldots,e_n$ be a local basis of $\ov{\VV}_{\bos}$.
By \cite[Prop.\ 5.2]{Deligne-eq-diff},
for any multivalued horizontal section $\phi_0$ of $\VV^\vee_{\bos}$ the multivalued functions $\lan \phi_0, e_i\ran$ have at most logarithmic growth (i.e., $O((\log |z|)^k)$). Note that any horizontal section $\phi$ of $\VV^\vee$ is annihilated by derivative
in odd directions, so it belongs to the subsheaf $\VV^\vee_{\bos}\sub \VV^\vee$.
This easily implies the assertion.
\end{proof}

\subsection{Canonical extension of the Gauss-Manin connection for supercurves}

Here we work with 
a family of stable supercurves, $\pi:\ov{X}\to \ov{S}$, where $\ov{S}$ is smooth superscheme, extending a smooth family $X\to S$ over
$S=\ov{S}\setminus S_0$, where $S_0$ is a normal crossing divisor.
We denote by $X_0=\pi^{-1}(S_0)$ the corresponding divisor in $\ov{X}$.
We assume the map from formal neighborhood of each $p\in S_0$ to the deformation functor of each node of $\pi^{-1}(p)$ is \'etale.
We also make an important assumption that the spin-structures corresponding to the fibers $X_s$ have no global sections.

The relative de Rham complex $\Om_{X/S}^\bullet$ provides a resolution for $\C_{X/S}$, so 
$$\VV:=R^1\pi_*(\C_{X/S})\simeq R^1\pi_*(\Om_{X/S}^\bullet).$$
 As in the classical case, the Gauss-Manin connection on $\VV$ over $S$ is obtained by considering the exact sequence of complexes
$$0\to \pi^*\Om^1_S\ot \Om^{\bullet}_{X/S}[-1]\to \Om^{\bullet}_X/(\pi^*\Om^2_{S}\we \Om^{\bullet-2}_X)\to \Om^\bullet_{X/S}\to 0$$
and the connecting homomorphism
$$R^1\pi_*(\Om^\bullet_{X/S})\to \pi^*\Om^1_S\ot R^2\pi_*(\Om^{\bullet}_{X/S}[-1])\simeq \pi^*\Om^1_S\ot R^1\pi_*(\Om^\bullet_{X/S}).$$

We would like to construct explicitly a canonical extension of $\VV$.

\begin{lemma}
(i) Under the above assumptions, the coherent sheaf 
\begin{equation}\label{can-ext-main-eq}
\ov{\VV}:=R^1\pi_*[\OO_{\ov{X}}\rTo{\de} \om_{\ov{X}/\ov{S}}]
\end{equation}
is locally free and its formation commutes with base changes.

\noindent
(ii) If $S$ is even then we have a natural identification $\ov{\VV}\simeq R^1\pi_*[\OO_C\rTo{d}\om_{C/\ov{S}}]$,
where $\pi:C\to \ov{S}$ is the corresponding family of usual stable curves. 
\end{lemma}

\begin{proof}
(i) This follows from a similar statement fo $R^i\pi_*\OO_{\ov{X}}$ and $R^i\pi_*\om_{\ov{X}/\ov{S}}$, which is proved exactly as in the case of
smooth supercurves, see \cite[Prop.\ 3.2]{FKP-per}.

\noindent
(ii) In this case we have 
$$\OO_{\ov{X}}=\OO_C\oplus L, \ \ \om_{\ov{X}/\ov{S}}=\om_{C/\ov{S}}\oplus L,$$ 
where $L$ is the relative spin structure, and the $\de^-$ is an isomorphism, so we get a quasi-isomorphism
of $[\OO_{\ov{X}}\to \om_{\ov{X}/\ov{S}}]$ with $[\OO_C\to \om_{C/\ov{S}}]$.
\end{proof}








\begin{theorem}\label{super-GM-thm} 
The vector bundle $\ov{\VV}$ given by \eqref{can-ext-main-eq} is a canonical extension of $\VV$, i.e., $\ov{\VV}$ admits a connection with poles of order $1$,
$$\nabla:\ov{\VV}\to \pi^*\Omega^1_{\ov{S}}(S_0)\ot \ov{\VV},$$
restricting to the Gauss-Manin connection on $\VV$, such that the residue of $\nabla$ at any branch of $S_0$ is nilpotent.
\end{theorem}

We prove this theorem further below.
Recall (see \cite[Prop.\ 2.1]{P-superell}) that for a smooth supercurve $X/S$ there is a natural exact sequence of complexes
$$0\to K^{\bullet}_{X/S}\to \Om^\bullet_{X/S}\to [\OO_X\rTo{\de} \om_{X/S}]\to 0$$
where the complex $K^\bullet_{X/S}$ is contractible.

\begin{lemma}\label{exact-complex-lem}
Let $j:U\to \ov{X}$ denote the complement to the nodes. Then we have an exact sequence of complexes
$$0\to j_*K^\bullet_{U/\ov{S}}\to j_*\Om^\bullet_{U/\ov{S}}\to [\OO_{\ov{X}}\to \om_{\ov{X}/\ov{S}}]\to 0.$$
In particular, there is a quasi-isomorphism $j_*\Om^\bullet_{U/\ov{S}}\to [\OO_{\ov{X}}\to \om_{\ov{X}/\ov{S}}]$.
\end{lemma}

\begin{proof}
By \cite[Lem.\ 8.1]{FKP-supercurves}, we have an exact sequence
$$0\to j_*\om_{U/\ov{S}}^2\to j_*\Om^1_{U/\ov{S}}\to \om_{\ov{X}/\ov{S}}\to 0.$$
Since $j_*\OO_U\simeq \OO_{\ov{X}}$, this immediately gives the required exact sequence of complexes.
\end{proof}

Let $j:U\to \ov{X}$ denote the complement to the nodes.
We consider the logarithmic forms and relative logarithmic forms on $U$ (resp., $S$) with respect to the normal crossing divisor $X_0\cap U$ (resp., $S_0$).
Namely, as in the classical case we set  
$$\Om^1_{U/\ov{S}}(\log):=\Om^1_U(\log)/\pi^*\Om^1_{\ov{S}}(\log), \ \ \Om^i_{U/\ov{S}}(\log)={\bigwedge}^i(\Om^1_{U/\ov{S}}(\log)).$$
Since the morphism $\pi:U\to \ov{S}$ is smooth, it is easy to see that the natural map
$$\Om^1_{U/\ov{S}}\to \Om^1_{U/\ov{S}}(\log)$$
is an isomorphism, and hence $\Om^\bullet_{U/\ov{S}}\simeq \Om^\bullet_{U/\ov{S}}(\log)$.

\begin{lemma}\label{logarithmic-image-lem}
The morphism
$$j_*\Om^1_U(\log)\to j_*\Om^1_{U/\ov{S}}$$
is surjective.
\end{lemma}

\begin{proof} Let us consider separately the situation near the nodes of two types.

\noindent
{\bf Case of an NS node}. By Lemma 8.2 of \cite{FKP-supercurves}, in the formal neighborhood of an NS node
the sheaf $j_*\Om^1_{U/\ov{S}}$ is generated over $\OO_X$ by sections $d\th_1$, $d\th_2$,
$e=dz_1/z_1=-dz_2/z_2$ and
$$f=\frac{\th_1d\th_1}{z_1}=-\frac{\th_2d\th_2}{z_2}$$
(where $(z_1,z_2,\th_1,\th_2)$ are standard generators of $\OO_X$ in the formal neighborhood of an NS node).
It is clear that the first two generators extend to sections of $j_*\Om^1_U(\log)$.
For $e$ this follows from the identity
$$\frac{dz_1}{z_1}=2\frac{dt}{t}-\frac{dz_2}{z_2}$$
in the module of absolute differentials $\Om^1_X[z_1^{-1}z_2^{-1}]$.
It remains to check the same for the generator $f$. We claim that in fact, the identity
$$\frac{\th_1d\th_1}{z_1}=-\frac{\th_2d\th_2}{z_2}$$
still holds in $\Om^1_X[z_1^{-1}z_2^{-1}]$.
Indeed, let us express $\frac{\th_2d\th_2}{z_2}$ in terms of $z_1,\th_1$ (assuming that both $z_1$ and $z_2$ are invertible).
We have
$$\th_2=\frac{t\th_1}{z_1}, \ \ z_2=-\frac{t^2}{z_1}, \ \ d\th_2=t[\frac{d\th_1}{z_1}-\frac{\th_1 dt}{z_1t}-\frac{\th_1 dz_1}{z_1^2}],$$
$$\frac{\th_2d\th_2}{z_2}=-\th_1[\frac{d\th_1}{z_1}-\frac{\th_1 dt}{z_1t}-\frac{\th_1 dz_1}{z_1^2}]=-\frac{\th_1d\th_1}{z_1}.$$ 

\noindent
{\bf Case of a Ramond node}. In this case the assertion follows similarly from Lemma 8.3 of \cite{FKP-supercurves}: $j_*\Om^1_{U/\ov{S}}$ is generated
by sections $d\th$ and $e=dz_1/z_1=-dz_2/z_2$, both of which extend to sections of $j_*\Om^1_U(\log)$.
\end{proof}




\begin{proof}[Proof of Theorem \ref{super-GM-thm}]
We have an exact sequence
$$0\to \pi^*\Om^1_{\ov{S}}(\log)|_U\ot \Om^{\bullet}_{U/\ov{S}}[-1]\to \Om^{\bullet}_U(\log)/(\pi^*\Om^2_{\ov{S}}(\log)|_U\we \Om^{\bullet-2}_U(\log))\to \Om^\bullet_{U/\ov{S}}\to 0.$$
Thus, applying (the underived) $j_*$ we obtain the left exact sequence of complexes
$$0\to \pi^*\Om^1_{\ov{S}}(\log)\ot j_*\Om^{\bullet}_{U/\ov{S}}[-1]\to j_*(\Om^{\bullet}_U(\log)/\pi^*\Om^2_{\ov{S}}(\log)|_U\we \Om^{\bullet-2}_U(\log))\to j_*\Om^\bullet_{U/\ov{S}}$$
Let us denote by $C^\bullet$ the image of the last arrow, so that we have an exact sequence
\begin{equation}\label{GM-main-ex-seq}
0\to \pi^*\Om^1_{\ov{S}}(\log)\ot j_*\Om^{\bullet}_{U/\ov{S}}[-1]\to j_*(\Om^{\bullet}_U(\log)/\pi^*\Om^2_{\ov{S}}(\log)|_U\we \Om^{\bullet-2}_U(\log))\to C^\bullet\to 0.
\end{equation}
Then the connecting homomorphism gives a morphism
$$R^1\pi_*C^\bullet\to \pi^*\Om^1_{\ov{S}}(\log)\ot R^1\pi_*(j_*\Om^\bullet_{U/\ov{S}}).$$
On the other hand, we have an exact sequence of complexes
$$0\to C^\bullet\to  j_*\Om^\bullet_{U/\ov{S}}\to Q^\bullet\to 0.$$
Clearly $Q^0=0$ and by Lemma \ref{logarithmic-image-lem}, we have $Q^1=0$. 
It follows that $R^{\le 1}\pi_*Q^\bullet=0$.
Hence, the natural map
$$R^1\pi_*C^\bullet\to R^1\pi_*( j_*\Om^\bullet_{U/\ov{S}})$$
is an isomorphism, and we obtain a connection with first order poles
$$R^1\pi_*(j_*\Om^\bullet_{U/\ov{S}})\to \pi^*\Om^1_{\ov{S}}(\log)\ot R^1\pi_*(j_*\Om^\bullet_{U/\ov{S}}).$$

Next, by Lemma \ref{exact-complex-lem}(ii), the natural projection 
$$j_*\Om^\bullet_{U/\ov{S}}\to [\OO_{\ov{X}}\to \om_{\ov{X}/\ov{S}}]$$ 
induces an isomorphism on $R^i\pi_*$, so we get the identification of $\VV$ with $R^1\pi_*[\OO_{\ov{X}}\to \om_{\ov{X}/\ov{S}}]$.

It remains to check that $\Res(\nabla)$ is nilpotent. 
Our proof of the fact is similar to the proof of \cite[Prop.\ (2.20)]{Steenbrink}.
The idea is to use local to global spectral sequences computing $R^n\pi_*(j_*\Om^\bullet_{U/\ov{S}}|_{X_0})$
and $R^n\pi_*(C^\bullet|_{X_0})$ starting from $R^p(\und{H}^q j_*\Om^\bullet_{U/\ov{S}}|_{X_0})$ (resp., 
$R^p(\und{H}^q C^\bullet|_{X_0})$).
Note that since $Q^{\le 1}=0$, the embedding $C^{\bullet}|_{X_0}\to j_*\Om^\bullet_{U/\ov{S}}|_{X_0}$ induces
an isomorphism on $\und{H}^{\le 1}$.
Using these spectral sequences we see that it suffices to show that for $i=0,1$,
the vanishing of the connecting homomorphisms for the short exact sequence \eqref{GM-main-ex-seq}
restricted to $X_0$,
$$\und{H}^iC^\bullet|_{X_0}\to \pi^*\Om^1_{\ov{S}}(\log)\ot \und{H}^i j_*\Om^\bullet_{U/\ov{S}}|_{X_0}.$$

By Lemma, \ref{exact-complex-lem}(ii), we have a quasi-isomorphism 
$j_*\Om^\bullet_{U/\ov{S}}\to [\OO_{\ov{X}}\to \om_{\ov{X}/\ov{S}}]$, which induces isomorphisms
$$\und{H}^q(j_*\Om^\bullet_{U/\ov{S}}|_{X_0})\rTo{\sim} \und{H}^q[\OO_{X_0}\to \om_{X_0/S_0}].$$
The latter cohomology sheaves are supported at the nodes, so it suffices to compute the connecting homomorphisms
at the formal neighborhood of each node. Furthermore, it is enough to make the latter computation for the standard deformation of the nodes.
In particular, we can assume that $S$ is a formal disk, and $S_0=p$ is its origin.

\noindent
{\bf Case of an NS node}. 
We use a local description of
$\om_{X_0/p}$ near the NS node by generating sections $s_1,s_2,s_0$, where $\de(\th_1)=s_1$, $\de_2(\th_2)=s_2$,
subject to the relations described in \cite[Lem.\ 6.4]{FKP-supercurves}. This easily gives
$$\und{H}^0[\OO_{X_0}\to \om_{X_0/p}]\simeq \OO_{p}, \ \ \und{H}^1[\OO_{X_0}\to \om_{X_0/p}]\simeq \OO_q\cdot s_0,$$
where $p\simeq q\sub X_0$ is the node.

The generator $1\in \und{H}^1[\OO_{X_0}\to \om_{X_0/p}]$ lifts to $1\in j_*\OO_U$ which is a cocycle in $j_*\Om^\bullet_U(\log)$, so it maps
to zero under the connecting map.
The generator $s_0$ of $\und{H}^1[\OO_{X_0}\to \om_{X_0/p}]$ lifts to a section $\frac{dz_1}{z_1}$ of
$j_*\Om^1_U(\log)$, which is still a cocycle in $j_*\Om^\bullet_U(\log)$, so the image under the connecting map is also zero.

\noindent
{\bf Case of a Ramond node}.
In this case we have coordinates $(z_1,z_2,\th)$, and a generator $e$ of $\om_{X_0/p}$, so that
$$\und{H}^0[\OO_{X_0}\to \om_{X_0/p}]\simeq \OO_{p}, \ \ \und{H}^1[\OO_{X_0}\to \om_{X_0/p}]\simeq \OO_q\cdot \th\cdot e.$$
As before the generator of $\und{H}^0$ lifts to $1\in j_*\OO_U$, which is a cocycle, while
the generator $\th\cdot e$ of $\und{H}^1$ lifts to a section $\frac{dz_1}{z_1}$ of $j_*\Om^1_U(\log)$, which is still a cocycle
in $j_*\Om^\bullet_U(\log)$.
\end{proof}

\subsection{Superperiods of supercurves}\label{superperiod-deg-sec}

Let $\pi:\ov{X}\to \ov{S}$ be a family of stable supercurves such that corresponding spin-structures have no global sections.
Assume that for a dense open $S\sub \ov{S}$, the induced family $\pi:X\to S$ is smooth, and let $C\to S_{\bos}$ be the corresponding family of usual smooth curves.
We are interested in the behavior at infinity (i.e. near points of $\ov{S}\setminus S$) of the periods associated with the subbundle
$$\pi_*\om_{X/S}\hra \VV:=R^1\pi_*(\C_{X/S})=\OO_S\ot_{\C} R^1\pi_*(\C_{C/S_{\bos}}).$$

We further assume
that $\ov{S}\to \ov{\SS}_g$ is an \'etale neighborhood of a stable supercurve with $k\le g$ nonseparating nodes (and some separating nodes), so we have a normal crossing
divisor $D\sub \ov{S}$, such that $S=\ov{S}\setminus D$. Furthermore, $D$ has $k$ local branches
$D^{ns}_1,\ldots,D^{ns}_k$ corresponding to nonseparating nodes and some branches $(D^s_j)$ corresponding to the separating nodes. 

Let $\ov{\MM}_g$ be the moduli space of stable curves,
$\De\sub \ov{\MM}_g$ the boundary divisor, and let $(\De^{ns}_i)$ be the local branches of $\De$ corresponding to $D^{ns}_i$.

We assume that a basis of multivalued horizontal sections $\a_1,\ldots,\a_g,\b_1,\ldots,\b_g$ of $\VV^{\vee}$ is given
that corresponds to a symplectic basis in homology of the fibers of $X/S$ with the following {\it standard monodromy} near $\De$:
\begin{enumerate}
\item $\a_1,\ldots,\a_g$ are univalued; 
\item the monodromy around $\De^s_j$ preserves all $\b_i$;  and 
\item for $i=1,\ldots,k$, the monodromy around $\De^{ns}_j$
transforms $\b_j$ to $\b_j+\a_j$ and preserves $\b_{i}$ for $i\neq j$.
\end{enumerate} 

Let $a_i$ denote the degree of ramification of the projection $\ov{S}\to \ov{\MM}_g$ at $D^{ns}_i$, so $a_i=1$ for the Ramond node and $a_i=2$ for the NS node
(see \cite[Sec.\ 1]{Farkas}). Then the monodromy around $D^{ns}_j$ maps $\b_j$ to $\b_j+a_j\a_j$.

As usual, we define the global sections $(\om_i)$ of $\pi_*\om_{X/S}$ such that $\lan \om_i, \a_j\ran=\de_{ij}$ 
(we can do this due to the assumption that the corresponding spin structures have no global sections).
We would like to know the growth of the periods  
$$\Om_{ij}:=\lan \om_i,\b_j\ran.$$ 
Note that it is well known that $\Om_{ij}=\Om_{ji}$. By Lemma \ref{log-growth-lem}, if we prove that $\om_i$ extend to regular sections
of the canonical extension $\ov{\VV}$ then it would follow that $\lan \om_i,\b_j\ran$ have at most logarithmic growth. In fact, we will
prove an even more precise statement in Theorem \ref{regular-norm-diff-thm} below.

Note that via the embedding $\pi_*\om_{X/S}\to R^1\pi_*(\C_{X/S})\simeq \VV$ we can view $(\om_i)$ as sections of $\VV$.

\begin{lemma}\label{fun-reg-lem}
Let $\ov{S}\sub \C\times\C^{m|n}$ be a neighborhood of the origin, and let $D\sub \ov{S}$ be the divisor $t=0$ (where $t$ is the coordinate on $\C$), $f$ a holomorphic function
on $\ov{S}\setminus D$. Let $(z_\bullet,\th_\bullet)$ denote coordinates on $\C^{m|n}$. 
For every $a\in \C^m$, let $\ov{S}_a=\ov{S}\cap \{z=a\}$ denote the corresponding slice (of dimension $1|n$) transversal to $D$. Assume that for every $a$,
$f|_{\ov{S}_a\setminus D}$ extends regularly to $\ov{S}_a$. Then $f$ extends regularly to $\ov{S}$.
\end{lemma}

\begin{proof} We can write $f$ in the form $f=\sum_{I=\{i_1<\ldots<i_k\}} f_I(t,z)\th_{i_1}\ldots \th_{i_k}$.
Then $f$ is regular if and only if each $f_I(t,z)$ is regular. Thus, we reduce to the purely even case $n=0$. In this case we can use Cauchy's integral formula to see that $f$
is bounded near every compact piece of $\ov{S}\cap \De$, and hence, $f$ is regular on $\ov{S}$.
\end{proof}

\begin{theorem}\label{regular-norm-diff-thm} Assume the symplectic basis $(\a_\bullet,\b_\bullet)$ has a standard monodromy near $\De$.

\noindent
(i) The sections $(\om_1,\ldots,\om_g)$ extend to global sections of $\ov{\VV}$.

\noindent
(ii) Suppose $k=1$ (i.e., we are near a supercurve with $1$ nonseparating node). Then 
all the entries $\Om_{ij}=\lan \om_i,\b_j\ran$ of the superperiod matrix, except for $\Om_{11}$ are regular,
while $\Om_{11}-\frac{\log(t^{a_1})}{2\pi i}$ is regular, where $t=0$ is a local equation of $\De^{ns}_1$. Furthermore, the matrix $(\ov{\Om}_{ij}-\Om_{ij})_{2\le i,j\le g}$ is invertible.

\noindent
(iii) For the general $k\ge 1$, all the entries except for $\Om_{11},\ldots,\Om_{kk}$ are regular,
while $\Om_{ii}-\frac{\log(t_i^{a_i})}{2\pi i}$ is regular, where $t_i=0$ is a local equation of $\De^{ns}_i$.
Furthermore, the matrix $(\ov{\Om}_{ij}-\Om_{ij})_{k+1\le i,j\le g}$ is invertible.  
\end{theorem}

\begin{proof} 
(i) Since regularity of a section can be checked in codimension $1$, we can assume that $D$ (resp., $\De$) is smooth. Furthermore, by Lemma \ref{fun-reg-lem}, we can assume that 
$\ov{S}_{\bos}$ is $1$-dimensional, and $D_{\bos}=\{s_0\}$ is a point. Thus, by Theorem \ref{super-GM-thm}, we can use the identification of the canonical extension of $\VV$ with $\ov{\VV}$ given by \eqref{can-ext-main-eq}.

Let $\La_{\bos}\sub \VV^\vee_{\bos}$ be the trivial subbundle with connection, with basis $\a_1,\ldots,\a_g$ (it corresponds to
the trivial subrepresentation of the fundamental group). By functoriality of the canonical extension, we get an embedding
$$\ov{\La}_{\bos}\sub \ov{\VV}^\vee_{\bos}$$
of the trivial subbundle, where $\a_1,\ldots,\a_g$ extend to a basis of $\ov{\La}_{\bos}$.
Hence, we get an embedding of the trivial subbundle $\ov{\La}$ with a basis $(\a_i\ot 1)$ in $\ov{\VV}^\vee$.

Note that the morphism in derived category $\om_{\ov{X}/\ov{S}}[-1]\to [\OO_{\ov{X}}\to \om_{\ov{X}/\ov{S}}]$ induces a morphism
$\pi_*\om_{\ov{X}/\ov{S}}\to \ov{\VV}$.
We claim that the composed map of vector bundles
$$\pi_*\om_{\ov{X}/\ov{S}}\to \ov{\VV}\to \ov{\La}^\vee$$
is an isomorphism. It is enough to check the same for the restriction of this map to the reduced space,
$$\pi_*\om_{\ov{C}/\ov{S}_{\bos}}\simeq \pi_*\om_{\ov{X}/\ov{S}}|_{\ov{S}_{\bos}}\to \ov{\La}_{\bos}^\vee,$$
where $\ov{C}\to \ov{S}_{\bos}$ is the underlying family of stable curves.
In fact, we only need to check this for the fibers over the point $p=\De\in\ov{S}_{\bos}$.

Since $\pi_*\om_{\ov{C}/\ov{S}_{\bos}}$ is compatible with base changes, we can replace $\ov{C}/\ov{S}_{\bos}$ by the corresponding miniversal family,
in particular, we can assume that the total space $C$ is smooth and $D=\pi^{-1}(p)$ is a normal crossing divisor.
Note that in this case there is a natural identification of $\OO\to \om_{\ov{C}/\ov{S}_{\bos}}$ with the relative logarithmic de Rham complex used in \cite{Steenbrink} to
define the limiting mixed Hodge structure.
 
In the case of a separating node the statement is clear (as the monodromy is trivial). In the case of a nonseparating node we can assume 
that the logarithm of the monodromy $N$ sends $\b_1$ to $\a_1$ and all other basis vectors to $0$.
Now we can use the well known identification of the fiber of $\ov{\VV}_{\bos}$ at $s_0$ with the limiting mixed Hodge structure $\lim H^1(C_s,\C)$,
so that $H^0(\om_{\ov{C}_{s_0}})$ gets identified with the Hodge subspace $F^1$ (see \cite{Schmid}, \cite{Steenbrink}).
Identifying $\lim H^1(C_s,\C)$ with $H_1(C_s,\C)$, we can think of $\La_{\bos}$ as a subspace of $H^1(C_s,\C)$, so that we have an exact sequence
$$0\to \La_{\bos}\to H^1(C_s,\C)\to \La_{\bos}^\vee\to 0.$$ 
Thus, we need to check that $F^1\cap \La_{\bos}=0$.

Since $N^2=0$, the weight filtration is given by
$$W_0=\im(N)=\lan \a_1\ran, \ \ W_1=\ker(N)=\lan \a_1,\ldots,\a_g,\b_2,\ldots,\b_g\ran=W_0^\perp.$$
Then we have $F^1\cap W_0=0$ and $F^1\cap W_1$ is the Hodge subspace of the pure Hodge structure on $W_1/W_0$ which is identified with $H^1(\wt{C}_0,\C)$,
where $\wt{C}_0$ is the normalization of the special fiber. 
Since $\La_{\bos}$ is contained in $W_1$ and $\La_{\bos}/W_0$ is the subspace generated by the Lagrangian subspace in $H^1(\wt{C}_0,\R)$,
we get
$$F^1\cap \La_{\bos}=(F^1\cap W_1)\cap \La_{\bos}=0,$$
as required.

\noindent
(ii) By part (i) and Lemma \ref{log-growth-lem}(ii), all the entries $\Om_{ij}$ have at most logarithmic growth.
Since the monodromy fixes
$(\b_i)_{i\ge 2}$, all $\Om_{ij}=\lan \om_i,\b_j\ran$ with $j>1$ are univalued. 
Since they have at most logarithmic growth, they must be regular. By the symmetry $\Om_{ij}=\Om_{ji}$, we get regularity of all the entries
except for $\Om_{11}$, 

Next, since the monodromy around $D^{ns}_i$ 
changes $\Om_{11}$ to $\Om_{11}+a_1$, the function $\Om_{11}-\frac{\log(t^{a_1})}{2\pi i}$ is univalued and has at most logarithmic growth,
hence it is regular.

Finally, it is well known that modulo the nilpotents, the limit of the submatrix $(\Om_{ij})_{2\le i,j\le g}$ as the curve tends to the stable curve $C_0$ can be identified with 
the period matrix of the normalization of $C_0$. This implies the last assertion.

\noindent
(iii) Using Hartogs theorem and part (ii), we immediately see that all entries except for $\Om_{11},\ldots,\Om_{kk}$, as well as $\Om_{ii}-\frac{\log(t_i^{a_i})}{2\pi i}$, are regular.
The last statement is proved as in part (ii).
\end{proof}
 
\begin{cor}\label{log-growth-main-cor}
Near the point of the quasidiagonal corresponding to a curve with $k$ nonseparating nodes (and possibly some separating nodes), such that the symplectic basis $(\a_\bullet,\b_\bullet)$ has a standard monodromy
near this point (as described above),
there exist trivializations of $\Ber_1$ and $\wt{\Ber}_1$
such that
$$h^{-1}=g\cdot\log(t_1^{a_1}\wt{t}_1^{b_1})\ldots \log(t_k^{a_k}\wt{t}_k^{b_k})$$
where $t_i=0$ and $\wt{t}_i=0$ are equations of the corresponding branches of the boundary divisor on $\ov{\SS}_g$ and $\ov{\SS}_g^c$, and $g$ is bounded and nonvanishing
near $t_1=\ldots=t_k=0$
(and $b_i$ have the same meaning on the antiholomorphic side as $a_i$ on the holomorphic side).
\end{cor} 

\begin{proof}
This follows from formula \eqref{h-def-eq}, taking into account that 
$$\Om_{ii}\sim \frac{\log(t_i^{a_i})}{2\pi i}$$ 
for $1\le i\le k$, and the matrix $(\ov{\Om}_{ij}-\Om_{ij})_{k+1\le i,j\le g}$ is invertible (see Theorem \ref{regular-norm-diff-thm}(iii)).
Namely, we deduce that
$$\wt{\Om}_{ii}\sim -\frac{\log(\wt{t}_i^{b_i})}{2\pi i}$$
(the minus sign is due to complex conjugation), and hence,
$$\Om_{ii}-\wt{\Om}_{ii}\sim \frac{\log(t_i^{a_i}\wt{t}_i^{b_i})}{2\pi i}.$$
\end{proof}
   




Theorem \ref{regular-norm-diff-thm} (resp., Corollary \ref{log-growth-main-cor}) does not describe the situation near all nodal spin curves such that
the spin structure does not have global sections. 
In genus $2$ there is one additional nodal curve, namely, $C=C_1\cup C_2$, where $C_1=C_2=\P^1$, and the two components are joined nodally at
three points. One can also realize this curve by starting with the smooth Riemann surface $\wt{C}$ of genus $2$ and pinching the simple curves with the classes
$$\a_1, \ \ \a_{12}:=\a_2-\a_1, \ \ \a_2.$$

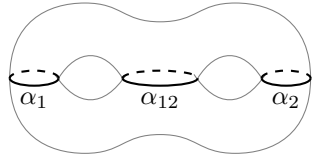
\begin{figure}[h]
\begin{tikzpicture}[baseline=0cm,xscale=.5,yscale=.5]
\draw [gray]
(0,0) .. controls (0,1.5) and (1,2) .. (2,2)
(2,2) .. controls (3,2) and (3,1.5) .. (4,1.5)
(0,0) .. controls (0,-1.5) and (1,-2) .. (2,-2)
(2,-2) .. controls (3,-2) and (3,-1.5) .. (4,-1.5)
(8-0,0) .. controls (8-0,1.5) and (8-1,2) .. (8-2,2)
(8-2,2) .. controls (8-3,2) and (8-3,1.5) .. (8-4,1.5)
(8-0,0) .. controls (8-0,-1.5) and (8-1,-2) .. (8-2,-2)
(8-2,-2) .. controls (8-3,-2) and (8-3,-1.5) .. (8-4,-1.5)
;
\draw[gray] (1.3,0) .. controls (1.8,.8) and (2.5,.8).. (3,0);
\draw[gray] (1.2,0.1) .. controls (1.8,-.8) and (2.5,-.8).. (3.1,.1);
\draw[gray] (8-1.3,0) .. controls (8-1.8,.8) and (8-2.5,.8).. (8-3,0);
\draw[gray] (8-1.2,0.1) .. controls (8-1.8,-.8) and (8-2.5,-.8).. (8-3.1,.1);
\draw [thick](0,0) arc (-180:0:0.65 and 0.2);
\draw [thick,dashed](1.3,0) arc (0:180:0.65 and 0.2) ;
\draw (0.65,-.6) node {$\alpha_1$};
\draw [thick](8-0,0) arc (-180:0:-0.65 and 0.2);
\draw [thick,dashed](8-1.3,0) arc (0:180:-0.65 and 0.2) ;
\draw (8-0.65,-.6) node {$\alpha_2$};
\draw [thick](3,0) arc (-180:0:1 and 0.2);
\draw [thick,dashed](3,0) arc (0:180:-1 and 0.2) ;
\draw (4,-.6) node {$\alpha_{12}$};
\end{tikzpicture}
\end{figure}


Let us denote the corresponding three branches of the non-separating node divisor as $D_1=(t_1=0)$, $D_{12}=(t_{12}=0)$ and $D_2=(t_2=0)$. 
Let $a_1$, $a_{12}$ and $a_2$ be the corresponding ramification indices ($1$ or $2$ depending on whether the node is Ramond or NS).

\begin{lemma}\label{tough-curve-per-lem}
Near the point of the quasidiagonal corresponding to the union of two $\P^1$'s glued at three points, one has for the superperiod matrix $\Om$ (up to a regular summand and an invertible multiple)
\begin{equation}\label{twoP1s-periods}
(2\pi i)\Om\sim \left(\begin{matrix} \log(t_1^{a_1})+\log(t_{12}^{a_{12}}) & -\log(t_{12}^{a_{12}}) \\ -\log(t_{12}^{a_{12}}) & \log(t_2^{a_2})+\log(t_{12}^{a_{12}})\end{matrix}\right),
\end{equation}
$$
h^{-1}=g\cdot[\log(t_1^{a_1}\wt{t}_1^{b_1})\log(t_2^{a_2}\wt{t}_2^{b_2})+\log(t_1^{a_1}\wt{t}_1^{b_1})\log(t_{12}^{a_{12}}\wt{t}_{12}^{b_{12}})
+
\log(t_2^{a_2}\wt{t}_2^{b_2})\log(t_{12}^{a_{12}}\wt{t}_{12}^{b_{12}})],
$$
where $g$ is bounded and nonvanishing.
One has a similar behavior for the usual period matrix near this curve, without the multiplicites $(a_*)$, $(b_*)$ (and with $t_*$
replaced by the equations of the corresponding branches of the nonseparating boundary divisor on $\ov{\MM}_2$).
\end{lemma}

\begin{proof}
The monodromies
around each branch, $M_1$, $M_{12}$ and $M_2$, are given by the symplectic transformations $x\mapsto x+m(\a\cdot x)\a$, with $\a=\a_1$, $\a_{12}$ or $\a_2$, respectively,
and $m$ the corresponding ramification index.
Thus, they act trivially on $\a_i$, and
\begin{align*}
&M_1(\b_1)=\b_1+a_1\a_1, \ \ M_2(\b_2)=\b_2+a_2\a_2, \ \ M_{12}(\b_1)=\b_1+a_{12}(\a_1-\a_2), \\
& M_{12}(\b_2)=\b_2+a_{12}(\a_2-\a_1),
\end{align*}
while $M_i(\b_j)=\b_j$ for $i\neq j$.
It follows that $\Om_{11}$ (resp., $\Om_{22}$) is changed to $\Om_{11}+a_1$ (resp., $\Om_{22}+1$) by $M_1$ and $M_{12}$ (resp., $M_2$ and $M_{12}$) and preserved by 
$M_2$ (resp., $M_1$).
Similarly, $\Om_{12}$ is preserved by $M_1$ and $M_2$ and is changed to $\Om_{12}-1$ by $M_{12}$. This gives \eqref{twoP1s-periods}.
Passing to $\Om-\wt{\Om}$ and computing the determinant gives the result.
\end{proof}

\section{Preliminaries on supercurves of genus $2$}

\subsection{Stable spin curves of genus $2$}

The following general result is a consequence of the local results on spin structures near nodes from \cite{Jarvis}.

\begin{lemma}\label{NS-R-lem}
Let $(C,L)$ be a nodal curve with a spin structure, $q\in C$ a node, $\rho:\wt{C}\to C$ the corresponding partial normalization (an isomorphism away from $q$). 

\noindent
(i) If $q$ is NS for $L$ if and only if
$L\simeq \rho_*\wt{L}$, where $\wt{L}$ is a spin structure on $\wt{C}$. 

\noindent
(ii) If $q$ is Ramond for $L$ and $\rho^{-1}(q)=\{q_1,q_2\}$ then $\rho^*L$ is a spin structure for $\wt{C}$ equipped with Ramond punctures $q_1,q_2$.
\end{lemma}

\begin{proof}
(i) In this case $L\simeq \rho_*\wt{L}$ where $\wt{L}=\rho^*L/T$, where $T\sub \rho^*L$ is the torsion subsheaf.
Note that by the relative duality for $\rho$ we have
$$R\und{\Hom}(\rho_*\wt{L},\om_C)\simeq \rho_*R\und{\Hom}(\wt{L},\om_{\wt{C}}).$$
Thus, a spin structure isomorphism $\wt{L}\simeq R\und{\Hom}(\wt{L},\om_{\wt{C}})$ leads to a spin structure for $L=\rho_*\wt{L}$.

Conversely, a spin structure on $L$ can be viewed as a map $L\ot L\to \om_C$, or equivalently, a map
$$\rho_*(\wt{L}\ot \rho^*L)\simeq \rho_*(\wt{L})\ot L\to \om_C.$$
By adjunction of $(\rho_*,\rho^!)$, this gives a map
$$\wt{L}\ot \rho^*L\to \om_{\wt{C}}.$$
This map necessarily is zero on $\wt{L}\ot T$, so it factors through a map $\wt{L}\ot \wt{L}\to \om_{\wt{C}}$. A local calculation shows that
the corresponding map $\wt{L}\to\und{\Hom}(\wt{L},\om_{\wt{C}})$ is an isomorphism, so this is a spin structure on $\wt{L}$.

One can check that the above two constructions are mutually inverse.

\noindent
(ii) This immediately follows from the isomorphism $\rho^*\om_C\simeq \om_{\wt{C}}(q_1+q_2)$.
\end{proof}

Recall (see \cite[Prop.\ 9.1]{Mum}) that there are the following types of curves in the boundary of $\ov{\MM}_2$:
\begin{enumerate}
\item $C_1\cup C_2$, with $C_1$ and $C_2$ irreducible curves of arithmetic genus $1$, glued nodally;
\item irreducible curve of geometric genus $1$ with one node;
\item irreducible curve of geometric genus $0$ with two nodes;
\item union of two $\P^1$'s glued along three pairs of points.
\end{enumerate}


Recall that a spin structure $L$ on a stable curve $C$ is called {\it even} (resp., {\it odd}) if $h^0(C,L)$ is even (resp., odd).
This condition is stable under deformations, so we have the corresponding components of the moduli stack of supercurves.
We denote by $\ov{\SS}_g$ the moduli space of stable supercurves of genus $g$ such that the underlying spin structures are even.

Let $(C,L)$ be a stable curve with an even spin structure, $q\in C$ are separating node, so that $C=C_1\cup C_2$, with $q=C_1\cap C_2$.
Note that in this case the partial normalization of $q$ is the natural map
$C_1\sqcup C_2\to C$. 
Hence, by Lemma \ref{NS-R-lem}(i), $q$ is NS for $L$ if and only if $L=L_1\oplus L_2$, where $L_1$ (resp., $L_2$) is a spin structure on $C_1$
(resp., $C_2$).

\begin{definition} 
We say that a separating NS node $q$ for an even stable spin curve $(C,L)$ is of type $(+,+)$ (resp., $(-,-)$) if $L=L_1\oplus L_2$ where both $L_1$ and $L_2$ are
even (resp., odd).
\end{definition}

For a smooth curve $C$ of genus $2$ and an even spin structure $L$ one has $\deg(L)=1$, which implies that $H^*(C,L)=0$. 
We need the following generalization of this fact to stable curves.

\begin{lemma}\label{coh-van-spin-str-g2-lem} 
(i) Let $(C,L)$ be an irreducible nodal curve of arithmetic genus $1$ with an even spin structure. Then $H^0(C,L)=0$.

\noindent
(ii) Let $(C,L)$ be an even stable spin curve of genus $2$, which does not have a separating node of type $(-,-)$.
Then $H^0(C,L)=H^1(C,L)=0$.
\end{lemma}

\begin{proof} (i) If $C$ is smooth then this is clear. Suppose $C$ is nodal, and let $\rho:\wt{C}\to C$ be the normalization.

If the node of $C$ is NS then $L=\rho_*\wt{L}$, where $\wt{L}$ is a spin structure on $\wt{C}\simeq \P^1$, so $\wt{L}=\OO(-1)$ and
$H^0(C,L)=H^0(\P^1,\OO(-1))=0$. If the node is Ramond then $\wt{L}=\rho^*L$ satisfies $\wt{L}^2\simeq \om_{\wt{C}}(q_1+q_2)$,
so $\wt{L}\simeq \OO_{\wt{C}}$, and $h^0(C,L)\le h^0(C,\wt{L})=1$. 
But $h^0(C,L)$ is even, so $h^0(C,L)=0$.

\noindent
(ii) Since $H^1(C,L)^*\simeq H^0(C,\und{\Hom}(L,\om_C))\simeq H^0(C,L)$, it is enough to prove the vanishing of $H^0(C,L)$.

\noindent
{\bf Case of an irreducible curve with at least one NS node}.

Let $\rho:\wt{C}\to C$ be the normalization of the NS node $q\in C$ (so it is an isomorphism away from $q$).
Then $\wt{C}$ is an irreducible curve of arithmetic genus $1$, and $L=\rho_*\wt{L}$, where $\wt{L}$ is an even spin structure on  $\wt{C}$.
Hence, $H^0(C,L)=H^0(\wt{C},\wt{L})$ which vanishes by part (i).

\noindent
{\bf Case of an irreducible curve of geometric genus $1$ with a Ramond node}.

Let $\rho:\wt{C}\to C$ be the normalization, where $\wt{C}$ is a smooth curve of genus $1$. Let $q\in C$ be the node
and $\rho^{-1}(q)=\{q_1,q_2\}\sub \wt{C}$.
In this case $L$ is locally free and $L^2\simeq \om_C$, so $(\rho^*L)^2\simeq \OO_{\wt{C}}(q_1+q_2)$.
Thus, $\rho^*L$ is of degree $1$, and therefore
$h^0(C,L)\le h^0(\wt{C},\rho^*L)=1$, which implies that $h^0(L)=0$.

\noindent
{\bf Case of an irreducible curve of geometric genus $0$ with two Ramond nodes}.

Let $\rho:\wt{C}\to C$ be the normalization, where $\wt{C}=\P^1$.
Let $q,q'\in C$ be the nodes, and $\rho^{-1}(q)=\{q_1,q_2\}$, $\rho^{-1}(q')=\{q'_1,q'_2\}$. 
In this case $L$ is locally free and 
$$(\rho^*L)^2\simeq \om_{\P^1}(q_1+q_2+q'_1+q'_2)\simeq \OO(2).$$
Hence, $\rho^*L\simeq \OO(1)$. The space $H^0(C,L)$ embeds into the subspace of the $2$-dimensional space $H^0(\wt{C},\rho^*L)$,
consisting of sections $s$ such that $s|_{q_1}=s|_{q_2}$ and $s|_{q'_1}=s_{q'_2}$. Since the
restriction map $H^0(\OO(1))\to \OO(1)|_{q_1}\oplus \OO(1)|_{q_2}$ is an isomorphism, $H^0(C,L)$ is a proper subspace of $H^0(\wt{C},\OO(1))$,
hence $h^0(C,L)\le 1$, which implies that $h^0(C,L)=0$.

\noindent
{\bf Case of a separating node of type $(+,+)$}.

In this case $C$ is the nodal union of two irreducible curves $C_1$ and $C_2$ of arithmetic genus $1$, and $L=L_1\oplus L_2$ where $L_i$ is an even spin structure on $C_i$.
By part (i), we get $H^*(C_i,L_i)=0$, so $H^*(C,L)=0$.

\noindent
{\bf Case of two $\P^1$'s glued along three pairs of points}.

We have $C=C_1\cup C_2$, where $C_i\simeq \P^1$ and the point $p_i\in C_1$ is glued with $q_i\in C_2$, for $i=1,2,3$.
By parity, at least one of the nodes is NS, so $L$ is the push forward of an even spin structure $L'$ on $C'=C_1\cup C_2$, where $C_1$ and $C_2$ are glued at two pairs of points.
If both nodes of $C'$ are NS then $L'$ is the direct sum of spin structures $L_i$ on $C_i$, which do not have global sections. Otherwise, both nodes are Ramond,
which means that $L'$ is a line bundle, a square root of $\om_{C'}\simeq \OO_{C'}$, so $h^0(L')\le 1$, hence $h^0(L')=0$.
\end{proof}

\subsection{Hyperelliptic point of view on genus $2$ curves}\label{hyperell-basics-section}

Recall that the canonical linear system of a smooth projective genus $2$ curve $C$ gives a double covering
$$f:C\to \P^1$$
ramified at $6$ points. It is also well known that even spin structures on $C$ are in bijection with partitions of the ramification locus
into two subsets of $3$ points, $(p_1,p_2,p_3)$, $(q_1,q_2,q_3)$ (where the order of two subsets is not fixed).
Namely, the corresponding spin-structure is $\OO_C(p_2+p_3-p_1)=\OO_C(q_2+q_3-q_1)$ (this relation and independence on the ordering of points follows
from the isomorphisms $\OO_C(2p_i)\simeq \OO_C(2q_i)\simeq \om_C$, $\OO_C(p_1+p_2+p_3+q_1+q_2+q_3)\simeq \om_C^3$). 
 
If we pick coordinates on $\P^1$ so that the ramification points are $\A^1\sub \P^1$, then $C\setminus f^{-1}(\infty)$ can be identified with the affine curve
$$y^2=\prod_{i=1}^3 (x-u_i)(x-v_i),$$
where $u_i=f(p_i),v_i=f(q_i)\in \A^1$.
The canonical line bundle $\om_C$ is trivialized away from $f^{-1}(\infty)$ by $dx/y$ (note that near the ramification point $x=u_i$, $y=0$, the function $y$ is a local coordinate,
and $x-u_i$ differs from $y^2$ by an invertible function). In fact, $dx/y$ is a regular differential on $C$ with the divisor of zeros $f^{-1}(\infty)$. Since $x$ has a simple pole at $\infty$,
$xdx/y$ is still a regular differential on $C$, so $(dx/y, xdy/y)$ is a basis of $H^0(C,\om_C)$.
Similarly, 
\begin{equation}\label{qu-diff-basis-eq}
(\frac{dx}{y})^2, \ x(\frac{dx}{y})^2, \ x^2(\frac{dx}{y})^2)
\end{equation}
is a basis of $H^0(C,\om^2_C)$.

The above spin-structure $L$ is equipped with a rational section $\bs$ with a simple pole at $p_1$ and simple zeros at $p_2$, $p_3$, such that under the isomorphism
$L^2\rTo{\sim}\om_C$, one has
$$\bs^2=\frac{(x-u_2)(x-u_3)}{x-u_1} \frac{dx}{y}.$$
Note that the right-hand side is a rational differential with the divisor $2p_3+2p_3-2p_1$.
We also have a rational section $\bs'$ of $L$ such that
$$(\bs')^2=\frac{(x-v_2)(x-v_3)}{x-v_1} \frac{dx}{y}, \ \ \ \ \bs'=\frac{(x-u_1)(x-v_2)(x-v_3)}{y}\cdot \bs.$$
We can also construct a basis $(\chi_1,\chi_2)$ of the $2$-dimensional space $H^0(C,L\ot \om_C)$ by setting
\begin{equation}\label{chi1-chi2-eq}
\chi_1=(x-u_1)\cdot \bs\cdot \frac{dx}{y}, \ \ \chi_2=(x-v_1)\cdot \bs'\cdot \frac{dx}{y}=\frac{y}{(x-u_2)(x-u_3)}\cdot \bs\cdot \frac{dx}{y}.
\end{equation}
Note that the fact that $\chi_1$ is a regular section of $L\ot \om_C$ is equivalent to the fact that
$(x-u_1)\cdot \frac{dx}{y}$ is a regular section of $\om_C(p_2+p_3-p_1)$. In fact, this is a regular differential on $C$ with double zero at $p_1$,
which means that divisor of zeros of $\chi_1$ is exactly $p_1+p_2+p_3$. Similarly, the divisor of zeros of $\chi_2$ is $q_1+q_2+q_3$.
  
Letting the ramification points on $\P^1$ vary in the configuration space of $\A^1\sub \P^1$ we get a morphism $\Conf_6(\A^1)\to \MM_2\to \ov{\MM}_2$.
Following \cite{Witten}, we will use the related map $\Conf_6(\P^1)\to \SS_{2,\bos}$, which is invariant under $\SL_2(\C)$
and under $S_2\ltimes(S_3\times S_3)$, to do computations with the super Mumford form. 

One can allow some of the points to merge, and still get the corresponding a stable genus $2$ curve with a generalized spin structure. Below we will describe how
this works in families.
Furthermore, we need some information on the relation between boundary divisors in $\ov{\SS}_{2,\bos}$ and the divisors $(u_i-u_j)$, $(v_i-v_j)$,
$(u_i-v_j)$.

   
Note that we can use the construction of the cyclic double covering associated with any effective divisor $D$ of degree 6 on $\P^1$. The corresponding curve
$C$ will have arithmetic genus $2$, and will have at most nodal singularities, provided multiplicities of points in $D$ are $\le 2$.
However, we need also to construct a generalized spin structure on $C$.
   
First, let us consider the situation when points $v_1$ and $v_2$ merge (and otherwise all points are distinct). More precisely, we consider the situation in an affine neighborhood
of a generic point in the configuration space where $v_1$ merge with $v_2$.
In this case $L=\OO(p_2+p_3-p_1)$ is still a locally free spin-structure on $C$ (since we still have isomorphisms $\OO_C(2p_i)\simeq \om_C$).
We claim that the pull-back of the boundary divisor on the moduli stack of generalized spin-structures is the divisor $(v_1-v_2)^2$ near such a point.
Indeed, 
note that locally $C$ is given by the equation 
$$y^2=(x-v_1)(x-v_2),$$
and $L$ is isomorphic to the ideal of $p_1$, which is equal to $(x-u_1,y)$. 
Now we can rewrite the equation of $C$ as
$$z_1z_2=t^2,$$
where 
$$z_1=x-y-\frac{v_1+v_2}{2}, \ \ z_2=x+y-\frac{v_1+v_2}{2}, \ \ t=\frac{v_2-v_1}{2}.$$
The assertion follows from this.

Next, lt us consider the case of a family where points $u_1$ and $v_1$ merge. We claim that in this case we can construct a family of generalized spin-structures $L$ with
NS type behavior at the node of the special fiber. Namely, the corresponding family $C/S$ still has relative points $p_i$, $q_i$, where $p_1$ and $q_1$ specialize to the node, and
we set
$$L=\II_{p_1}(p_2+p_3),$$
where $\II_{p_1}$ is the ideal sheaf of $p_1$. Here we use the fact that $p_2$ and $p_3$ are still Cartier divisors on $C$. The ideal $\II_{p_1}$ is isomorphic to the standard CM module
near the node (see the computation below), so $\und{\Hom}(L,\om_C)$ is still a CM module.  Let $j:U\to C$ be the open embedding of the complement of the nodes.
Then both $L$ and $\und{\Hom}(L,\om_C)$ are naturally isomorphic to $j_*$ extensions of their restrictions to $U$, which are naturaly isomorphic.
This gives an isomorphism $L\simeq \und{\Hom}(L,\om_C)$, so $L$ is a generalized spin-structure.

We claim that the pull-back of the boundary divisor in this situation is given by $(u_1-v_1)$. It is enough to check that
the pair $(\widehat{\OO}_{C,p_1},\hat{L}_{p_1})$ near a point where $p_1$ coincides with the node, is isomorphic to the pull-back under $u_1-v_1$ of the standard family over $\A^1$
with coordinate $t$,
$$(A:=\C[z_1,z_2]/(z_1z_2-t^2), M:=Ae_1\oplus Ae_2/(z_2e_1-te_2,z_1e_2-te_1).$$
We can use 
$$y^2=(x-u_1)(x-v_1)$$
as a local equation of $C$, so that $L$ is isomorphic to the ideal of $p_1$, which is equal to $(x-u_1,y)$. 
As before, we rewrite the equation of $C$ as
$$z_1z_2=t^2,$$
where 
$$z_1=x-y-\frac{u_1+v_1}{2}, \ \ z_2=x+y-\frac{u_1+v_1}{2}, \ \ t=\frac{v_1-u_1}{2}.$$
In terms of these coordinates we can rewrite our ideal as
$$(x-u_1,y)=(z_1+t,z_2+t).$$
Now it is easy to check that the generators $e_1=z_1+t$ and $e_2=z_2+t$ of this ideal satisfy the relations
$z_2e_1=te_2$, $z_1e_2=te_1$, and that the quotient of the free module by these relations is exactly the ideal $(z_1+t,z_2+t)$.
This proves what we need.
  
\subsection{Torelli map and the canonical projection for genus $2$}

Let $\SS_g$ denote the moduli superstack of supercurves of genus $g$ corresponding to an even spin structure.
Its bosonization is the moduli stack $\SS_{g,\bos}$ of smooth curves of genus $g$ with even spin structures.

There is a well defined Torelli (or superperiod) map (see \cite[Thm.\ 6.4]{BHRP}),
$$\per:\SS_g\setminus \DD\to \AA_g,$$
where $\AA_g$ is the moduli stack of principally polarized abelian varieties of dimension $g$,
and $\DD$ is the theta-null divisor corresponding to spin structures with nonzero global sections.

Note that the restriction of the Torelli to the bosonization is the composition
$$\ov{\per}:\SS_{g,\bos}\setminus D\to \SS_{g,\bos}\to \wt{\AA}_g\to \AA_g,$$
where $\wt{\AA}_g$ is the \'etale covering of $\AA_g$ corresponding to a choice of an even theta-characteristic,
i.e., of an even symmetric theta-divisor. Here the map 
\begin{equation}\label{enh-cl-period-eq}
\SS_{g,\bos}\to \wt{\AA}_g
\end{equation}
associates to a spin curve $(C,L)$ the pair $(J,\Th_L)$, where $\Th_L$ is the theta-divisor associated with the spin structure $L$.

\begin{prop}
The superperiod map factors uniquely through a morphism
$$\wt{\per}:\SS_g\setminus \DD\to \wt{\AA}_g$$
restricting to the above map $\SS_{g,\bos}\to \wt{\AA}_g$.
\end{prop}

\begin{proof} This follows from the invariance of \'etale topology with respect to nilpotent extensions.
Namely, the pull-back of the \'etale covering $\wt{\AA}_g\to \AA_g$ by $\per$ gives an \'etale covering
$$\wt{\SS}_g\to \SS_g\setminus \DD,$$
which has a canonical section over $\SS_{g,\bos}\setminus D$. Since the categories of \'etale covers of $\SS_g\setminus \DD$ and $\SS_{g,\bos}\setminus D$
are equivalent we get a unique section $\SS_g\setminus \DD\to \wt{\SS}_g$. This gives the required map.
\end{proof}

\begin{cor}\label{g2-proj-cor}
There exists a unique projection $\pi^{\can}:\SS_2\to \SS_{2,\bos}$ such that $\wt{\per}:\SS_2\to \wt{\AA}_2$ is the composition of $\pi^{\can}$ with \eqref{enh-cl-period-eq}.
\end{cor}

\begin{proof}
This follows from the fact for genus $2$ we have $\DD=\emptyset$ and \eqref{enh-cl-period-eq} is an open embedding.
\end{proof}

Note that the classical Torelli map for genus $2$ is an open embedding $\MM_2\hra \AA_2$, and the superperiod map for $\SS_2$
 factors as the composition
\begin{equation}\label{g2-superperiod-factor-eq}
\per:\SS_2\rTo{\pi^{\can}} \SS_{2,\bos}\to \MM_2\hra \AA_2.
\end{equation}

We have a natural rank $2$ bundle ${\bf H}^1:=R^1\pi_*\OO_{X}$ on $\SS_2$, where $\pi:X\to \SS_2$ is the universal supercurve, and a similar bundle
${\bf H}^1_{\bos}$ on $\SS_{2,\bos}$ (in fact, ${\bf H}^1_{\bos}\simeq {\bf H}^1|_{\SS_{2,\bos}}$). 
Recall that we have a line bundle on $\SS_2$,
$$\Ber_1:=\Ber(R\pi_*\OO_{X})\simeq \Ber(R\pi_*\om_{X/\SS_2})$$
and a similar line bundle $\Ber_{1,\bos}$ on $\SS_{2,\bos}$.
We have a natural isomorphism $\Ber_1\simeq \det({\bf H}^1)^{-1}$.

\begin{prop}\label{Ber1-can-proj-prop} 
One has a natural isomorphism of bundles on $\SS_2$,
$${\bf H}^1\simeq (\pi^{\can})^*{\bf H}^1_{\bos}.$$
Hence, passing to the determinant line bundles, we get an isomorphism
$$\Ber_1\simeq (\pi^{\can})^*\Ber_{1,\bos}.$$
\end{prop}

\begin{proof} 
The required isomorphism is obtained by combining the factorization
\eqref{g2-superperiod-factor-eq} of the superperiod map $\per$ with
the natural isomorphisms
$${\bf H}^1\simeq \per^*{\bf H}^1_{\AA_2}, \ \ {\bf H}^1_{\bos}\simeq \ov{\per}^*{\bf H}^1_{\AA_2},$$
where ${\bf H}^1_{\AA_2}:=R^1p_*\OO_A$ for the universal abelian scheme $p:A\to \AA_2$.
\end{proof}

\subsection{Moduli of stable supercurves of genus $1$ with $1$ NS puncture (even component)}\label{Section-6.1}

Let $\ov{\SS}_{1,1}$ be the moduli space of supercurves of genus $1$ with one NS puncture (with an even underlying spin structure).
Note that $\ov{\SS}_{1,1,\bos}$ classifies stable curves of genus $1$ with one marked point and an even spin structure (it is a triple \'etale covering of $\MM_{1,1}$). 
Let $(C,p,L)$ denote the corresponding
universal data.
We denote by $\ov{\SS}'_{1,1,\bos}$ (resp., $\SS'_{1,1,\bos}$) the total space of the $\G_m$-torsor over $\ov{\SS}_{1,1,\bos}$ (resp., $\SS_{1,1,\bos}$) corresponding to $L|_p$.
Let $\ov{\SS}_{1,1}^{(\infty)}$ denote the moduli space of even stable supercurves of genus $1$ with an NS puncture and a choice of formal superconformal parameters $(x,\th)$.

  

\begin{lemma}\label{phi-map-lem}
One has a natural morphism
\begin{equation}\label{SS-11-morphism}
\ov{\SS}'_{1,1,\bos}\times \A^{0|1}\to \ov{\SS}_{1,1}
\end{equation}
which can be identified with the $\G_m$-torsor over $\ov{\SS}_{1,1}$
associated with the line bundle
$\Pi\om_{X/\ov{\SS}_{1,1}}|_P$, where $(X,P)$ is the universal supercurve of genus $1$ with
one marked point, and $\G_m$ acts diagonally on $\ov{\SS}'_{1,1,\bos}\times \A^{0|1}$.
Furthermore, this morphism lifts to a $\G_m$-equivariant morphism
$$\phi:\ov{\SS}'_{1,1,\bos}\times \A^{0|1}\to \ov{\SS}^{(\infty)}_{1,1},$$
where $\la\in \G_m$ acts by the natural rescaling on both factors $\ov{\SS}'_{1,1,\bos}$ and $\A^{0|1}$, while its action on the target is by rescaling the superconformal parameters 
by $(x,\th)\mapsto (\la^2x,\la \th)$.
The action of $-1\in \G_m$ on $\ov{\SS}'_{1,1,\bos}$ is trivial.
\end{lemma}

\begin{proof}
We have a natural family of supercurves $X$ over $\ov{\SS}'_{1,1,\bos}$ with $\OO_X=\OO_C\oplus L$, equipped with the NS puncture $p$ and the formal superconformal coordinates
$(x,\th)$, such that $dx$ extends to a global section of $\om_{C/\ov{\SS}'_{1,1,\bos}}$ and $\th$ induces a given trivialization of $L|_p$ (here we use the triviality of the dualizing sheaf on irreducible nodal curves of genus $1$).
Now, the map $\phi$ corresponds to the pull-back of this supercurve $X$ to $\ov{\SS}'_{1,1,\bos}\times \A^{0|1}$, but we change the 
NS puncture $p$ and the superconformal coordinates at $p$ to
$$\tau_\eta(x,\th)=(x+\eta\th,\th+\eta),$$ 
where $\eta$ is a coordinate on the factor $\A^{0|1}$.
It is easy to see that the map $\phi$ is $\G_m$-equivariant.

Let $T\to\ov{\SS}_{1,1}$ denote the $\G_m$-torsor associated with the line bundle $\Pi\om_{X/\ov{\SS}_{1,1}}|_P$. It is easy to see that its bosonization is identified with
$\ov{\SS}'_{1,1,\bos}$. On the other hand, in the above construction of the family $(X,S)$, we have a natural trivialization of $\om_{X/S}|_P$, hence, \eqref{SS-11-morphism} lifts
to a map
$$\ov{\SS}'_{1,1,\bos}\times \A^{0|1}\to T.$$
Furthermore, this map of smooth stacks of dimension $2|1$ induces an isomorphism of bosonizations and on tangent spaces. Hence, it is an isomorphism.

The last assertion follows from the existence of the canonical automorphism $-1$ of any (generalized) spin structure $L$. 
\end{proof}

\begin{remark}
Following \cite{CV}, for any superscheme $X$ let us denote by $\Ga=\Ga_X$ the canonical involution of $X$ corresponding to the $\Z_2$-grading of $\OO_X$:
it acts trivially on the underlying topological space, by $1$ on even functions and by $-1$ on odd functions. This involution commutes with all morphisms and makes sense for superstacks.
For the universal supercurve $X\to \ov{\SS}^{(\infty)}_{1,1}$ we have a commutative diagram
\begin{diagram}
X&\rTo{\Ga_{X}}& X\\
\dTo{}&&\dTo{}\\
\ov{\SS}^{(\infty)}_{1,1}&\rTo{\Ga}&\ov{\SS}^{(\infty)}_{1,1}
\end{diagram}
Furthermore, $\Ga_{X}$ changes the universal  formal superconformal parameters $(x,\th)$ to $(x,-\th)$. This implies
that the action of $-1\in \G_m$ on $\ov{\SS}^{(\infty)}_{1,1}$ coincides with the canonical involution $\Ga$.
Thus, compatibility of the morphism
$\phi$ with the action of $-1\in \G_m$ is equivalent to its compatibility with the canonical involutions $\Ga$.
\end{remark}

Since $\ov{\SS}_{1,1}$ has dimension $1|1$, we have a canonical projection
$$\rho:\ov{\SS}_{1,1}\to \ov{\SS}_{1,1,\bos}.$$
By the above Lemma, we have an identification  of $\ov{\SS}_{1,1} $ with the quotient of $\ov{\SS}'_{1,1,\bos}\times \A^{0|1}$ by the diagonal action of $\G_m$,
hence, $\ov{\SS}_{1,1}$ is isomorphic to the total space of the odd line bundle $\Pi L^{-1}|_p$ over $\SS_{1,1,\bos}$.

The universal curve $X\to \ov{\SS}_{1,1}$ has $\OO_X=\OO_C\oplus L$, where $(C,L)$ is the pull-back under $\rho$ of the universal spin curve over $\ov{\SS}_{1,1,\bos}$, and the universal NS puncture $P\sub X$ is defined
by the homomorphism
$$\ev_P:\OO_X=\OO_C\oplus L\rTo{(\ev_p,s\cdot\ev_p)}\OO_{\ov{\SS}_{1,1}},$$
where $s$ is the tautological odd section of $L^{-1}|_p$ over $\ov{\SS}_{1,1}$ and $\ev_p:\OO_C\to \OO_p,\ L\to L|_p$ is evaluation at $p$.
Furthermore, it is easy to see that we have an isomorphism of odd line bundles
$$\LL:=\Pi \rho^*L|_p\simeq \om_{X/\ov{\SS}_{1,1}}|_P.$$

\subsubsection{Standard family of the upper half-plane}\label{even-g1-family-sec}

We will also use the standard family over the upper half-plane $H$. Namely, for $\tau\in H$, we consider the elliptic curve
$C=C_\tau=\C/(\Z+\Z\tau)$ with the marked point $p$ corresponding to $z=0$ (where $z$ is the coordinate on $\C$) and the spin-structure $L=\OO(u-p)$, where
$u\in C$ is the point corresponding to $z=1/2$. More precisely, we use the isomorphism 
$$L^2\rTo{\sim}\om_C$$
given by the rational differential $(\wp(z)-\wp(u))\cdot dz$. We also use $z\mod (z^2)$ as the trivialization of $L|_p\simeq \OO(-p)|_p$.
It is well known that this gives an \'etale surjective map
$$H\to \SS'_{1,1,\bos}.$$

There is an equivalent way to describe this family and the resulting map to $\ov{\SS}^{(\infty)}_{1,1}$ by presenting supercurves of genus $1$ as quotients of $\C^{1|1}$
(see e.g., \cite{Levin1}). Namely, let $X_\tau$ denote the quotient of $\C^{1|1}$ by the $\Z^2$-action generated by
$$(z,\nu)\mapsto (z+1,\nu), \ \ (z,\nu)\mapsto (z+\tau,-\nu).$$
Then $X_\tau$ inherits the standard superconformal structure from $\C^{1|1}$, so that $(z,\nu)$ define local superconformal coordinates near $z=0$.
The underlying spin-structure $M$ corresponds to the order $2$ line bundle over $C_\tau$ obtained by descending the trivial bundle on $\C$ with respect to the action
$$f(z)\mapsto f(z+1),  \ \ f(z)\mapsto -f(z+\tau).$$
The isomorphism $L\to M$ is given by a suitable normalization of $\th_{10}(z,\tau)/\th_{11}(z,\tau)$, where $\th_{ij}$ are standard theta functions with characteristics.

\section{Genus $2$ separating node boundary: the gluing coordinates and the superperiod map}

\subsection{Gluing and periods in the classical (even) case}\label{even-gluing-sec}

Let us recall the computation of the period map near the separating node divisor in $\ov{\MM}_2$ from \cite[Sec.\ 5]{P-bu}.
For an (analytic) family of smooth curves $\pi:C\to S$ the period map is a natural map $\pi_*\om_{C/S}\to \HH^1:=R^1\pi_*\pi^{-1}\OO_S$.
For a family of stable curves $\pi:C\to S$, the canonical extension of the local system $\HH^1$ is given by 
$\ov{\HH}^1:=R^1\pi_*[\OO_C\to\om_{C/S}]$, where $\om_{C/S}$ is the relative dualizing sheaf, and the differential is the composition of the natural maps
$d:\OO_C\to \Om_{C/S}$ and $\Om_{C/S}\to \om_{C/S}$. The period map extends to a map
$$\Pi:\pi_*(\om_{C/S})\to R^1\pi_*[\OO_C\to\om_{C/S}].$$ 

In \cite[Sec.\ 5]{P-bu}, it is shown how to compute the map $\Pi$ in a formal neighborhood of a stable curve of the form $C=C_1\cup C_2$, where
$(C_1,p_1)$ and $(C_2,p_2)$ are smooth genus $1$ curves with marked points, which are glued into a nodal curve. 
More precisely, we consider two standard families of elliptic curves with one marked point
$C_i=\C/(\Z+\Z\tau_i)$ over the base $B$ (the product of two copies of the upper-half space with coordinates $\tau_1$ and $\tau_2$), 
with $p_i$ given by $0\in \C$. We denote by $x_i$ the coordinate on $\C$ which induces local coordinates on $C_i$.
We have an extra parameter scheme $S_q^{(N)}=\Spec(R_N)$, where $R_N=\C[q]/(q^{N+1}))$, and we define the stable curve $C$ over $S=B\times S_q^{(N)}$ which
is glued from $U_1$ and $U_2$
(where $U_1$ is open and $U_2$ is a formal neighborhood of the node), given by
$$U_1=S_q^{(N)}\times \bigl((C_1\setminus\{p_1\})\sqcup (C_2\setminus\{p_2\})\bigr),$$ 
$$U_2=B\times \Spf(R_N[\![x_1,x_2]\!]/(x_1x_2-q)),$$
along 
$$U_{12}=B\times \Spf(R_N(\!(x_1)\!)\oplus R_N(\!(x_2)\!))$$
(see \cite[Sec.\ 3]{P-bu} for details).

For this family, the bundle $\ov{\HH}^1$ can be identified with $\om(U_1)/d\OO(U_1)$ (where $\om$ is the relative dualizing sheaf for our family) (see \cite[Lem.\ 5.1]{P-bu}).
Let us write functions (and sections of other sheaves) on $U_1$ as pairs $(f_1,f_2)$ where $f_i$ is a function on $C_i\setminus\{q_i\}$.
The basis of $\ov{\HH}^1$ is given by the classes of the following elements in $\om(U_1)$:
\begin{equation}\label{H1-basis-eq}
e_1:=(dx_1,0),\ e_2:=(0,dx_2),\ f_1:=(\wp(x_1,\tau_1)dx_1,0),\ f_2:=(0,\wp(x_2,\tau_2)dx_2),
\end{equation}
where $\wp(x_i,\tau_i)$ is the Weierstrass $\wp$-function associated with the lattice $\Z+\Z\tau_i$.
Furthermore, it is easy to see that the Gauss-Manin connection on $\ov{\HH}^1$ is induced by the Gauss-Manin connection on each $H^1(C_i,\C)$
(see Lemma \ref{cech-diff-lem} below for a similar computation for supercurves). 

Let $\a_i$, $\b_i$ be the cycles on $C_i$ given by the loops $[0,1]$ and $[0,\tau_i]$, respectively.
The cohomology basis $(e_i,f_i)$ in $H^1(C_i,\C)$ has the following periods:
$$\int_{\a_i}e_i=1, \ \ \int_{\a_i}f_i=A_i, \ \ \int_{\b_i}e_i=\tau_i, \ \ \int_{\b_i}f_i=2\pi i+\tau_iA_i,$$
where
\begin{equation}\label{Ai-eq}
A_i:=\frac{(2\pi i)^2}{12}E_2(\tau_i),
\end{equation}
(see Sec.\ \ref{ell-sec} and \cite[Sec.\ 1.2]{Katz}).


An element of $\pi_*\om_{C/S}$ is determined by a quadruple 
$$(\om_1(x_1),\om_2(x_2),\phi_1(x_1),\phi_2(x_2))$$ 
with
$\om_i$ being a global form on $C_i\setminus \{p_i\}$ and $\phi_i(x_i)\in \OO_S[\![x_i]\!]$, such that
$$\om_1(x_1)=\phi_1(x_1)dx_1-q \phi_2(\frac{q}{x_1})\frac{dx_1}{x_1^2}, \ \ \om_2(x_2)=\phi_2(x_2)dx_2-q \phi_1(\frac{q}{x_2})\frac{dx_2}{x_2^2}.$$

Let us consider unique global functions $(f_n(x_i))$ on $C_i\setminus\{p_i\}$, for $n\ge 2$,
such that $f_n(x_i)=\frac{1}{x_i^n}+O(z)$. More precisely, for $k\ge 1$,
\begin{equation}\label{f-functions-eq}
f_{2k+1}(x_i)=-\frac{1}{(2k)!}\wp^{(2k-1)}(x_i,\tau_i), \ \ f_{2k}(x_i)=\frac{1}{(2k-1)!}\wp^{(2k-2}(x_i,\tau_i)+\phi_i[2k],
\end{equation}
for some constants $\phi_i[2k]$ depending holomorphically on $\exp(2\pi i\tau_i)$. 

Let $\La_i\sub \OO(C_i\setminus\{p_i\})$ denote the linear span of $(f_n)_{n\ge 2}$. Then as shown in \cite[Prop.\ 5.5]{P-bu},
there exists a unique $\OO_S$-basis of $\pi_*\om_{C/S}$, 
$$\om^{(i)}=(\om^{(i)}_1,\om^{(i)}_2,\phi_1^{(i)}(x_1)dx_1+\phi_2^{(i)}(x_2)dx_2), \ i=1,2,$$
with 
$$\om^{(1)}_1\equiv dx_1 \mod q\La_1 dx_1, \ \ \om^{(1)}_2\equiv 0 \mod q\La_2 dx_2,$$
$$\om^{(2)}_2\equiv dx_2 \mod q\La_2, \ \ \om^{(2)}_1\equiv 0 \mod q\La_1 dx_1.$$
Furthermore, there exist even elliptic functions $\phi(x_i,\tau_i)(q)$ and $\psi(x_i,\tau_i)(q)$, in $x_i$ with respect to $\Z+\Z\tau_i$, with coefficients in $\C[\![q]\!]$,
with poles only at lattice points,
such that
\begin{equation}\label{om1-om2-main-eq}
\begin{array}{l}
\om^{(1)}=((1+q^4\phi(x_1,\tau_1)(q))dx_1,(-q\wp(x_2,\tau_2)+q^4\psi(x_2,\tau_2))dx_2),\\
\om^{(2)}=((-q\wp(x_1,\tau_1)+q^4\psi(x_1,\tau_1))dx_1,(1+q^4\phi(x_2,\tau_2)(q))dx_2).
\end{array}
\end{equation}

The period map $\Pi:\pi_*\om_{C/S}\to \ov{\HH}^1$ can be calculated by normalizing this basis with respect to the $\a$-periods, and computing the corresponding $\b$-periods.
This gives
$$\Om\equiv \left(\begin{matrix} \tau_1-2\pi i A_2q^2 & -2\pi i q(1-q^2A_1A_2) \\
-2\pi i q(1-q^2A_1A_2) &  \tau_2-2\pi i A_1q^2\end{matrix}\right) \mod (q^4),$$
where the coefficients with the higher degrees of $q$ are holomorphic in $\exp(2\pi i\tau_1)$ and $\exp(2\pi i\tau_2)$.



\subsection{Separating node gluing for supercurves}\label{super-gluing-sec}

We can mimic the gluing construction of Sec.\ \ref{even-gluing-sec} in the super case. 
We start with two families
$(X_1,q_1)$, $(X_2,q_2)$ of smooth supercurves  with NS punctures over a base $B$, together with formal superconformal parameters
$(x_i,\th_i)$ at $q_i$, for $i=1,2$. Recall that this means that near $q_i$ the superconformal structure $\de:\OO_{X_i}\to \om_{X_i/B}$ is given by 
$$\de(f)=(\pa_{\th_i}+\th_i\pa_{x_i})(f)[dx_i|d\th_i].$$

Below we construct a stable supercurve $X$ over the base $S=B\times S^{(N)}_t$, where 
$S^{(N)}_t:=\Spec R_N$, $R_N:=\C[t]/(t^{N+1})$
(note that we use the coordinate $t$, whereas in bosonic case the name of the coordinate is $q$; they will be related by $q=-t^2$).
The family $X/S$ will be a deformation of the stable supercurve obtained by gluing $X_1$ and $X_2$ along $q_1$ and $q_2$ (see \cite[Sec.\ 7.5]{FKP-supercurves}).

Namely, we define $X$ as glued from $U_1$ and $U_2$
(where $U_1$ is open and $U_2$ is a formal neighborhood of the node), given by
$$U_1=S^{(N)}_t\times \bigl((X_1\setminus\{q_1\})\sqcup (X_2\setminus\{q_2\})\bigr),\quad
U_2=B\times\Spf(A_N),$$
where
$$A_N=R_N[\![x_1,x_2,\th_1,\th_2]\!]/(x_1x_2+t^2,x_1\th_2-t\th_1,x_2\th_1+t\th_2,\th_1\th_2)),$$
along
$$U_{12}=B\times \Spf(R_N(\!(x_1)\!)[\th_1]\oplus R_N(\!(x_2)\!)[\th_2]).$$
Here we use the Laurent expansions of functions on $X_i\setminus\{q_i\}$, as well as the embedding
\begin{align*}
&\iota: A_N\to R_N(\!(x_1)\!)[\th_1]\oplus R_N(\!(x_2)\!)[\th_2]: \\ 
&x_1\mapsto (x_1, -\frac{t^2}{x_2}), \ x_2\mapsto (-\frac{t^2}{x_2},x_2), \ \th_1\mapsto (\th_1,-\frac{t\th_2}{x_2}), \ \th_2\mapsto (\frac{t\th_1}{x_1},\th_2).
\end{align*} 
More precisely, similarly to \cite[Sec.\ 3.1]{P-bu}, we start with an affine neighborhoods $V_i\sub X_i$ of $q_i$, for $i=1,2$, and first define an affine supercurve $\Spec(A)$,  where
$A$ is defined from the cartesian diagram
\begin{diagram}
A&\rTo{}& R_N\ot[\OO(V_1\setminus\{q_1\}\oplus \OO(V_2\setminus\{q_2\})]\\
\dTo{}&&\dTo{\kappa}\\
A_N&\rTo{\iota}& R_N(\!(x_1)\!)[\th_1]\oplus R_N(\!(x_2)\!)[\th_2]
\end{diagram}
where $\kappa$ is given by the Laurent expansions. As in the even case, we check that for small enough $V_i$, the morphism 
$$U'_1:=S^{(N)}_t\times \bigl(V_1\setminus\{q_1\}\sqcup V_2\setminus\{q_2\}\bigr)\to \Spec(A)$$
is an open embedding, and then define $X$ by gluing $U_1$ and $\Spec(A)$ along the common open $U'_1$.

We also have the superconformal structure $\de:\OO_X\to \om_{X/S}$, extending the given ones on $X_i\setminus\{q_i\}$, and given by the standard formulas on $U_2$
(see \cite[Sec.\ 3.2.2]{FKP-supercurves}). The compatibility over $U_{12}$ is guaranteed by the condition that the parameters $(x_i,\th_i)$ are superconformal.

Let us denote by $\CC[\OO_X\rTo{\de}\om_{X/S}]$ the \v Cech complex with respect to the flat covering $(U_1,U_2)$. 

\begin{lemma}\label{cech-diff-lem}
(i) The natural projection induces a quasi-isomorphism
$$\CC[\OO_X\to \om_{X/S}]\to [\OO(U_1)/\C e_1\stackrel{\de}\to \om_{X/S}(U_1)\to \om_{X/S}(U_{12})/(\de\OO(U_{12})+\om_{X/S}(U_2))],$$
where $e_1=(1,0)\in \OO(U_1)$.
Furthermore,
$$[\OO(U_1)/\C e_1\to \om_{X/S}(U_1)]$$ is a subcomplex in the above complex, and the embedding induces an isomorphism on $H^1$. Thus, we get an identification 
$$\ov{\HH}^1:=R^1\pi_*[\OO_X\to \om_{X/S}]\simeq \om_{X/S}(U_1)/\de\OO(U_1)\simeq \HH^1_{X_1/S}\oplus \HH^2_{X_2/S},$$
where $\HH^1_{X_i/S}=\om_{X_i/S}(X_i-q_i)/\de \OO(X_i-q_i)$.

\noindent
(ii) The Gauss-Manin connection on $\ov{\HH}^1$, defined in Theorem \ref{super-GM-thm}, is regular and is induced by the Gauss-Manin connections on $\HH^1_{X_i/S}$. 
\end{lemma}

\begin{proof}
(i) The proof is analogous to the even case, see \cite[Lem.\ 5.1]{P-bu}.
It is based on the following two assertions: (1) the map 
$$\de:\OO(U_{12})/(\C e_1+\OO(U_2)\to \om_{X/S}(U_{12})/\om_{X/S}(U_2)$$
is injective; (2) the map
$$\om_{X/S}(U_1)\to \om_{X/S}(U_{12})/\de\OO(U_{12})$$
is zero.

We have natural topological bases for the $\OO_S$-modules $\C e_1+\OO(U_2)$, $\OO(U_{12})$, $\om_{X/S}(U_2)$ and $\om_{X/S}(U_{12})$
(see \cite[Sec.\ 6.2]{FKP-supercurves}):
\begin{align*}
& \C e_1+\OO(U_2): (x_i^n, \th_i x_i^n)_{i=1,2,n\ge 0}; \ \ \OO(U_{12}): (x_i^n, \th_i x_i^n)_{i=1,2,n\in\Z};\\
& \om_{X/S}(U_2): s_0,(s_ix_i^n,s_i\th_ix_i^n)_{i=1,2,n\ge 0}; \ \ \om_{X/S}(U_{12}): (s_ix_i^n,s_i\th_ix_i^n)_{i=1,2,n\in\Z}.
\end{align*}
Recall that under the embedding $\om_{X/S}(U_2)\to \om_{X/S}(U_{12})$ the element $s_0$ maps to $(-s_1\th_1/x_1,s_2\th_2/x_2)$.
We get that 
$$(x_i^{-n},\th_ix_i^{-n})_{i=1,2;n>0}$$
is a basis of $\OO(U_{12})/\OO(U_2)$, while
$$s_1x_1^{-n},\ s_1\th_1x_1^{-n}, \ s_2x_2^{-n}, s_2\th_1x_1^{-n-1}, \ n>0,$$
is a basis of $\om_{X/S}(U_{12})/\om_{X/S}(U_2)$.
Now (1) follows from the formula for $\de$:
$$\de(x_i^n)=n s_i\th_i x_i^{n-1}, \ \ \de(\th_i x_i^n)=s_i x_i^n.$$
This also implies that $\om_{X/S}(U_{12})/\de\OO(U_{12})$ has an $\OO_S$-basis $(s_1\th_1x_1^{-1},s_2\th_2x_2^{-1})$,
so (2) follows from the vanishing of the residues.

\noindent
(ii) According the construction of Theorem \ref{super-GM-thm}, we have to use the
connecting homomorphism associated with the exact sequence of \v Cech complexes
$$0\to \CC(j_*\Om^\bullet_{U/S}\ot \pi^*\Om^1_S(\log)[-1])\to \CC(j_*(\Om^\bullet_U(\log)/\pi^*\Om^2_{\ov{S}}(\log)|_U\we \Om^{\bullet-2}_U(\log)))$$
$$\to \CC(C^\bullet)\to 0,$$
where $C^\bullet=\im(j_*(\Om^\bullet_U(\log)/\pi^*\Om^2_{\ov{S}}(\log)|_U\we \Om^{\bullet-2}_U(\log))\to j_*\Om^\bullet_{U/S})$, 
and we use the fact that the map $\CC(C^\bullet)\to \CC(\OO_X\to \om_{X/S})$
induces an isomorphism on $H^1$. We use the same covering $(U_1,U_2)$ of $X$ as before.
Note that $C^\bullet|_{U_1}=\Om^\bullet_{U_1/S}$ and $C^{\le 1}=j_*\Om_{U/S}^{\le 1}$.

Starting from a class in $\ov{\HH}^1$, represented by a \v Cech cocycle 
$$(f,\Phi,\a)\in \OO(U_{12})\oplus \om_{X/S}(U_1)\oplus \om_{X/S}(U_2)=\CC^1[\OO_X\to \om_{X/S}],$$
where $\de(f)=\a|_{U_{12}}-\Phi_{U_{12}}$, we lift it to a \v Cech cocycle
$$(f,\om_1,\om_2)\in \OO(U_{12})\oplus \Om^1_{U_1/S}(U_1)\oplus j_*\Om^1_{U/S}(U_2)=Z^1(C^\bullet)=Z^1(j_*\Om^\bullet_{U/S}),$$
where $Z^1=\ker(\partial:\CC^1\to \CC^2)$. Note that the cocycle condition  
implies that 
\begin{equation}\label{om1-cocycle-eq}
\de(\om_1)=\Phi, \ \ d_{U_1/S}(\om_1)=0,
\end{equation} 
and by \cite[Prop.\ 3.3]{P-superell}, $\om_1$ with such properties is unique.

Next, by the definition of the connecting homomorphism, we have to lift the
cocycle $(f,\om_1,\om_2)$ to a \v Cech cochain
$$(f,\wt{\om}_1,\wt{\om}_2)\in \OO(U_{12})\oplus \Om^1(\log)(U_1)\oplus j_*\Om^1_U(\log)(U_2)=\CC^1(j_*\Om^1_U(\log)),$$
and consider its differential 
$(-df+\wt{\om}_2-\wt{\om}_1,d\wt{\om}_1,d\wt{\om}_2)\in \Om^1(\log)(U_{12})\oplus \Om^2(\log)(U_1)\oplus j_*\Om^2_U(\log)(U_2)=\CC^2(j_*\Om^\bullet_U(\log)),$
which will belong to the subspace $\CC^1(j_*\Om_{U/S}\ot \pi^*\Om^1_S(\log))$.
Thus, $d\wt{\om}_1 \mod\Om^2_S$ is in the image of $\Om^1_{U_1/S}(U_1)\ot \pi^*\Om^1_S(\log)$, and
the class $\nabla([\Phi])$ in $\ov{\HH}^1\ot \pi^*\Om^1_S(\log)$ is represented by $(\de\ot\id)(d\wt{\om}_1)$.

Hence, like in the case of a smooth supercurve, the recipe for calculating $\nabla([\Phi])$ is to find the unique $\om_1\in\Om^1_{U_1/S}(U_1)$ satisfying \eqref{om1-cocycle-eq},
lift it to $\wt{\om}_1\in \Om^1(\log)(U_1)$ and then take $(\de\ot\id)(d\wt{\om}_1)$.
Now the assertion followis from the fact that $U_1\to S$ is the disjoint union of the families $X_1-q_1$ and $X_2-q_2$ (constant in the $t$ direction).
\end{proof}

\subsection{Case of genus $2$}
\label{formal-nbhd-div-sec}

The $(+,+)$ separating node divisor $D_0$ in $\ov{\SS}_2$ is the image of the clutching morphism
$$\ov{\SS}_{1,1}\times \ov{\SS}_{1,1}\to \ov{\SS}_2,$$
where $\ov{\SS}_{1,1}$ is the moduli space of even supercurves of genus $1$ with one NS puncture
(see \cite[Sec.\ 7.5]{FKP-supercurves}).
This morphism is an \'etale double cover (here we use the fact that there could be at most one separating node).
The corresponding involution of the source space corresponds to the swapping of the components.

Now we will show that similarly to the even case (considered in \cite{P-bu}), the gluing construction of Sec.\ \ref{super-gluing-sec}
gives an explicit description of the infinitesimal neighborhoods $D_0^{(n)}$ of $D_0$.

Recall that $\ov{\SS}_{1,1}^{(\infty)}$ denotes the moduli space of even stable supercurves of genus $1$ with an NS puncture and a choice of formal superconformal parameters $(x,\th)$.
The super gluing construction of Sec.\ \ref{super-gluing-sec} gives a morphism
\begin{equation}\label{g2-gluing-morphism-general-version}
\bigl(\ov{\SS}_{1,1}^{(\infty)}\times \ov{\SS}_{1,1}^{(\infty)}\times S^{(n)}_t\bigr)/\G_m^2\to D_0^{(n)}\sub \ov{\SS}_2,
\end{equation}
where $D_0^{(n)}$ is the $n$th infinitesimal neighborhood of $D_0$ (corresponding to the ideal $I_{D_0}^{n+1}$), and the action of $\G_m^2$ is induced by 
$$(\la_1,\la_2)((x_1,\th_1),(x_2,\th_2),t)=((\la_1^2x_1,\la_1\th_1),(\la_2^2 x_2,\la_2\th_2),\la_1\la_2t).$$
Combining this gluing morphism with the map $\phi\times\phi$ (where $\phi$ was defined in Lemma \ref{phi-map-lem}), we get for each $n\ge 0$ a morphism
\begin{equation}\label{nth-gluing-g2-sep-div-eq}
\wt{D}_0^{(n)}:=((\ov{\SS}'_{1,1,\bos}\times \A^{0|1})^2\times S^{(n)}_t)/\G_m^2\to D_0^{(n)},
\end{equation}
where 
the action of $\G_m^2$ is given by
$$(\la_1,\la_2)((\th_1,\eta_1),(\th_2,\eta_2),t)=((\la_1\th_1,\la_1\eta_1),(\la_2 \th_2,\la_2\eta_2),\la_1\la_2t)$$
(here we think of $\th_i$ as trivializations of $L|_p$).

Restricting the morphism \eqref{nth-gluing-g2-sep-div-eq} to the bosonizations we get a morphism
\begin{equation}\label{nth-gluing-g2-sep-div-bos-eq}
\wt{D}_{0,\bos}^{(n)}=((\SS'_{1,1\bos})^2\times S^{(n)}_t)/\G_m^2\to D_{0,\bos}^{(n)}\sub \ov{\SS}_{2,\bos}.
\end{equation}

These morphisms fit into commutative diagrams
\begin{diagram}
((\ov{\SS}'_{1,1,\bos})^2\times S^{(2n+1)}_t)/\G_m^2 &\rTo{}& \ov{\SS}_{2,\bos}\\
\dTo{} &&\dTo{}\\
((\ov{\MM}_{1,1}^{(1)})^2\times S^{(n)}_q)/\G_m^2 &\rTo{}&\ov{\MM}_2
\end{diagram}
where $\ov{\MM}_{1,1}^{(1)}\to \ov{\MM}_{1,1}$ is the $\G_m$-torsor corresponding to a choice of a nonzero tangent vector at the marked point, 
and the bottom horizontal arrow is given by the even gluing construction from Sec.\ \ref{even-gluing-sec},
and the map $S^{(2n+1)}_t\to S^{(n)}_q$ is given by 
$$q=-t^2.$$

\begin{prop}
The maps \eqref{nth-gluing-g2-sep-div-eq} and \eqref{nth-gluing-g2-sep-div-bos-eq} are \'etale and surjective.
\end{prop}

\begin{proof}
The proof is analogous to the classical case (see \cite[Thm.\ 1.2]{P-bu}).
\end{proof}


\begin{remark}
Note that the quotient $(\ov{\SS}'_{1,1,\bos}\times \A^{0|1})^2\times \A^1/\G_m^2$ is exactly the total space of the line bundle $N:=\Pi\LL^{-1}\boxtimes \Pi\LL^{-1}$ over $(\ov{\SS}_{1,1})^2$, so 
$\wt{D}_0^{(n)}$ can be identified with the $n$th neighborhood of the zero section in this total space. Here $N$ is isomorphic to the normal line bundle to $D_0$ 
(see \cite{FKP-supercurves}).
\end{remark}


\subsection{Superperiods near the $(+,+)$ separating node boundary for supercurves of genus $2$}\label{ss-2.2}

Now we will compute the superperiod matrix in the formal neighborhood of the separating node divisor.

We will apply the gluing construction of Sec.\ \ref{super-gluing-sec} 
to two copies of the standard family of genus $1$ supercurves with 1 NS-puncture over the upper half-plane (see Sec.\ \ref{even-g1-family-sec}).
More precisely, let $B$ denote the product of two copies of upper half-planes with coordinates $\tau_1$, $\tau_2$.
Then over $B$ we have two families of elliptic curves $C_1=\C/(\Z+\Z\tau_1)$ and $C_2=\C/(\Z+\Z\tau_2)$ equipped with even spin structures $L_i=\OO(u_i-p_i)$,
where $p_i$ corresponds to $z_i=0$ and $u_i$ corresponds to $z_i=1/2$.
The corresponding supercurves $X_1=(C_1,\OO\oplus L_1)$ and $X_2=(C_2,\OO\oplus L_2)$
are equipped with natural superconformal coordinates $(z_i,\nu_i)$ on $X_i$ near $0\in C_i$ (for $i=1,2$), where $z_i$ is the coordinate on $\C$, and 
$\nu_i$ corresponds to the trivializing local section of $L_i$ at $0\in C_i$ such that $(\nu_i)^{\ot 2}=dz_i$ in $\om_{C_i}$.
Next, we change the base to $B\times \A^{0|2}$, where we fix odd coordinates $\eta_1,\eta_2$ on $\A^{0|2}$, and define new superconformal coordinates $(x_i,\th_i)$ near $0\in C_i$,
by
\begin{equation}\label{x-th-z-nu-change-eq}
x_i=z_i+\eta_i\nu_i, \ \ \th_i=\nu_i-\eta_i.
\end{equation}
We define the NS-puncture $q_i\in X_i$, supported at $0\in C_i$, by the ideal $(x_i,\th_i)$ (for $i=1,2$). 

Note that by Lemma \ref{phi-map-lem}, the morphism
$$B\times \A^{0|2}\to \ov{\SS}^{(\infty)}_{1,1}\times \ov{\SS}^{(\infty)}_{1,1} $$
associated with our family is $\Z_2^2$-equivariant with respect to the natural action of $\Z_2$ on each of the factors $\A^{0|1}$ (resp., $\ov{\SS}^{(\infty)}_2$).
Hence, composing this morphism with the gluing map \eqref{g2-gluing-morphism-general-version}, we obtain a morphism
$$S/\Z_2^2\to D_0^{(n)}\sub \ov{\SS}_2,$$
where $S=B\times \A^{0|2}\times S^{(N)}_t$, and the action of $\Z_2^2$ is given by 
\begin{equation}\label{main-Z2-action-eq}
(\eps_1,\eps_2)(b,\eta_1,\eta_2,t)=(b,\eps_1\eta_1,\eps_2\eta_2,\eps_1\eps_2t)
\end{equation}
 
By definition, the corresponding stable supercurve $X$ over $S$
is glued from $U_1$ and $U_2$, where
$$U_1=S^{(N)}_t\times\A^{0|2}\times \bigl((X_1\setminus\{q_1\})\sqcup (X_2\setminus\{q_2\})\bigr)$$ and 
$$U_2=B\times\A^{0|2}\times \Spec(R_N[\![x_1,x_2,\th_1,\th_2]\!]/(x_1x_2+t^2,x_1\th_2-t\th_1,x_2\th_1+t\th_2,\th_1\th_2)),$$
where $R_N=\C[t]/(t^{N+1})$, $S^{(N)}_t:=\Spec R_N$.

As was discussed in Sec.\ \ref{superperiod-deg-sec}, the extended superperiod map for our family $X/S$ can be viewed as a natural map
of bundles over $S$,
$$\pi_*\om_{X/S}\to \ov{\HH}^1=R^1\pi_*[\OO_X\rTo{\de}\om_{X/S}].$$ 
Now, as in the even case, we will use the identification 
\begin{align*}\ov{\HH}^1&\simeq\om_{X/S}(U_1)/\de\OO_X(U_1)
  \\
                        &\simeq \om_{X_1/S}(X_1-q_1)/\de \OO(X_1-q_1)\oplus \om_{X_2/S}(X_2-q_2)/\de \OO(X_2-q_2)
\end{align*}
(see Lemma \ref{cech-diff-lem}) to compute this map explicitly.
We claim that similarly to \eqref{H1-basis-eq} the following elements provide a basis
 of $\ov{\HH}^1$:
\begin{equation}\label{++H1-basis-eq} 
 e_1=(\de(z_1),0), \ e_2=(0,\de(z_2)), \ f_1=(\wp(z_1,\tau_1)\de(z_1),0),\ f_2=(0,\wp(z_2,\tau_2)).
 \end{equation}
Indeed, this follows from the fact that for a holomorphic function $f(z_i)$ on $C_i\setminus\{0\}$, we have $\de(f)=f'(z_i)\de(z_i)$, and from the fact
that $\de^-$ induces an isomorphism of $\OO(X_i\setminus \{q_i\})^-=L_i(C_i\setminus\{0\})$ and $\om_{X_i/B}(X_i\setminus\{q_i\})^-$.

Note that the gluing construction is compatible with the $\G_m^2$-action, where the action of $(\la_1,\la_2)$ rescales $(x_i,\th_i,\eta_i)$ to
$(\la_i^2x_i,\la_i\th_i,\la_i\eta_i)$ and rescales $t$ to $\la_1\la_2 t$.

Recall that $\om_{U_2/S}(U_2)$ is generated as $\OO(U_2)$-module by $s_i=\de(\th_i)$, for $i=1,2$ (and $\de(x_i)=s_i\th_i$) and $s_0$ 
(satisfying $x_1s_0=-s_1\th_1$, $x_2s_0=s_2\th_2$ and other relations, see \cite[Sec.\ 6.2]{FKP-supercurves}).
Thus, a global section of $\om_{X/S}$ is given by a triple
$$(\om_1,\om_2,s_1\phi_1(x_1,\th_1)+s_2\phi_2(x_2,\th_2))+s_0\phi_0,$$
where $\om_i$ extends to a global section of $\om_{X/S}$ over $X_i\setminus\{q_i\}$, such that
\begin{equation}\label{om-phi-super-relation-eq}
  \begin{array}{l}\displaystyle
    \om_1=s_1\phi_1(x_1,\th_1)+\frac{ts_1}{x_1}\phi_2(-\frac{t^2}{x_1},\frac{t\th_1}{x_1})-\frac{s_1\th_1}{x_1}\phi_0, \\ 
\om_2=s_2\phi_2(x_2,\th_2)-\frac{ts_2}{x_2}\phi_1(-\frac{t^2}{x_2},-\frac{t\th_2}{x_2})+\frac{s_2\th_2}{x_2}\phi_0.
  \end{array}
\end{equation}

As in Sec.\ \ref{even-gluing-sec}, we consider unique global functions $f_n(z_i)$ on $C_i\setminus\{q_i\}$, for $n\ge 2$,
such that $f_n(z_i)=\frac{1}{(z_i)^n}+O(z_i)$, where $f_2(z_i)=\wp(z_i,\tau_i)$, for $i=1,2$. In addition, we have unique regular sections $\kappa_n(z_i)\nu_i$ of $L_i$ on 
$C_i\setminus\{q_i\}$, for $n\ge 1$, such that
$\kappa_n(z_i)=(\frac{1}{(z_i)^n}+O(1))$. Note that by the uniqueness, we get
$$f_n(-z_i)=(-1)^nf_n(z_i), \ \ \kappa_n(-z_i)=(-1)^n\kappa_n(z_i).$$
Recall that $(f_n(z_i))$ are expressed in terms of the derivatives of $\wp(z_i,\tau_i)$ (see \eqref{f-functions-eq}).
We will need some information on the functions $(\kappa_n)$.
We will write for brevity $\wp(z_i)=\wp(z_i,\tau_i)$, etc. We will use the function $h_{u_i}(z_i)=h_{u_i}(z_i,\tau_i)$ given by \eqref{h-u-eq}.

\begin{lemma}\label{kappa-functions-lem}
(i) One has 
$$\kappa_1(z_i)=\sqrt{\wp(z_i)-\wp(u_i)}=\frac{1}{z_i}-\frac{\wp(u_i)}{2}z_i+O(z^3),$$
$$\wp(z_i)+\kappa'_1(z_i)=-\frac{\wp(u_i)}{2}+O(z^2).$$

\noindent
(ii) One has
$$\kappa_2(z_i)=\kappa_1(z_i)h_{u_i}(z_i)=\frac{1}{z_i^2}+\frac{1}{2}\wp(u_i)+O(z^2).$$

\noindent
(iii) Let us set $f_1(z_i)=h_{u_i}(z_i)$, and $f_0(z_i)=1$.
Then for each $n\ge 3$, one has
$$\kappa_n(z_i)=\kappa_1(z_i)[f_{n-1}(z_i)+c_{n,3}f_{n-3}(z_i)+c_{n,5}f_{n-5}+\ldots]$$
for some constants $(c_{n,m})$ depending holomorphically on $\exp(2\pi i\tau_i)$. 
\end{lemma}

\begin{proof}
(i) Due to the relation between $\nu_i$ and $z_i$, we have 
$$\kappa_1^2(z_i)=\wp(z_i)-\wp(u_i),$$
where $u_i$ is a point of order $2$ on $C_i$ corresponding to the spin-structure $L_i$. 
Indeed, $\kappa_1(z_i)\nu_i$ is a global section of $L_i(q_i)$ vanishing at $u_i$, so
$\kappa_1^2(z_i)dz_i$ is a global section of $\om_{C_i}(2q_i)$ with double zero at $u_i$, which implies the above relation.
Hence,
$$\kappa_1(z_i)=\sqrt{\wp(z_i)-\wp(u_i)}=\frac{1}{z_i}-\frac{\wp(u_i)}{2}z_i+O(z^3),$$


\noindent
(ii) Since $\kappa_1(z_i)$ vanishes at $u_i$, we see that $\kappa_1(z_i)h_{u_i}(z_i)$ is regular at $u_i$.
Hence, the assertion follows from the expansion
$$\kappa_1(z_i)h_{u_i}(z_i)=\frac{1}{z_i^2}+\frac{1}{2}\wp(u_i)+O(z^2).$$

\noindent
(iii) This follows from (ii) by induction on $n$, using the fact that the difference
$\kappa_n(z_i)-\kappa_1(z_i) f_{n-1}(z_i)$ has a pole of order $\le n-2$ at $0$, and is
an even (resp., odd) function of $z_i$ for $n$ even (resp., odd).
\end{proof}

The following result is an analog of \cite[Prop.\ 5.5]{P-bu}.

\begin{lemma}\label{super-global-diff-base-lem} 
There exists a unique $\OO_S$-basis of $\pi_*\om_{X/S}$
$$\om^{(i)}=(\om^{(i)}_1,\om^{(i)}_2,s_1\phi_1^{(i)}(x_1,\th_1)+s_2\phi_2^{(i)}(x_2,\th_2)), \ i=1,2,$$
with 
\begin{align*}
  \om^{(1)}_1&=[1+\sum_{i\ge 2}a_if_i(z_1)]\de(z_1)+\sum_{i\ge 1}\de(\kappa_i(z_1)\nu_1)\a_i, \\ \om^{(1)}_2&=[\sum_{i\ge 2}b_if_i(z_2)]\de(z_2)+\sum_{i\ge 1}\de(\kappa_i(z_2)\nu_2)\b_i,\\
  \om^{(2)}_2&=[1+\sum_{i\ge 2}c_if_i(z_2)]\de(z_2)+\sum_{i\ge 1}\de(\kappa_i(z_2)\nu_2)\ga_i,
  \\
  \om^{(2)}_1&=[\sum_{i\ge 2}d_if_i(z_1)]\de(z_1)+\sum_{i\ge 1}\de(\kappa_i(z_1)\nu_1)\de_i,
\end{align*}
where $a_i,b_i,c_i,d_i\in t\OO_S^+$ and $\a_i,\b_i,\ga_i,\de_i\in t\OO_S^-$, and the coefficients of their expansions in $t$ are holomorphic as functions of 
$q_1=\exp(2\pi i\tau_1)$ and $q_2=\exp(2\pi i\tau_2)$. 
Furthermore, we have
$$\om^{(1)}_1\equiv \de(z_1)[1-t^3\eta_1\eta_2\wp(u_2)\wp(z_1)]-\de(\nu_1)\frac{\wp(u_2)}{2}[t^3\eta_2\kappa_1(z_1)+t^4\eta_1\kappa_2(z_1)] \mod (t^5),$$
$$\om^{(1)}_2\equiv \de(z_2)(t\eta_1\eta_2+t^2)\wp(z_2)-\de(\nu_2)[t\eta_1\kappa_1(z_2)+t^2\eta_2\kappa_2(z_2)] \mod (t^5),$$
$$\om^{(2)}_2\equiv \de(z_2)[1-t^3\eta_1\eta_2\wp(u_1)\wp(z_2)]+\de(\nu_2)\frac{\wp(u_1)}{2}[t^3\eta_1\kappa_1(z_2)-t^4\eta_2\kappa_2(z_2)] \mod (t^5),$$
$$\om^{(2)}_1\equiv \de(z_1)(t\eta_1\eta_2+t^2)\wp(z_1)+\de(\nu_1)[t\eta_2\kappa_1(z_1)-t^2\eta_1\kappa_2(z_1)] \mod (t^5).$$
\end{lemma}

\begin{proof}
It is easy to see that
$$\de(\nu_i)=s_i, \ \ \de(z_i)=s_i\nu_i, \ \ \de(f(z_i)\nu_i)=s_if(z_i).$$
Thus, the expansions of $\om^{(1)}_j$ can be rewritten as
\begin{align*}
&\om^{(1)}_1=s_1\nu_1[1+\sum_{i\ge 2}a_if_i(z_1)]+s_1\sum_{j\ge 1}\a_j\kappa_j(z_1)=\\
&s_1\th_1[1+\sum_{i\ge 2}a_if_i(x_1)+\sum_{j\ge 1}\eta_1\a_jg'_j(x_1)]+s_1[\eta_1+\sum_{i\ge 2}\eta_1a_if_i(x_1)+\sum_{j\ge 1}\a_j\kappa_j(x_1)],
\end{align*}
\begin{align*}
&\om^{(1)}_2=s_2\nu_2[\sum_{i\ge 2}b_if_i(z_2)]+s_2\sum_{j\ge 1}\b_j\kappa_j(z_2)=\\
&s_2\th_2[\sum_{i\ge 2}b_if_i(x_2)+\sum_{j\ge 1}\eta_2\b_jg'_j(x_2)]+s_2[\sum_{i\ge 2}\eta_2b_if_i(x_2)+\sum_{j\ge 1}\b_j\kappa_j(x_2)],
\end{align*}
and similarly for $\om^{(2)}_i$.


Let us focus on the existence and uniqueness of $\om=\om^{(1)}$ (the case of $\om^{(2)}$ is considered similarly).
The equations \eqref{om-phi-super-relation-eq} are satisfied modulo $t$ for
$$\om_1\equiv s_1\nu_1 \!\!\!\!\mod(t), \ \ \om_2\equiv 0 \!\!\!\!\mod(t), \ \ \phi_1\equiv \nu_1=\th_1+\eta_1 \!\!\!\!\mod(t), \ \ \phi_2\equiv 0 \!\!\!\!\mod(t).$$

Next, assuming that $\om_1,\om_2,\phi_1,\phi_2$ are known modulo $t^n$ and \eqref{om-phi-super-relation-eq} holds modulo $t^n$,
we observe that \eqref{om-phi-super-relation-eq} modulo $t^{n+1}$ (more precisely the second summand of each of the right-hand sides)
determine polar parts of $\om_1$ and $\om_2$ modulo $t^{n+1}$,
and hence, all the coefficients $a_i,b_i,\a_j,\b_j$ modulo $t^{n+1}$. In more detail, we first determine $\a_1$ (resp., $\b_1$)
by looking at the coefficient of $x_1^{-1}s_1$ (resp., $x_2^{-1}s_2$). Then we determine $a_2$ (resp., $b_2$) by looking at the coefficient
of $x_1^{-2}s_1\th_1$ (resp., $x_2^{-2}s_2\th_2$). Next, we determine $\a_2$ (resp., $\b_2$) by looking at the coefficient of
$x_1^{-2}s_1$ (resp., $x_2^{-2}s_2$), etc. After $\om_1$ (resp., $\om_2$) is determined modulo $t^{n+1}$,
we obtain that $s_1\phi_1$ (resp., $s_2\phi_2$) modulo $t^{n+1}$ is given by the regular part of the expansions of $\om_1$ (resp., $\om_2$).



Here are the first few steps of calculating $\om_1,\om_2,\phi_1,\phi_2$, where we take $f_2=\wp$.

\medskip

\noindent
{\bf mod $t^2$}.
The polar part of $\om_1$ is zero $\mod(t^2)$, so $\om_1\equiv s_1\nu_1 \mod(t^2)$ and $\phi_1\equiv \th_1+\eta_1\mod(t^2)$.
 
The polar part of $\om_2$ comes from the constant term of $\phi_1$, so it
 is $-\frac{ts_2}{x_2}\eta_1$. Hence,
 $$\b_1\equiv -t\eta_1, \ \ b_2=\eta_2\b_1=t\eta_1\eta_2\mod(t^2).$$
Therefore,
\begin{align*}
  \om_2&\equiv s_2\nu_2 t\eta_1\eta_2\wp(z_2)-s_2t\eta_1\kappa_1(z_2)
  \\
  &\equiv
    s_2\th_2 t\eta_1\eta_2[\wp(x_2)+\tau'_1(x_2)]-s_2t\eta_1\kappa_1(x_2)\!\!\!\!\mod(t^2),
  \\
  \phi_2&\equiv \th_2 t\eta_1\eta_2[\wp(x_2)+\tau'_1(x_2)]-t\eta_1\kappa_1(x_2)_{\ge 0}.
\end{align*}

\medskip

\noindent
{\bf mod $t^3$}.
The polar part of $\om_1$ is still zero $\mod(t^3)$, so $\om_1\equiv s_1\nu_1 \mod(t^3)$ and $\phi_1\equiv \th_1+\eta_1\mod(t^3)$.

The polar part of $\om_2$ comes from the both terms $\eta_1$ and $\th_1$ in $\phi_1$, and is given by
$$-s_2\frac{t\eta_1}{x_2}+s_2\th_2\frac{t^2}{x_2^2}.$$
Thus, we get
$$\b_1\equiv -t\eta_1, \ \ b_2\equiv \eta_2\b_1+t^2=t\eta_1\eta_2+t^2, \ \ \b_2\equiv-\eta_2b_2=-t^2\eta_2,$$
\begin{align*}
&\om_2\equiv s_2\nu_2(t\eta_1\eta_2+t^2)\wp(z_2)-s_2[t\eta_1\kappa_1(z_2)+t^2\eta_2\kappa_2(z)]=\\
&s_2\th_2[(t\eta_1\eta_2+t^2)\wp(x_2)+t\eta_1\eta_2\tau'_1(x_2)]+s_2[t^2\eta_2\wp(x_2)-t\eta_1\kappa_1(x_2)-t^2\eta_2\kappa_2(x_2)],
\end{align*}
$$\phi_2\equiv \th_2[t^2\wp(x_2)_{\ge 0}+t\eta_1\eta_2(\wp(x_2)+\tau'_1(x_2))]+t^2\eta_2(\wp(x_2)-\kappa_2(x_2))-t\eta_1\kappa_1(x_2)_{\ge 0}.$$

\medskip

\noindent
{\bf mod $t^4$}.
The polar part of $\om_1$ now has contributions from the terms $\th_2 t\eta_1\eta_2(\wp(x_2)+\tau'_1(x_2))$ and
$t^2\eta_2(\wp(x_2)-\kappa_2(x_2))$ in $\phi_2$, so it is given by
$$-\frac{s_1\th_1}{x_1^2} t^3\eta_1\eta_2\frac{\wp(u_2)}{2}-\frac{s_1}{x_1} t^3\eta_2\frac{\wp(u_2)}{2}.$$
Hence,
$$\a_1\equiv-t^3\eta_2\frac{\wp(u_2)}{2}, \ \ a_2\equiv\eta_1\a_1-t^3\eta_1\eta_2\frac{\wp(u_2)}{2}=-t^3\eta_1\eta_2\wp(u_2),$$
\begin{align*}
&\om_1\equiv s_1\nu_1[1-t^3\eta_1\eta_2\wp(u_2)\wp(z_1)]-s_1t^3\eta_2\frac{\wp(u_2)}{2}\kappa_1(z_1)=\\
&s_1\th_1[1-t^3\eta_1\eta_2\wp(u_2)\wp(x_1)-t^3\eta_1\eta_2\frac{\wp(u_2)}{2}\tau'_1(x_1)]+s_1[\eta_1-t^3\eta_2\frac{\wp(u_2)}{2}\kappa_1(x_1)],
\end{align*}
$$\phi_1\equiv \th_1[1-t^3\eta_1\eta_2\wp(u_2)[\wp(x_1)_{\ge 0}+\frac{1}{2}\tau'_1(x_1)_{\ge 0}]]+\eta_1-t^3\eta_2\frac{\wp(u_2)}{2}\kappa_1(x_1)_{\ge 0}.$$

The polar part of $\om_2$ is the same as modulo $t^3$, so we have the same formulas for $\om_2$ and $\phi_2$ as modulo $t^3$.

\noindent
{\bf mod $t^5$}. 
The polar part of $\om_1$ is given by
\begin{align*}
&\frac{ts_1}{x_1}\phi_2(-\frac{t^2}{x_1},\frac{t\th_1}{x_1})\equiv\frac{s_1\th_1}{x_1^2}[t^4\wp(-\frac{t^2}{x_1})_{\ge 0}+
t^3\eta_1\eta_2(\wp(-\frac{t^2}{x_1})+\tau'_1(-\frac{t^2}{x_1}))]+\\
&\frac{s_1}{x_1}[t^3\eta_2(\wp(-\frac{t^2}{x_1})-\kappa_2(-\frac{t^2}{x_1}))-t^2\eta_1\kappa_1(-\frac{t^2}{x_1})_{\ge 0}]\equiv\\
&-\frac{s_1\th_1}{x_1^2}t^3\eta_1\eta_2\frac{\wp(u_2)}{2}-
\frac{s_1}{x_1}t^3\eta_2\frac{\wp(u_2)}{2}-
\frac{s_1}{x_1^2}t^4\eta_1\frac{\wp(u_2)}{2}].
\end{align*}
Hence, we get
$$\a_1\equiv-t^3\eta_2\frac{\wp(u_2)}{2}, \ \ a_2\equiv\eta_1\a_1-t^3\eta_1\eta_2\frac{\wp(u_2)}{2}=-t^3\eta_1\eta_2\wp(u_2), \ \
\a_2\equiv -t^4\eta_1\frac{\wp(u_2)}{2},$$
$$\om_1\equiv s_1\nu_1[1-t^3\eta_1\eta_2\wp(u_2)\wp(z_1)]-s_1\frac{\wp(u_2)}{2}[t^3\eta_2\kappa_1(z_1)+t^4\eta_1\kappa_2(z_1)].$$

The polar part of $\om_2$ is still the same as before, so $\om_2$ and $\phi_2$ do not change.
\end{proof}

\begin{prop}\label{++periods-wrong-basis-prop}
With the respect to the basis \eqref{++H1-basis-eq}, the map 
$$\pi_*\om_{X/S}\to \ov{\HH}^1$$
is given by 
$$\om^{(1)}\mapsto \phi_{11}(t)e_1+\phi_{12}(t)e_2+\psi_{11}(t)f_1+\psi_{12}(t)f_2,$$
$$\om^{(2)}\mapsto \phi_{21}(t)e_1+\phi_{22}(t)e_2+\psi_{21}(t)f_1+\psi_{22}t)f_2,$$
where 
$$\phi_{11}\equiv \phi_{22}\equiv 1\mod(t^5),
$$
$$\psi_{11}\equiv-t^3\eta_1\eta_2\wp(u_2) \mod(t^5), \ \ \psi_{22}\equiv -t^3\eta_1\eta_2\wp(u_1) \mod(t^5)$$
$$\phi_{12}\equiv \phi_{21}\equiv 0\mod(t^5), \ \ \psi_{12}\equiv \psi_{21}\equiv t\eta_1\eta_2+t^2 \mod(t^5).$$
Furthermore, the coefficients of $t^k$ (resp., $\eta_1\eta_2 t^k$) in $\phi_{ij}$ and of $\psi_{ij}$ are holomorphic functions of $q_1=\exp(2\pi i \tau_1)$ and $q_2=\exp(2\pi i\tau_2)$.
\end{prop}

\begin{proof}
Note that for $n\ge 1$, one has $\wp^{(n)}(z_i,\tau_i)\de(z_i)\equiv 0$ in $\ov{\HH}^1$.
By \eqref{f-functions-eq}, for each $k\ge 1$, $f_{2k+1}(z_i)$ is proportional to $\wp^{(2k-1)}$, whereas
$$f_{2k+2}(z_i)\equiv \phi_i[2k+2] \mod (\wp^{(2k)}(z_i)),$$
for some constants $\phi_i[2k+2]$.
Hence, with the notation of Lemma \ref{super-global-diff-base-lem}, we get 
$$\om^{(1)}\equiv (1+\sum_{k\ge 1}a_{2k+2}\phi_1[2k+2])e_1+a_2f_1+(\sum_{k\ge 1}b_{2k+2}\phi_2[2k+2])e_2+b_2f_2,$$
$$\om^{(2)}\equiv (1+\sum_{k\ge 1}c_{2k+2}\phi_2[2k+2])e_2+a_2f_2+(\sum_{k\ge 1}d_{2k+2}\phi_1[2k+2])e_1+d_2f_1$$
in $\ov{\HH}^1$. Now the assertion follows from  Lemma \ref{super-global-diff-base-lem}.
\end{proof}

\subsubsection{Periods calculation}\label{periods-calculation-subsubsec}

The basis $(e_1,f_1,e_2,f_2)$ of $\ov{\HH}^1$ is not horizontal with respect to the Gauss-Manin connection $\nabla^{GM}$. Lemma \ref{cech-diff-lem}(ii) implies that
$\nabla^{GM}$ acts on $(e_i,f_i)$ in the same way as for the standard family of elliptic curves. Hence, a horizontal basis is obtained from $(e_i,f_i)$ by using the periods
along the standard loops $(\a_i,\b_i)$ (where $(\a_i,\b_i)$ are the cycles considered in Sec.\ \ref{even-gluing-sec}).
Equivalently, we should replace the basis $(\om^{(1)},\om^{(2)})$ of $\pi_*\om_{X/S}$ by a basis with the normalized $\a$-periods, and then compute their $\b$-periods.

From the expressions for $\om^{(1)}$ and $\om^{(2)}$, we get the following $\a$-periods of $\om^{(1)}$ and $\om^{(2)}$:
$$\int_{\a_1}\om^{(1)}=\phi_{11}+A_1\psi_{11}\equiv 1-A_1t^3\eta_1\eta_2\wp(u_2) \mod(t^5),$$
$$\int_{\a_1}\om^{(2)}=\phi_{21}+A_1\psi_{21}\equiv A_1(t\eta_1\eta_2+t^2) \mod(t^5),$$
$$\int_{\a_2}\om^{(1)}=\phi_{12}+A_2\psi_{12}\equiv A_2(t\eta_1\eta_2+t^2) \mod(t^5),$$
$$\int_{\a_2}\om^{(2)}=\phi_{22}+A_2\psi_{22}\equiv 1-A_2t^3\eta_1\eta_2\wp(u_1) \mod(t^5),$$
where $A_i$ are given by \eqref{Ai-eq}.

Using this we can find the normalized basis $\om^{(1)}_{\norm},\om^{(2)}_{\norm}$ such that
$$\int_{\a_i}\om^{(j)}_{\norm}=\de_{ij}.$$
Namely, setting 
$$\De=\De(\tau_1,\tau_2):=\det(\phi+\psi A)=\left|\begin{matrix} \phi_{11}+\psi_{11}A_1 & \phi_{12}+\psi_{12}A_2\\ \phi_{21}+\psi_{21}A_1 & \phi_{22}+\psi_{22}A_2\end{matrix}\right|,$$
we have
\begin{align*}
&\om^{(1)}_{\norm}=\De^{-1}\cdot \left|\begin{matrix} \om^{(1)} & \phi_{12}+\psi_{12}A_2\\ \om^{(2)} & \phi_{22}+\psi_{22}A_2\end{matrix}\right|=\\
&[1+t^3\eta_1\eta_2A_1(\wp(u_2)+2A_2)+A_1A_2t^4]\cdot[\om^{(1)}-A_2(t\eta_1\eta_2+t^2)\om^{(2)}] \mod(t^5),\\
&\om^{(2)}_{\norm}=\De^{-1}\cdot
\left|\begin{matrix} \phi_{11}+\psi_{11}A_1 & \om^{(1)}\\ \phi_{21}+\psi_{21}A_1 & \om^{(2)}\end{matrix}\right|=\\
&[1+t^3\eta_1\eta_2A_2(\wp(u_1)+2A_1)+A_1A_2t^4]\cdot[\om^{(2)}-A_1(t\eta_1\eta_2+t^2)\om^{(1)}] \mod(t^5).
\end{align*}

We derive the following information about the canonical projection and the superperiod matrix $\Om_{ij}=\int_{\b_i}\om^{(j)}_\norm$ in terms of the gluing coordinates
$(\tau_1,\tau_2,t,\eta_1,\eta_2)$ near the $(+,+)$ separating node divisor.
As before, for any holomorphic object $X$ over $\ov{\SS}_2$ we denote by $\wt{X}$ the corresponding holomorphic object over the complex conjugate space $\ov{\SS}_2^c$.

\begin{prop}\label{period-det-prop} 
(i) One has
$$\Om_{11}=\tau_1+\De^{-1}\cdot \left|\begin{matrix} \psi_{11} & \phi_{12}+\psi_{12}A_2\\ \psi_{21} & \phi_{22}+\psi_{22}A_2\end{matrix}\right|=
\tau_1-2\pi i (t^3\eta_1\eta_2(\wp(u_2)+2A_2)+ A_2t^4) \!\!\!\!\mod\! (t^5),$$
$$\Om_{22}=\tau_2+\De^{-1}\cdot \left|\begin{matrix} \phi_{11}+\psi_{11}A_1 & \psi_{12}\\ \phi_{21}+\psi_{21}A_1 & \psi_{22}\end{matrix}\right|=
\tau_2-2\pi i (t^3\eta_1\eta_2(\wp(u_1)+2A_1)+A_1t^4)\!\!\!\! \mod \!(t^5),$$
$$\Om_{12}=\Om_{21}=2\pi i\De^{-1}\cdot \left|\begin{matrix} \psi_{12} & \phi_{12} \\ \psi_{22} & \phi_{22} \end{matrix}\right|=
2\pi i\De^{-1}\cdot \left|\begin{matrix} \phi_{11} & \psi_{11} \\ \phi_{21} & \psi_{21} \end{matrix}\right|=
2\pi i (t\eta_1\eta_2+t^2) \!\!\!\!\mod (t^5).$$

(ii) The pullbacks $t',\tau'_1,\tau'_2$ of $t,\tau_1,\tau_2$ under the canonical projection $\pi^{\can}$ are given by 
$$t'=t+\eta_1\eta_2/2+O(t^4)\eta_1\eta_2,$$
$$\tau'_1=\tau_1-2\pi i t^3\wp(u_2)\eta_1\eta_2+O(t^5)\eta_1\eta_2, \ \ \tau'_2=\tau_2-2\pi i t^3\wp(u_1)\eta_1\eta_2+O(t^5)\eta_1\eta_2.$$

\noindent
(iii) Set $s:=\om^{(1)}\we \om^{(2)}$, $s_{\norm}:=\om^{(1)}_{\norm}\we \om^{(2)}_{\norm}$. Also, 
let us write $\Om=\Om_0+\Om_1\eta_1\eta_2$, where $\Om_0$ and $\Om_1$ depend only on even variables.
The canonical section of
$\Ber_1\boxtimes \wt{\Ber}_1$, coming from the hermitian form on $\pi_*\om_{X/S}$ is given by
$$h:=\det(\wt{\Om}-\Om)^{-1}\cdot s_{\norm}\cdot \wt{s}_{\norm}=(h_0+h_1\cdot \eta_1\eta_2+\wt{h}_1\cdot\wt{\eta}_1\wt{\eta}_2+h_{11}\cdot\eta_1\eta_2\wt{\eta}_1\wt{\eta}_2)\cdot s\cdot\wt{s},$$
where 
$$h_0=d^{-1}(1+O(t^4)+O(\wt{t}^4)),$$
$$h_1=t[-8\pi^2 d^{-2} \wt{t}^2+O(t^2)+O(\wt{t}^5)), \ \ \wt{h}_1=\wt{t}[-8\pi^2 d^{-2} t^2+O(\wt{t}^2)+O(t^5)),$$
$$h_{11}=t\wt{t}[-8\pi^2 d^{-2}+O(t^4)+O(\wt{t}^4)+O(t^2\wt{t}^2)],$$
where $d=\det(\wt{\Om}_0-\Om_0)$.
\end{prop}

\begin{proof}
(i) The formulas for $\Om_{ij}$ follow directly from the formulas for $\om^{(1)}_{\norm}$ and $\om^{(2)}_{\norm}$ and from Proposition \ref{++periods-wrong-basis-prop}.
The symmetry of $\Om$ follows from the similar symmetry for smooth supercurves, which is well known.

\noindent
(ii) This follows from (i). Namely, let $\Om^{\bos}$ denote the corresponding classical period matrix. Then $(\pi^{\can})^*\Om^{\bos}_{12}=\Om_{12}$, i.e.,
$$(\pi^{\can})^*(t^2)=t\eta_1\eta_2+t^2 \mod (t^5).$$
Writing $t'=(\pi^{\can})^*t=t+f\eta_1\eta_2$, we deduce that $2tf\equiv t \mod (t^5)$, i.e., $f\equiv 1/2 \mod (t^4)$.
The formulas for $\tau'_i=(\pi^{\can})^*\tau_i$ are obtained similarly from the equations $(\pi^{\can})^*\Om^{\bos}_{ii}=\Om_{ii}$, for $i=1,2$.

\noindent
(ii) Consider the symmetric bilinear form on matrices given by
$$\lan A,B\ran=\tr(A^{adj}B)=a_{11}b_{22}+a_{22}b_{11}-a_{12}b_{21}-a_{21}b_{21}.$$
Then if $\eps_1$ and $\eps_2$ are commuting variables satisfying $\eps_i^2=0$, we have
$$\det(A+B_1\eps_1+B_2\eps_2)=\det(A)+\lan A,B_1\ran\eps_1+\lan A,B_2\ran\eps_2+\lan B_1,B_2\ran\eps_1\eps_2.$$
Applying this in our case we get
\begin{align*}
&\det(\wt{\Om}-\Om)=\det(\wt{\Om}_0-\Om_0+\wt{\Om}_1\wt{\eta}_1\wt{\eta}_2-\Om_1\eta_1\eta_2)=\\
&\det(\wt{\Om}_0-\Om_0)+\lan \wt{\Om}_0-\Om_0,\wt{\Om}_1\ran\wt{\eta}_1\wt{\eta}_2-\lan \wt{\Om}_0-\Om_0,\Om_1\ran\eta_1\eta_2-
\lan\wt{\Om}_1,\Om_1\ran\eta_1\eta_2\wt{\eta}_1\wt{\eta}_2.
\end{align*}
Note that we have
$$\lan \wt{\Om}_0-\Om_0,\wt{\Om}_1\ran=:\wt{t}a=\wt{t}(8\pi^2t^2+O(\wt{t}^2)+O(t^5)),$$
$$\lan \wt{\Om}_0-\Om_0,\Om_1\ran=-t\wt{a}=t(-8\pi^2\wt{t}^2+O(t^2)+O(\wt{t}^5)),$$
$$\lan \wt{\Om}_1,\Om_1\ran=:t\wt{t}b=t\wt{t}(-8\pi^2+O(t^4)+O(\wt{t}^4)+O(t^2\wt{t}^2)).$$
Thus, we obtain
\begin{equation}
\det(\wt{\Om}-\Om)^{-1}=d^{-1}-d^{-2}\wt{a}t\eta_1\eta_2-d^{-2}a\wt{t}\wt{\eta}_1\wt{\eta}_2+(d^{-2}b+2d^{-3}a\wt{a})t\wt{t}\eta_1\eta_1\wt{\eta}_1\wt{\eta}_2.
\end{equation}

Next, let $M$ be the transition matrix from $\om^{(1)},\om^{(2)}$ to $\om^{(1)}_{\norm}$, $\om^{(2)}_{\norm}$. We can write $M=M_0+M_1\eta_1\eta_2$,
where 
$$M_0\equiv \left(\begin{matrix} 1+A_1A_2t^4 & -A_1t^2 \\ -A_2t^2 & 1+A_1A_2t^4\end{matrix}\right) \mod t^5,$$
$$M_1\equiv \left(\begin{matrix} (A_1\wp(u_2)+2A_1A_2)t^3 & -A_1t \\ -A_2t & (A_2\wp(u_1)+2A_1A_2)t^3\end{matrix}\right) \mod t^5.$$
Hence, 
$$s_{\norm}\cdot \wt{s}_{\norm}=\det(M)\det(\wt{M}) s\cdot\wt{s},$$
where
$$\det(M)=m_0+m_1t^3\eta_1\eta_2, \ \ \det(\wt{M})=\wt{m}_0+\wt{m}_1\wt{t}^3\wt{\eta}_1\wt{\eta}_2,$$
with
$$m_0=\det(M_0)=1+A_1A_2t^4+O(t^5),$$
$$t^3m_1:=\lan M_0,M_1\ran=t^3(A_1\wp(u_2)+A_2\wp(u_1)+2A_1A_2+O(t^2)).$$ 
Now we have 
$$h=\det(\wt{\Om}-\Om)^{-1})\cdot \det(M)\cdot \det(\wt{M})\cdot s\cdot \wt{s},$$
so we get 
$$h_0=d^{-1}m_0\wt{m}_0, \ \ h_1=-d^{-2}\wt{a}m_0\wt{m}_0t+d^{-1}m_1\wt{m}_0t^3,$$
$$h_{11}=t\wt{t}\cdot[(d^{-2}b+2d^{-3}a\wt{a})m_0\wt{m}_0-d^{-2}\wt{a}m_0\wt{m}_1\wt{t}^2-d^{-2}a\wt{m}_0m_1t^2+d^{-1}m_1\wt{m}_1t^2\wt{t}^2],$$
which gives the claimed result.
\end{proof}

We will need the following corollary in our study of the superstring measure.

\begin{cor}\label{h0-h11-cor} 
One has
$$\frac{h_0^4h_{11}}{t\wt{t}}|_{t=0}=-8\pi^2h_0^6|_{t=\wt{t}=0}.$$
\end{cor}

\subsection{Gluing construction near deeper strata of the $(+,+)$ separating node divisor}

\subsubsection{Weierstrass model near a nodal cubic}

Let $\Disc$ denote the disc $|q|<1$, and let $\Disc'=\Disc\setminus\{0\}$. 
We have a standard family $E_q$ of elliptic curves over $\Disc'$ given by the Weierstrass model
$$y^2=4x^3-g_2x-g_3,$$
where 
$$g_2=60G_4=\frac{4\pi^4}{3}(1+240q+\ldots), \ \ g_3=140G_6=\frac{8\pi^6}{27}(1-504q+\ldots),$$
with the uniformization given by $x=\wp(z,\tau)$, $y=\wp'(z,\tau)$, where $q=\exp(2\pi i \tau)$ (and $z\not\in \Z+\Z\tau$).

This family extends to a family of curves $(E,p)$ in $\ov{\MM}_{1,1}$ over $\Disc$, with $E_0$ being the nodal cubic
corresponding to $q=0$ (here $p$ is the point of the cubic at infinity).
We can still use the uniformization in a neighborhood of $q=0$, viewing $\wp$ and $\wp'$ as functions of $z$ and $q$.
Furthermore, the limit of $\wp$ as $q\to 0$ (for $0<\operatorname{Im}(z)<\operatorname{Im}(\tau)$) is 
$$\wp_0(u)=(\pi i)^2\cdot [\frac{4u}{(1-u)^2}+\frac{1}{3}],$$
where $u=\exp(2\pi i z)$. 
The derivation $\pa_z$ corresponds to $(2\pi i)u^{-1}\pa_u$, so the degeneration of $\wp'$ is
$$\wp'_0(u)=(2\pi i)u^{-1}\frac{d}{du}\wp_0(u)=(\pi i)^3 \cdot \frac{8u(1+u)}{(1-u)^3}.$$
The coordinate $u$ can be viewed as the coordinate on the $\P^1\setminus\{0,\infty\}$, where we consider the normalization
morphism $\P^1\to E_0$ gluing $0$ and $\infty$ into the node.

Similarly, all elliptic functions $(f_n(z))_{n\ge 2}$, regular on $E_q\setminus p$, where $f_n=\frac{1}{z^n}+O(1)$,
extend to the whole family over $\Disc$.

The global differential $\om=dx/y=dz$ extends to a regular section of the relative dualizing sheaf, so that
$\om|_{E_0}=(2\pi i)^{-1}du/u$.

Furthermore, we can lift the above family to a family of spin-curves, by considering the relative point of order $2$ given by $z=1/2$,
which specializes to a nontrivial point of order $2$ given by $u=-1$ on $E_0$. The corresponding spin structure $L$ has a natural trivialization $\nu$ near $0$,
so that $\nu^2=dz$.

We claim that the sections $\kappa_n(z)\nu$ of $L$ on $E_q\setminus p$, such that $\kappa_n(z)=\frac{1}{z^n}+O(1)$, still make sense on the family over $\Disc$.
Indeed, by Lemma \ref{kappa-functions-lem}(iii),
it is enough to check this for $\kappa_1$ and $\kappa_2$.
For $\kappa_1$ this follows from the relation $\kappa_1^2=\wp(z)-\wp(1/2)$, which leads to
$$\kappa_1|_{E_0}=(\pi i)\cdot \frac{u+1}{u-1}.$$
Recall that 
$$\kappa_2(z)=\kappa_1(z)\cdot h_{1/2}(z),$$
where $h_{1/2}(z)=\zeta(z)-\zeta(z-1/2)-\zeta(1/2)$ (see Sec.\ \ref{ell-sec}).
Now it is easy to check that $h_{1/2}(z)$ extends over $\Disc$, with
$$h_{1/2}|_{E_0}=(\pi i)\cdot \frac{4u}{u^2-1}.$$

\subsubsection{The glued family}

Now we consider the glued family $X/S$ with the base $B\times \A^{0|2}\times S_t$,
where $B$ is the base of an even separating node degeneration of spin-curves of genus $2$.

We claim that the basis $(\om^{(1)},\om^{(2)}$ of $\pi_*\om_{X/S}$ constructed in Lemma \ref{super-global-diff-base-lem},
as well as the basis $(e_1,e_2,f_1,f_2)$ of $\ov{\HH}^1=R^1\pi_*[\OO_X\to \om_{X/S}]$
make sense near nodal curves glued out of two genus curves $C_1$, $C_2$, where $C_1$ or both $C_1$ and $C_2$ can be singular.

Indeed, for the basis of $\ov{\HH}^1$, this is clear from formulas \eqref{++H1-basis-eq}.
The proof of Lemma \ref{super-global-diff-base-lem} only uses the fact that the functions $(1,(f_n(z_i))_{n\ge 2})$ form a basis of $\OO(C_i\setminus q_i)$, while 
$\kappa_n(z_i)\nu_i$ form a basis of $L_i(C_i\setminus q_i)$.
But this continues to hold for the entire family where $C_i$ are allowed to degenerate.

\subsubsection{Superperiods near the separating node degeneration}\label{superperiods-deg-sec}

The calculation of the superperiods for the glued family in Sec.\ \ref{periods-calculation-subsubsec} still makes sense when the elliptic curves $C_i$ degenerate (i.e., corresponding
parameters $q_i=\exp(2\pi i\tau_i)$ can be zero).
The superperiods are regular along this family and have expansions
$$\Om_{11}=\tau_1+O(t^3), \ \ \Om_{22}=\tau_2+O(t^3), \ \ \Om_{12}=\Om_{21}=2\pi i (t\eta_1\eta_2+t^2)+O(t^5).$$

This implies that the classical period matrix $\Om^{\bos}$ near the point $C_1\cup C_2$, where $C_1$ is a nodal cubic, satisfy
$$\exp(2\pi i \Om^{\bos}_{11})=q_1\cdot (1+O(q^2)), \ \ \Om^{\bos}_{22}=\tau_2+O(q^2), \ \ \Om^{\bos}_{12}=2\pi i q+O(q^2),$$
which implies that $\exp(2\pi i \Om^{\bos}_{11})$, $\Om^{\bos}_{22}$ and $\Om^{\bos}_{12}$ form local coordinates on $\ov{\MM}_2$ near this point.

Similarly, we see that near the point $C_1\cup C_2$, where both $C_1$ and $C_2$ are nodal, the functions
$\exp(2\pi i\Om^{\bos}_{11})$, $\exp(2\pi i\Om^{\bos}_{22})$ and $\Om^{\bos}_{12}$ form local coordinates on $\ov{\MM}_2$. 

These facts will be used in proving that the canonical projection of $\ov{\SS}_2$ is regular everywhere along the $(+,+)$ separating node divisor
(see Prop.\ \ref{ext-Tor-prop} and Theorem \ref{projection-nonsep-prop} below).

\subsection{Superperiods near the $(-,-)$ separating node boundary for supercurves of genus $2$}\label{--sep-sec}

The gluing construction also works to give a description of the formal neighborhood of the $(-,-)$ separating node divisor, i.e., the divisor corresponding
to the stable spin curves given as the nodal union $C=C_1\cup C_2$, $L=L_1\oplus L_2$, where $(C_i,L_i)$ is a spin curve of genus $1$ with odd $h^0(L_i)$.
The difference is that the universal curve over the moduli space $\ov{\SS}^-_{1,1}$ corresponding to supercurves with odd underlying spin structure
is not split. Namely, the moduli space $\ov{\SS}^-_{1,1}$ (studied in \cite{Levin}, \cite{Rabin} and \cite{P-superell})
is a quotient of the product ${\frak H}\times \C^{0|1}$, where ${\frak H}$ is the upper halfplane with the
coordinate $\tau$, and $\C^{0|1}$ has coordinate $\eta$, and the universal curve
 is constructed as the quotient of the relative $\C^{1|1}$ by the action of $\Z^2$ 
$$(x,\th)\mapsto (x+1,\th), \  (x,\th)\mapsto (x+\tau+\th\eta,\th+\eta)$$
where $(\tau,\eta)$ are parameters on the moduli space ($\tau$ is in the upper half-plane, $\eta$ is an odd coordinate).
The (relative) NS-puncture at $x=0$ is given by the superconformal coordinates $(x,\th)$.

Now we consider the base $B$ with coordinates $(\tau_1,\tau_2,\eta_1,\eta_2)$ obtained as the product of two copies of ${\frak H}\times \C^{0|1}$,
and let $(X_1,q_1)$ and $(X_2,q_2)$ be the genus $1$ supercurves with NS punctures obtained as above using $(\tau_i,\eta_i)$, for $i=1,2$.
Below we will freely use some facts about $(X_i,q_i)$ from \cite{P-superell}.

Next, we change the base to $S=B\times \C[t]/(t^N)$ and 
use the superconformal coordinates $(x_i,\th_i)$ on $X_i$ near $0\in C_i$, and glue the smooth part of $X_1\cup X_2$ with
the standard deformation of the NS node in variables $(x_i,\th_i)$.
Note that $s_i:=\de(\th_i)$ is a free generator of the sheaf $\om_{X_i/S}$. Also we have the $\OO_S$-bases
\begin{equation}\label{--diff-basis-eq}
\begin{array}{l}
1,\ (f_n(x_i,\th_i):=D^{2(n-2)}R(x_i,\th_i;\tau_i,\eta_i))_{n\ge 2}, \\ 
\psi_1(x_i,\tau_i):=\th_i-\eta\zeta_1(x_i,\tau_i), \ \psi_2:=\th_i\zeta'_1(x_i,\tau_i)+\eta\dot{\zeta}_1(x_i,\tau_i),\\ 
(\psi_n(x_i,\th_i):=D^{1+2(n-3)}R(x_i,\th_i;\tau_i,\eta_i))_{n\ge 3},
\end{array}
\end{equation}
of $\OO(X_i\setminus q_i)$, where we use the notation $f'(x,\tau):=\pa_x f(x,\tau)$  and $\dot{f}(x,\tau):=\pa_{\tau}f(x,\tau)$;
\begin{equation}\label{R-zeta1-eq}
R(x,\th,\tau,\eta)=\wp(x,\tau+\th\eta)=\wp(x,\tau)+\th\eta\cdot \dot{\wp}(x,\tau), \ \
\end{equation}
(see Sec.\ \ref{ell-sec} and \cite[Prop.\ 2.1]{P-superell}).
This implies that even global sections $\om_{X_i/S}$ on $X_i\setminus q_i$, for $i=1,2$, have form
$$s_i[\a_0^{(i)}+\sum_{n\ge 2}\a_n^{(i)}f_n(x_i,\tau_i)+\sum_{n\ge 1}a_n^{(i)}\psi_n(x_i,\tau_i)]$$
with $a_n^{(i)}$ even and $\a_n^{(i)}$ odd. 
Now, elements of $\pi_*\om_{X/S}$ are described by the data
\begin{equation}\label{glued-differential-form}
(\om_1,\om_2,s_1\phi_1(x_1,\th_1)+s_2\phi_2(x_2,\th_2)+s_0\phi_0)
\end{equation}
where $\om_i\in \om_{X_i/S}(X_i\setminus\{q_i\})$, $\phi_1$ and $\phi_2$ are formal series and $\phi_0$ is a function on the base,
subject to the equations
\eqref{om-phi-super-relation-eq}.

By Lemma \ref{cech-diff-lem}, the bundle $\ov{\HH}^1=R^1\pi_*[\OO_X\to \om_{X/S}]$ is calculated similarly to the case of the $(+,+)$ separating node divisor
as $\om_{X_1/S}(X_1-q_1)/\de\OO(X_1-q_1)\oplus \om_{X_2/S}(X_2-q_1)/\de\OO(X_2-q_2)$ (and this decomposition is compatible with the Gauss-Manin connections). 
Hence, by the result of \cite{P-superell},
we have the following horizontal $\OO_S$-basis of $\ov{\HH}^1$:
\begin{equation}\label{H1-basis-psi-eq}
e_i=s_i(\psi_1(x_i,\th_i)-\tau_i\psi_2(x_i,\th_i)), \ \ f_i=s_i\psi_2(x_i,\th_i), \ \ i=1,2,
\end{equation}
dual to the standard basis in homology given by the classes $(\a_i,\b_i)$.

\begin{prop}\label{--even-diff-prop}
The $\OO_S^+$-module $\pi_*\om_{X/S}^+$ is isomorphic to the submodule of $(\a,\b,a,b)\in (\OO_S^-)^{\oplus 2}\oplus (\OO_S^+)^{\oplus 2}$,
given by the equations
\begin{equation}\label{--module-diff-eq}
t\a=-(2\pi i)^{-1}\eta_2 b, \ \ t\b=(2\pi i)^{-1}\eta_1 a.
\end{equation}
The global differential \eqref{glued-differential-form} corresponding to $(\a,\b,a,b)$ has
$$\om_1=s_1\cdot [\a+a\psi_1+t^2(2\pi i)^{-1}b\psi_2]+O(t^3), \ \ \om_2=s_2[\b+b\psi_1+t^2(2\pi i)^{-1}a\psi_2]+O(t^3).$$
The coefficients of the higher terms of these expansions remain regular as $q_i=\exp(2\pi i \tau_i) \to 0$.
\end{prop}

\begin{proof}
We can solve equations \eqref{om-phi-super-relation-eq} modulo $t^n$ iteratively.
Modulo $t$ these equations give that $\om_1$ and $\om_2$ are regular, hence, we should have
$$\om_1=s_1\cdot(\a+a\psi_1(x_1,\tau_1)), \ \ \om_2=s_2\cdot(\b+b\psi_1(x_2,\tau_2)) \mod(t),$$
where $\eta_1 a=0$, $\eta_2 b=0$. Hence, $\phi_i$ are obtained as regular parts:
$$\phi_1=\a+a\th_1+O(x_1), \ \ \phi_2=\b+b\th_2+O(x_2).$$

Modulo $t^2$ we get by considering polar parts of \eqref{om-phi-super-relation-eq},
$$\om_1=s_1\cdot(\a+a\psi_1(x_1,\tau_1)),$$
$$\om_2=s_2\cdot(\b+b\psi_1(x_2,\tau_2)) \mod(t^2),$$
where $\a,a,\b,b$ satisfy \eqref{--module-diff-eq} modulo $t^2$, and the formulas
for $\phi_i$ are still the same.

Modulo $t^3$ we get from the polar part of \eqref{om-phi-super-relation-eq}
$$\om_1=s_1\cdot(\a+a\psi_1(x_1,\tau_1)+(2\pi i)t^2b\psi_2(x_1,\tau_1)),$$ 
$$\om_2=s_2\cdot(\b+b\psi_1(x_2,\tau_2)+(2\pi i)t^2a\psi_2(x_2,\tau_2)) \mod(t^3),$$
where $\a,a,\b,b$ satisfy \eqref{--module-diff-eq} modulo $t^3$.

We can continue the same process modulo all powers of $t$, and this leads to the claimed assertion.
\end{proof}


Thus, in the punctured formal neighborhood of the $(-,-)$ separating node divisor
we have a normalized basis of global differentials $\om^{(1)}, \om^{(2)}$ (with normalized $\a$-periods):
\begin{equation}\label{--norm-diff-eq}
\begin{array}{l}
\om^{(1)}=(s_1\psi_1(x_1,\tau_1), s_2[(2\pi i)^{-1}\frac{\eta_1}{t}+(2\pi i)t^2\psi_2(x_2,\tau_2)])+O(t^3),\\
\om^{(2)}=(s_1[-(2\pi i)^{-1}\frac{\eta_2}{t}+(2\pi i)t^2\psi_2(x_1,\tau_1)], s_2\psi_1(x_2,\tau_2))+O(t^3).
\end{array}
\end{equation}
Now using the relations in cohomology $s_i\equiv f_i\eta_i$ for $i=1,2$ (see \cite[Cor.\ 3.8]{P-superell}),
we can decompose these differentials with respect to the basis \eqref{H1-basis-psi-eq} to get a superperiod map.

\begin{cor}\label{--cor}
(i) One has in $\ov{\HH}^1_{X/S}$,
\begin{align*}
&\om^{(1)}\equiv e_1+\tau_1f_1+[-(2\pi i)^{-1}\frac{\eta_1\eta_2}{t}+(2\pi i)t^2]\cdot f_2 \mod(t^3), \\
&\om^{(1)}\equiv e_2+\tau_2f_2+[-(2\pi i)^{-1}\frac{\eta_1\eta_2}{t}+(2\pi i)t^2]\cdot f_1 \mod(t^3),
\end{align*}
where all the higher terms of the expansion are regular as $q_i\to 0$. In other words, the superperiod matrix has form
$$\Om=\left(\begin{matrix} \tau_1 & -(2\pi i)^{-1}\frac{\eta_1\eta_2}{t}+(2\pi i)t^2 \\ 
-(2\pi i)^{-1}\frac{\eta_1\eta_2}{t}+(2\pi i)t^2 & \tau_2\end{matrix}\right) + O(t^3).$$

\noindent
(ii) The canonical projection near the $(-,-)$ separating node divisor satisfies
$$\pi^*t=t-(2\pi i)^{-2}\frac{\eta_1\eta_2}{2t^2}+O(t^2),$$
where the coefficients of the higher order terms are regular as $q_i\to 0$.
\end{cor}


\section{Canonical projection for genus $2$ moduli of supercurves}\label{can-proj-sec}

\subsection{Torelli map for stable supercurves of genus $2$}


We want to show that the Torelli map $\SS_2\to \MM_2$
extends to a regular map to $\ov{\MM}_2$ away from $(-,-)$ separating boundary divisor $D^{-,-}\sub \ov{\SS}_2$.


\begin{lemma}\label{reg-crit-lem}
Let $X$ be a smooth superscheme, $D=D_1\cup D_2\sub X$ an effective normal crossing Cartier divisor, $U:=X\setminus D$. Suppose we have a morphism $f:U\to Y\setminus \De$, where $Y$ is a smooth scheme, $\De$ a normal crossing divisor in $Y$. Let also $n_1,n_2$ be natural numbers.
Assume that locally near every point $x\in D$ there exist coordinates $((y_i),(z_1,\ldots,z_k),(t_1,\ldots,t_l))$ in a neighborhood of $f(x)$ in $Y$, such that the pullbacks $f^*y_i$ extend regularly over $D$, while 
there exist local equations $\wt{z}_1\ldots\wt{z}_k=0$ (resp., $\wt{t}_1\ldots\wt{t}_l=0$) of $D_1$ (resp., $D_2$) in a neighborhood $U$ of $x$, such that there is an equality of multivalued functions on $U-D$,
$$f^*\log(z_i)=n_1\log(\wt{z}_i)+\phi_i, \ \ f^*\log(t_j)=n_2\log(\wt{t}_j)+\psi_j,$$ 
where $\phi_i$ and $\psi_j$ extend regularly to $U$.
Then $f$ extends to a regular morphism $X\to Y$ such that $f^*\De=n_1D_1+n_2D_2$. 
\end{lemma}

\begin{proof} 
The assertion follows immediately from the fact that $f^*z_i$ (resp., $f^*t_j$) differs from $\wt{z}_i^{n_1}$ (resp., $\wt{t}_j^{n_2}$) by an invertible function near $x$.
\end{proof}


%

Let us denote by $D^{ns,R}$ and $D^{ns,NS}$ the Ramond and Neveu-Schwartz components of the non-separating node boundary divisor $D^{ns}\sub \ov{\SS}_2$,
and let $\De^{ns}\sub\ov{\MM}_2$ denote the non-separating boundary divisor.

\begin{prop}\label{ext-Tor-prop} 
The Torelli map $\per:\SS_2\to \MM_2\to \ov{\MM}_2$ extends to a regular morphism $\per:\ov{\SS}_2\setminus D^{-,-}\to \ov{\MM}_2$,
such that $\per^*\De^{ns}=D^{ns,R}+2D^{ns,NS}$.
\end{prop}

\begin{proof}  The proof is based on the analysis of the superperiod matrix near the boundary of the moduli space, mirroring the similar analysis in the classical case
(see e.g., \cite{Nam}, \cite[ch.\ III]{Fay}).

First, let us study the situation near a point $X_0$ with the underlying stable curve $C_0$
of the separating node boundary component of type $(+,+)$, which does not belong to deeper strata (i.e., $C_0=C_1\cup C_2$ where $C_1$ and $C_2$ are smooth curves if genus $1$). 
Note that the usual Torelli map
$\MM_2\to \AA_2$ to the moduli space of principally polarized abelian varieties is regular near $C_0$
and locally near $C_0$ we have an isomorphism $\ov{\MM}_2\simeq \AA_2$ (the nodal union of two curves of genus $1$ maps to their product viewed as an abelian variety). 
On the other hand, the local system $R^1\pi_*(\C_{X/S})$ extends regularly to our boundary component and the map
$$\pi_*(\om_{X/S})\to R^1\pi_*(\C_{X/S})$$ 
is an embedding of a subbundle of rank $g$. Indeed, by Lemma \ref{coh-van-spin-str-g2-lem}, using \cite[Prop.\ 3.2]{FKP-per} we see that $\pi_*(\om_{X/S})$ is a vector bundle of rank $g$,
with the fiber over $s$ identified with $H^0(X_s,\om_{X_s})\simeq H^0(C_s,\om_{C_s})$, where $C_s$ is the underlying usual curve. The induced map of the fibers
$H^0(C_s,\om_{C_s})\to H^1(C_s,\C)$ is given by the usual period matrix, so it is an embedding. This implies our claim. 
Thus, we see that the period matrix is regular near the generic point of the boundary component of type $(+,+)$, so the Torelli map $\per:\SS_2\to \ov{\MM}_2$ is also regular
at the generic point of this component.

Next, let us consider the situation at deeper strata of the boundary separating node divisor of type $(+,+)$.
Assume first the underlying stable curve is $C_0=C_1\cup C_2$, where $C_1$ is an irreducible nodal curve of genus $1$ and $C_2$ is smooth.
Let $\Om$ (resp., $\Om^{\bos}$) be the superperiod matrix (resp., usual period matrix) near $C_0$.
Then the local functions
\begin{equation}\label{three-per-entries-eq}
\exp(2\pi i\Om^{\bos}_{11}), \ \ \Om^{\bos}_{12}, \ \ \Om^{\bos}_{22}
\end{equation} 
are local coordinates near $C_0$ on $\ov{\MM}_2$ (see \ref{superperiods-deg-sec}).
The pull-backs of these functions under the Torelli map are $\exp(2\pi\Om_{11})$, $\Om_{12}$ and $\Om_{22}$, respectively.
By Theorem \ref{regular-norm-diff-thm}, the latter functions are regular near $X_0$ and $\exp(2\pi\Om_{11})=ut^a$, where $t$ is a local equation of $D^{ns}$, $a=1$ for
the Ramond branch and $a=2$ for the NS branch. 
Hence, by Lemma \ref{reg-crit-lem}, the Torelli map is regular near $X_0$, and we have the claimed relation for the divisors.

The case when $C_0=C_1\cup C_2$, where both $C_1$ and $C_2$ are nodal is proved similarly using Theorem \ref{regular-norm-diff-thm} and the
results of \ref{superperiods-deg-sec}: in this case we use
the local coordinates
\begin{equation}\label{three-per-entries-bis-eq}
\exp(2\pi i\Om^{\bos}_{11}), \ \ \Om^{\bos}_{12}, \ \ \exp(2\pi i\Om^{\bos}_{22})
\end{equation} 
on $\ov{\MM}_2$.

Next, let us consider the situation near an irreducible curve $C_0$ with one or two nodes. In this case the behavior of the superperiod matrix $\Om$ is still described locally by Theorem
\ref{regular-norm-diff-thm}(ii), and the similar description holds for the usual period matrix $\Om^{\bos}$. In particular, in the case of a curve with one node (resp., two nodes) 
the function $\exp(2\pi i\Om^{\bos}_{11})$ (resp., $\exp(2\pi i\Om^{\bos}_{11})$ and $\exp(2\pi i\Om^{\bos}_{22})$) is a local equation of a branch of
$\De$ (resp., of two branches of $\De$) near $C_0$.  
Furthermore, in the case of a curve with one node (resp., two nodes), the functions
\eqref{three-per-entries-eq} (resp., \eqref{three-per-entries-bis-eq}) 
are still local coordinates on $\ov{\MM}_2$. Indeed, this follows easily from the form of the period matrix near such degenerations established in \cite{Fay} (see
\cite[Cor.\ 3.8]{Fay} and an example on p.54).
Thus, we can prove the regularity and the relation between the divisors in the same way as above.



In the remaining case when $C_0$ is the union of two $\P^1$ glued at three nodal points, the local behavior of $\Om$ and $\Om^{\bos}$ is
described in Lemma \ref{tough-curve-per-lem}. In particular, the functions
$$\exp(2\pi i(\Om^{\bos}_{11}+\Om^{\bos}_{12})), \ \ \exp(-2\pi i\Om^{\bos}_{12}), \ \ \exp(2\pi i(\Om^{\bos}_{12}+\Om^{\bos}_{22}))$$
are the equations of three branches of $\De$, which give local coordinates on $\ov{\MM}_2$. Thus, the assertion follows as before.
\end{proof}

\subsection{Regularity of the canonical projection}

Recall that in Corollary \ref{g2-proj-cor} we described a canonical projection $\pi^{\can}:\SS_2\to \SS_{2,\bos}$. 
Since $\SS_{2,\bos}$ is affine, there exists an infinite-dimensional space of 
projections $\SS_2\to \SS_{2,\bos}$. We can view each such a projection as a rational map from $\ov{\SS}_2$ to $\ov{\SS}_{2,\bos}$, so it make sense to impose
regularity conditions near generic points of components of the boundary divisor. We will now prove that $\pi^{\can}$ extends to a projection of $\ov{\SS}_2\setminus D^{-,-}$,
and that even a certain weaker regularity along boundary components
characterizes the canonical projection of $\SS_2$ uniquely.

\begin{theorem}\label{projection-nonsep-prop} 
(i) The canonical projection $\pi^{\can}$ extends to a regular projection of $\ov{\SS}_2\setminus D^{-,-}$, which we still denote by $\pi^{\can}$.
One has $(\pi^{\can})^*(D^{ns}_{\bos})=D^{ns}$.
Furthermore, in terms of the gluing coordinates $y_1,y_2,t,\eta_1,\eta_2$ in a formal neighborhood of a generic point of the $(+,+)$ separating node divisor,
the canonical projection is given by the functions $y_1+O(t^3)\eta_1\eta_2$, $y_2+O(t^3)\eta_1\eta_2$, $t+\eta_1\eta_2/2+O(t^4)\eta_1\eta_2$.

\noindent
(ii) The canonical projection $\SS_2\to \SS_{2,\bos}$ is the unique projection of $\SS_2$, which is regular at a generic point of each boundary divisor corresponding 
to a non-separating node.

\noindent
(iii) The moduli stack $\ov{\SS}_2$ is not projected.
\end{theorem}


\begin{proof} (i) By Lemma \ref{coh-van-spin-str-g2-lem}, 
we need to show that the morphism $\pi^{\can}:\SS_2\to \ov{\SS}_{2,\bos}$ has a regular extension near every stable spin curve $(C,L)$
such that the corresponding spin structure has no global sections.

First, assume that $C$ is a stable curve with one non-separating node. 
Let us denote by $U$ a small \'etale neighborhood of $(C,L)$ in $\ov{\SS}_2$. 
We will use a regular extension of the Torelli map to a morphism $U\to \ov{\MM}_2$ (see Proposition \ref{ext-Tor-prop}). 
Note that near $(C,L)$ the map $\ov{\SS}_{2,\bos}\to \ov{\MM}_2$ is either \'etale,
or is a double covering ramified at the boundary divisor. Thus, in the former case we get a unique lifting $U\to \ov{\SS}_{2,\bos}$ which is identical on the reduced space.

In the latter case, $U$ is described by a superalgebra $A$, $\ov{\MM}_2$ is described by an algebra $B$.
We have an embedding $B\to A_{\bos}=A/\NN$ such that $A_{\bos}$ is generated over $B$ by one element $x$ such that $x^2=f$, 
where $f\in B$ is an equation of the boundary divisor.
In addition we have a commutative diagram 
\begin{diagram}
B&\rTo{}& A_{\bos}\\
\dTo{s}&&\dTo{s}\\
A&\rTo{}& A[\wt{x}^{-1}]
\end{diagram}
where $\wt{x}\in A$ is any lifting of $x$),
identical modulo $\NN$, where the vertical arrows reduce to the obvious embeddings modulo $\NN$.
It follows that $s(A_{\bos})$ is generated over $A$ by one even element $y=s(x)$ such that $y^2=s(f)$. One more piece of information we get from
Proposition \ref{ext-Tor-prop} is that $s(f)=ut^2$, where $u\in A^*$ and $t=0$ is the equation of the boundary such that $x=t\mod \NN$ (here we use that the node is non-separating).

Consider a local splitting of $A$, $A=A_{\bos}[\th_1,\th_2]$, such that $t\in A_{\bos}\sub A$. Then we have $y=t+a\th_1\th_2$, where $a\in A_{\bos}[t^{-1}]$. We then get
$$y^2=t^2+2ta\th_1\th_2=s(f)=t^2(u_0+b\th_1\th_2).$$
which implies that $2a=tb\in tA_{\bos}$. So $y\in tA$, hence, $s(A_{\bos})\sub A$, which gives the required regularity and the compatibility with the boundary divisors near $(C,L)$.

In the case when $C$ has more than one non-separating node and no separating nodes, the argument is very similar, except that we need to deal with possibly two or three branches
of the non-separating boundary divisor, where the projection $\ov{\SS}_{2,\bos}\to \ov{\MM}_2$ ramifies. 
We have their local equations $f_i\in B$ and $A_{\bos}$ is generated over $B$ by some even elements $x_i$ such that $x_i^2=f_i$. The rest of the argument is the same as above
using Proposition \ref{ext-Tor-prop}.

Now let us consider the case when $(C,L)$ belongs to the $(+,+)$ separating node component.
Recall that a separating node component is necessarily of NS type, and the corresponding components have genus $1$.
Let $t,y_1,y_2,\eta_1,\eta_2$ be formal local coordinates on $\ov{\SS}_2$ near the corresponding stable supercurve.
The projection $\pi^{\can}:\SS_2\to \SS_{2,\bos}$ in these coordinates takes form
$$t\mapsto t+f(t,y_1,y_2)\eta_1\eta_2, \ \ y_1\mapsto y_1+g_1(t,y_1,y_2)\eta_1\eta_2, \ \ y_2\mapsto y_2+g_2(t,y_1,y_2)\eta_1\eta_2.$$
We want to prove that $f$, $g_1$ and $g_2$ do not have pole at $t=0$.


The projection $\SS_{2,\bos}\to \ov{\MM}_2$ in our local coordinates has form $(t,y_1,y_2)\mapsto (t^2,y_1,y_2)$.
Thus, from the regularity of the Torelli map $\ov{\SS}_2\setminus D^{-,-}\to \ov{\MM}_2$ we deduce that $g_1$ and $g_2$ do not have a pole at $t=0$, and that 
$$\pi^*(t^2)=(t+f\eta_1\eta_2)^2=t^2+2tf\eta_1\eta_2.$$ 
Thus, it is enough to check that the coefficient of $\eta_1\eta_2$ in $(\pi^{\can})^*(t^2)$ is divisible by $t$.
But this follows from our calculation in Lemma \ref{super-global-diff-base-lem}.
The last assertion also follows from this calculation.

\noindent
(ii) 
Let $\NN$ denote the ideal generated by odd functions on $\SS_2$.
The difference between two projections is a global $\NN^2$-valued derivation on $\SS_{2,\bos}$, i.e., a global section of $\TT\ot \NN^2$, where $\TT$ is the tangent sheaf. 

We will use the hyperelliptic picture of Section \ref{hyperell-basics-section}, so we present $C$ as a double covering of $\P^1$ 
ramified at $u_1,u_2,u_3,v_1,v_2,v_3\in \A^1\sub \P^1$. In this picture we have natural trivializations of both $\TT$ and 
$\NN^2$. Namely, the fiber of $\TT$ at $C$ is dual to $H^0(C,\om_C^2)$. Let us denote by $(\de_0,\de_1,\de_2)$ the basis of $\TT$, dual to the basis 
\eqref{qu-diff-basis-eq}.
On the other hand, $\NN^2$ is trivialized by $\eta_1\eta_2$, where $(\eta_1,\eta_2)$ correspond to the basis 
$(\chi_1,\chi_2)$ of $H^0(C,\om_C\ot L)$
(see \eqref{chi1-chi2-eq}).
Thus, a global $\NN^2$-valued vector field can be written in the form
$$X=(f_0(u,v)\de_0+f_1(u,v)\de_1+f_2(u,v)\de_2)\eta_1\eta_2,$$
and this expression should be invariant with respect to all symmetries.

Note that the denominator in each $f_i$ is necessarily a product of some powers of $(u_i-u_j)$ and $(u_i-v_j)$. Since we assumed that no poles can appear as $u_i\to u_j$ or $u_i\to v_j$,
it follows that each $f_i$ is a polynomial.

Let us consider the invariance with respect to rescaling $u\mapsto \la u$, $v\mapsto \la v$, $x\mapsto \la x$, $y\mapsto \la^3 y$.
We have $\eta_1\eta_2\mapsto \la^{-3}\eta_1\eta_2$ and $dx/y\mapsto \la^{-2} dx/y$. Hence,
$$\de_0\mapsto \la^4\de_0, \ \de_1\mapsto \la^4\de_1, \ \de_2\mapsto \la^4\de_2.$$
It follows that 
$$\deg(f_0)=-1, \ \ \deg(f_1)=0, \ \ \deg(f_2)=1.$$
It follows that $f_0=0$ and $f_1$ is a constant.

Next, let us consider the invariance with respect to $u_i\mapsto 1/u_i$, $v_i\mapsto 1/v_i$, $x\mapsto 1/x$, $y\mapsto \pm \prod_i(u_iv_i)^{-1/2}\cdot y/x^3$.
Then we have $\eta_1\eta_2\mapsto \pm\eta_1\eta_2\prod_i u_iv_i$,
$$\de_0\eta_1\eta_2\mapsto \pm \de_2\eta_1\eta_2, \ \ \de_2\eta_1\eta_2\mapsto \pm \de_0\eta_1\eta_2, \ \ \de_1\eta_1\eta_2\mapsto \pm\de_1\eta_1\eta_2$$
Hence, we should have
$$f_2(u_i,v_i)=\pm f_0(1/u_i,1/v_i).$$
Therefore, $f_2=0$.

Finally, the symmetry swapping $u_i$ with $v_i$ acts by $-1$ on $\eta_1\eta_2$. Hence, $f_1$ should go to $-f_1$ under this transformation. But $f_1$ is constant, so $f_1=0$.

\noindent
(iii) This follows from (ii) and from the fact that the canonical projection is not regular near $D^{-,-}$ (see Corollary \ref{--cor}(ii)).
\end{proof}


%





\section{The Mumford form for genus $2$}\label{hyper-Mum-sec}

\subsection{From the even gluing construction to hyperelliptic covering}

Recall that for a smooth curve of genus $2$ the hyperelliptic covering $C\to \P^1$ is given by the linear system associated with the canonical linear system.
Thus, to find $6$ ramification points on $\P^1$, we have to find $6$ global differentials on $C$ with double zeros.

Now let us consider the degenerating family of curves of genus $2$, $\pi:C\to B$
over $B=\MM_{1,1}^{(1)}\times\MM_{1,1}^{(1)}\times S^{(n)}_q$, where $S^{(n)}=\Spec \C[q]/q^{n+1}$,
obtained by the gluing construction of Sec. \ref{even-gluing-sec}
(where we change the coordinate from $t$ to $q$). 
Let us use the basis $\om^{(1)}$, $\om^{(2)}$ of $\pi_*\om_{C/S}$ constructed in Sec.\ \ref{even-gluing-sec}.

\begin{prop}\label{ram-points-prop} 
Let $\a_1,\a_2,\a_3,\b_1,\b_2,\b_3:S\to C$ denote constant sections in 
$C_1\setminus \{q_1\}\sqcup C_2\setminus\{q_2\}$ corresponding to $3$ nontrivial points of order $2$ on $C_1$ and $3$ nontrivial points of order 
$2$ on $C_2$.
Then there exist $6$ sections of $\pi_*\om_{C/S}$ of the form
$$\eta^{(1)}_i:=qf(\a_i,\tau_1)\om^{(1)}+\om^{(2)}, \ \ \eta^{(2)}_i:=\om^{(1)}+qf(\b_i,\tau_2)\om^{(2)}, \ i=1,2,3,$$ 
such that 
$\eta^{(1)}_i$ (resp., $\eta^{(2)}_i$) has a double zero along $\a_i$ (resp., $\b_i$).
Furthermore, one has
$$f(u,\tau_i)(q)=\wp(u,\tau_i)\mod(q^4).$$
\end{prop}

\begin{proof} 
As we have seen in Sec.\ \ref{even-gluing-sec}, one has
$$\om^{(1)}_1=(1+q^4\phi(x_1,\tau_1)(q))dx_1, \ \ \om^{(2)}_1=(-q\wp(x_1,\tau_1)+q^4\psi(x_1,\tau_1)(q))dx_1,$$
where $\phi$ and $\psi$ are {\it even} elliptic functions of $x_1$ depending on $q$, with poles only at lattice points.
Thus, if $x_1=u$ is a nontrivial point of order $2$ then $\phi'(u,\tau_1)=\psi'(u,\tau_1)=0$. 
We can find a function $f(u)=f(u)(q)$ such that
$$qf(u)\om^{(1)}_1+\om^{(2)}_1=q[f(u)(1+q^4\phi(x_1,\tau_1)(q))-\wp(x_1,\tau_1)+q^3\psi(x_1,\tau_1)(q)]dx_1$$
vanishes at $x_1=u$. Namely, we can just set
$$f(u)=(\wp(u,\tau_1)-q^3\psi(u,\tau_1)(q))(1+q^4\phi(u,\tau_1)(q))^{-1}.$$
Then $qf(u)\om^{(1)}_1+\om^{(2)}_1$ automatically has a double zero at $x_1=u$. Applying this to $u=\a_i$ gives $\eta_i^{(1)}$.
Reversing the roles of $C_1$ and $C_2$ we get $\eta_i^{(2)}$.
\end{proof}

Note that if we use the rescaled basis 
$$\ov{\om}^{(1)}=q\om^{(1)}, \ \ \ov{\om}^{(2)}=\om^{(2)}$$
our $6$ global differentials with double zeros take form $\ov{\om}^{(1)}-a_i\ov{\om}^{(2)}$, $\ov{\om}^{(1)}-b_i\ov{\om}^{(2)}$, $i=1,2,3$,
where
\begin{equation}\label{ramification-points-gluing-formulas-eq}
a_i=-\frac{1}{f(\a_i)}= -\frac{1}{\wp(\a_i,\tau_1)}+O(q^4),\ \ b_i=-q^2f(\b_i)=-q^2(\wp(\b_i,\tau_2)+O(q^4)),
\end{equation}
(where we assume $\wp(\a_i,\tau_1)\neq 0$).

If we take the projective limit of the above construction over all $n$ and then invert $t$, we obtain a family of curves $C_{t\neq 0}$ such that the ring of functions on
the affine curve $C_{t\neq 0}\setminus Z(\om^{(2)})$ will be generated by functions $x$ and $y$, where
$$\ov{\om}^{(1)}=x\ov{\om}^{(2)},$$
and $y(\ov{\om}^{(2)})^{\ot 3}$ is a regular section of $\om_{C_{t\neq 0}/S_{t\neq 0}}^{\ot 3}$, so that 
\begin{equation}\label{hyperell-ab-eq}
y^2=\prod_i(x-a_i)(x-b_i).
\end{equation}

\begin{cor}\label{ram-points-cor}
The curve $C_{t\neq 0}$ is isomorphic to the ramified covering of $\P^1$ ramified at points $a_1,a_2,a_3,b_1,b_2,b_3$ given by \eqref{ramification-points-gluing-formulas-eq}.
\end{cor}

In terms of the gluing we have
$$x=(-\frac{1}{\wp(x_1)}+O(q^4),-q^2\wp(x_2)+O(q^6)),$$
We also have
$$\frac{dx}{y}=c\cdot \om^{(2)},$$
for some constant $c$ depending on $q$ (and the curves $C_1$, $C_2$).
We get
$$cy=(-q^{-1}(\frac{\wp'(x_1)}{\wp(x_1)^3}+O(q^4)),-q^2(\wp'(x_2)+O(q^4))).$$
To find $c$ we consider expansions of both sides of \eqref{hyperell-ab-eq} at $x_1=\a_1$.
For the left-hand side we get
$$c^{-2}\cdot q^{-2}(\frac{\wp''(\a_1)^2}{\wp(\a_1)^6}+O(q^4))(x_1-\a_1)^2+\ldots$$
For the right-hand side, using the expansion $\wp(x_1)-\wp(\a_1)=\frac{1}{2}\wp''(\a_1)(x-\a_1)^2+\ldots$,  we get
$$-[\frac{\wp''(\a_1)(\wp(\a_1)-\wp(\a_2))(\wp(\a_1)-\wp(\a_3))}{2\wp(\a_1)^7\wp(\a_2)\wp(\a_3)}+O(q^2)](x_1-\a_1)^2+\ldots$$
Hence,
$$c^2=-q^{-2}(\frac{2\wp''(\a_1)\wp(\a_1)\wp(\a_2)\wp(\a_3)}{\wp(\a_1)-\wp(\a_2)(\wp(\a_1)-\wp(\a_2))}+O(q^2))=
-q^{-2}(g_3(\tau_1)+O(q^2)),$$
where we use the relations
$$\wp''(\a_1)=2(\wp(\a_1)-\wp(\a_2))(\wp(\a_1)-\wp(\a_3)), \ \ 4\wp(\a_1,\tau_1)\wp(\a_2,\tau_1)\wp(\a_3,\tau_1)=g_3(\tau_1).$$
Thus,
\begin{equation}\label{c-eq}
c=q^{-1}(\sqrt{-g_3(\tau_1)}+O(q^2)).
\end{equation}

Now we can compare trivializations of $\det H^0(C,\om_C)$ obtained from the bases $(\om^{(1)},\om^{(2)})$ and
$(\frac{xdx}{y},\frac{dx}{y})$:
\begin{equation}\label{lambda-basis-gluing-eq}
\frac{xdx}{y}\we \frac{dx}{y}=c^2\cdot \ov{\om}^{(1)}\we \om^{(2)}=-q^{-2}(g_3(\tau_1)+O(q^2))\cdot \om^{(1)}\we \om^{(2)}.
\end{equation}

\subsection{Gluing near the separating node divisor and the spin structures}\label{super-glue-spin-sec}

\subsubsection{The $(+,+)$ separating node gluing}\label{spin++sec}

Here we consider the gluing construction near the $(+,+)$ separating node divisor in $\ov{\SS}_2$ (see Sec.\ \ref{super-gluing-sec} and 
Sec.\ \ref{formal-nbhd-div-sec}) and identify the corresponding hyperelliptic picture.

\begin{lemma}\label{gluing-spin-lem}
Consider the gluing construction restricted to the reduced base $B=\ov{\SS}'_{1,1,\bos}\times \SS'_{1,1\bos}\times S^{(2n+1)}_t$,
with the choice of the spin-structure on $C_1$ and $C_2$, corresponding to points of order $2$, $\a\in C_1$ and $\b\in C_2$.
Let $(\a_1=\a,\a_2,\a_3)$ (resp., $(\b_1=\b,\b_2,\b_3)$) denote all nontrivial points of order $2$ on $C_1$ (resp., $C_2$).
Then the corresponding spin-structure $L$ on the smooth part of the glued curve $C$ is isomorphic to 
$\OO_C(\b_2+\b_3-\a_1)\simeq \OO_C(\a_2+\a_3-\b_1)$.
Furthermore, one has sections $\bs_1\in L(\a_1-\b_2-\b_3)$, $\bs_2\in L(\b_1-\a_2-\a_3)$ given by 
$$\bs_1=((\frac{1}{\kappa_1(x_1)}+t^4\frac{\wp(\b_1)}{2}\kappa_1(x_1)+O(t^7))\th_1,  (t^3\kappa_2(x_2)+O(t^7))\th_2),$$
$$\bs_2=((-t^3\kappa_2(x_1)+O(t^7))\th_1, (\frac{1}{\kappa_1(x_2)}+t^4\frac{\wp(\a_1)}{2}\kappa_1(x_2)+O(t^7))\th_2).$$
Also, one has
$$\bs_1^2=-q^{-1}(\wp(\a_1)^{-1}+O(q^2))\cdot\frac{(x+q^2f(\b_2))(x+q^2f(\b_3))}{x+f(\a_1)^{-1}}\cdot \om^{(2)}|_C.$$
$$\bs_2^2=q^{2}(-\wp(\a_2)\wp(\a_3)+O(q^2))\cdot \frac{(x+f(\a_2)^{-1})(x+f(\a_3)^{-1})}{x+q^2f(\b_1)}\cdot \om^{(2)}|_C.$$
\end{lemma}

\begin{proof} We will focus on the statements concerning $\bs_1$ (those for $\bs_2$ are proved similarly).
Since the spin-structure is given by the odd part of the structure sheaf of the supercurve, it is enough to construct a rational odd function on 
the glued supercurve with the only pole of order $1$ at $\a_1\in C_1\setminus\{q_1\}$.
Such an odd function is given by
$$\bs_1:=(F_1\th_1,F_2\th_2,\phi_1(x_1)\th_1+\phi_2(x_2)\th_2),$$
where $\phi_i(x_i)$ are power series,
$F_1(x_1)\th_1$ (resp., $F_2(x_2)\th_2$) is a section of the spin structure $L_1\simeq \OO(q_1-\a_1)$ on $C_1-q_1$ (resp.,
$L_2\simeq \OO(q_2-\b_1)$ on $C_2-q_2$)
with a pole of order $1$ at $\a_1$ (resp., $\b_1$), satisfying
\begin{equation}\label{odd-functions-eq}
F_1(x_1)=\phi_1(x_1)+\frac{t}{x_1}\phi_2(-\frac{t^2}{x_1}), \ \ 
F_2(x_2)=\phi_2(x_2)-\frac{t}{x_2}\phi_1(-\frac{t^2}{x_2}).
\end{equation}
We can take 
$$F_1\equiv \frac{1}{\kappa_1(x_1)} \mod (t), \ \ F_2\equiv 0 \mod (t),$$
so that $\phi_1=F_1\mod(t)$, $\phi_2=0\mod(t)$. Then one can check that
$$F_1\equiv \frac{1}{\kappa_1(x_1)}+t^4\frac{\wp(\b_1)}{2}\kappa_1(x_1) \mod (t^7), \ \ F_2\equiv t^3\kappa_2(x_2) \mod (t^7)$$
and that only $\kappa_n(x_2)$ with even $n$ occur in the expansion of $F_2$.

Now we claim that $\kappa_{2k}(x_2)$
vanishes at $\b_2,\b_3$.
Indeed, by Lemma \ref{kappa-functions-lem}(iii), one has $\kappa_{2k}(x_2)=\phi(x_2)\kappa_1(x_2)$, where $\phi$ is an elliptic function with no poles away from $\b_1$ and $q_2$
and such that $\phi(-x)=-\phi(x)$. This implies that $\phi$ vanishes at $\b_2,\b_3$.

The identity for $\bs_1^2$ follows from the fact that it is the unique (up to a constant depending of $q$) section of $\om_{C/B}$ 
with the pole of order $2$ at $\a_1$ and double zeros at $\b_2$ and $\b_3$.
\end{proof}

Thus, we can complement the result of Corollary \ref{ram-points-cor} identifying the ramification points $a_1,a_2,a_3,b_1,b_2,b_3$ associated with the
glued curve, by identifying the spin-structure coming from the choice of order $2$ points $\a_1\in C_1$, $\b_1\in C_2$.

\begin{cor}\label{glued-spin-partition-cor}
Under the identification of the glued spin-curve $C$ with the the double covering of $\P^1$ ramified at the points $a_1,a_2,a_3,b_1,b_2,b_3$ given by 
$$a_i=-f(\a_i)^{-1}, \ b_i=-q^2f(\b_i), \ i=1,2,3,$$
the spin-structure on $C$ corresponds to the partition 
$$(a_1,b_2,b_3), \ (b_1,a_2,a_3)$$
of the ramification locus.
\end{cor}


Now we can use the setting of Section \ref{hyperell-basics-section} with
\begin{equation}\label{glued-partition-corr-eq}
u_1=a_1, u_2=b_2, u_3=b_3, \ v_1=b_1, v_2=a_2, v_3=a_3.
\end{equation}
Under this identification the rational sections $\bs,\bs'$ of $L$ satisfying
$$\bs^2=\frac{(x-b_2)(x-b_3)}{x-a_1}\frac{dx}{y}, \ \ (\bs')^2=\frac{(x-a_2)(x-a_3)}{x-b_1}\frac{dx}{y},$$
are given by
\begin{equation}\label{s-s'-s1-s2-eq}
\bs=t(\wp(\a_1)^{1/2}+O(t^4))\cdot c^{1/2}\cdot \bs_1, \ \ 
\bs'=-t^{-2}((-\wp(\a_2)\wp(\a_3))^{-1/2}+O(t^4))\cdot c^{1/2}\cdot \bs_2.
\end{equation}
Here we used the equation $dx/y=c\cdot \om^{(2)}|_C$, where $c$ is given by \eqref{c-eq} and 
$$\om^{(2)}|_C=(t^2\wp(x_1)dx_1,dx_2)+O(t^5).$$

Recall that there is a basis $(\chi_1,\chi_2)$ of $H^0(C,\om_{C/B}\ot L)$ over the smooth part, given by \eqref{chi1-chi2-eq}.

\begin{lemma}\label{chi1-chi2-(++)-lem} 
The rescaled basis of $H^0(C,\om_{C/B}\ot L)$,
\begin{align*}
&\chi_1=\wp(\a_1)^{-1/2}(-g_3(\tau_1))^{3/4}\cdot (\kappa_1(x_1)[dx_1|d\th_1]^3,t\kappa_2(x_2)[dx_2|d\th_2]^3) +O(t^3), \\ 
&t\cdot \chi_2=\\
&(-\wp(\a_2)\wp(\a_3))^{-1/2}(-g_3(\tau_1))^{3/4}\cdot(-t\kappa_2(x_2)[dx_1|d\th_1]^3,\kappa_1(x_2)[dx_2|d\th_2]^3) +O(t^3),
\end{align*}
is regular away from the locus where $g_3(\tau_1)=0$.
\end{lemma}

\begin{proof}
Using \eqref{s-s'-s1-s2-eq} and \eqref{c-eq}, we see that
$$\bs=\varphi_1\cdot \bs_1, \ \ \bs'=t^{-3}\cdot \varphi_2\cdot \bs_2,$$
where $\varphi_1$ and $\varphi_2$ are regular.
Hence,
$$\chi_1=(x-a_1)\cdot \bs\cdot \frac{dx}{y}=(x-a_1)\cdot\varphi_1\cdot (c\cdot \bs_1\cdot \om^{(2)}),$$
$$t\chi_2=t\cdot (x-b_1)\cdot \bs'\cdot \frac{dx}{y}=(x-b_1)\cdot\varphi_2\cdot (t^{-2}\cdot c\cdot \bs_2\cdot\om^{(2)}).$$
Thus, we need to check that $t^{-2} (x-a_1)\bs_1\om^{(2)}$ and $t^{-4}(x-b_1)\bs_2\om^{(2)}$ are regular.
But this follows immediately from the formulas
$$x=(\frac{-1}{\wp(x_1)}+O(q^4),-q^2\wp(x_2)+O(q^6)), \ a_1=\frac{-1}{\wp(\a_1)}+O(q^4), \ b_1=-q^2\wp(\b_1)+O(q^6),$$
$$\bs_1\om^{(2)}=(t^2(\frac{\wp(x_1)}{\kappa_1(x_1)}+O(t^3))\th_1 dx_1, t^3(\kappa_2(x_2)+O(t^4))\th_2 dx_2),$$
$$\bs_2\om^{(2)}=(t^5(-\kappa_2(x_1)\wp(x_1)+O(t^3))\th_1 dx_1,(\frac{1}{\kappa_1(x_2)}+O(t^4))\th_2 dx_2).$$
\end{proof}

\begin{remark}
Note that for a singular curve with a separating node $C=C_1\cup C_2$ such that $L=L_1\oplus L_2$, where $L_i$ is a spin structure on $C_i$, one has
$$\om_C\ot L\simeq \om_C|_{C_1}\ot L_1\oplus \om_C|_{C_2}\ot L_2\simeq \om_{C_1}(q_1)\ot L_1\oplus \om_{C_2}(q_2)\ot L_2,$$
where $q_1\in C_1$ and $q_2\in C_2$ glue into the node.
This means that the restrictions of $\chi_1$ and $t\chi_2$ to the components $C_i$ of a curve in the boundary are regular as sections of
$\om_{C_i}(q_i)\ot L_i$, i.e., they can be viewed as sections of $\om_{C_i}\ot L_i$ with the pole of order $1$ at $q_i$.
\end{remark}

\subsubsection{The $(-,-)$ separating node gluing}

Now we consider the gluing construction for the $(-,-)$ separating node (see Sec.\ \ref{--sep-sec}), but with the base $S$ changed to its bosonization $B=S_{\bos}$.
Similarly to Sec.\ \ref{spin++sec}, we first construct sections $\bs_1\in L(\a_1-\a_2-\a_3)$, $\bs_2\in L(\b_1-\b_2-\b_3)$.
Below we will use the elliptic functions $h_u(z_i,\tau_i)$ (see \eqref{h-u-eq}).

\begin{lemma}\label{--spin-lem}
There exist sections $\bs_1\in L(\a_1-\a_2-\a_3)$, $\bs_2\in L(\b_1-\b_2-\b_3)$, with 
$$\bs_1=(t\cdot h_{\a_1}(x_1,\tau_1)\th_1,(1+t^4\wp(\a_1,\tau_1)\wp(x_2,\tau_2))\th_2) \mod (t^5).$$
$$\bs_2=((1+t^4\wp(\b_1,\tau_2)\wp(x_1,\tau_1))\th_1,-t\cdot h_{\b_1}(x_2,\tau_2)\th_2) \mod (t^5),$$
such that 
$$\wp(\a_1)^2\cdot c_1\cdot \bs_1^2=\frac{(x-a_2)(x-a_3)}{x-a_1}\cdot\om^{(2)},$$
$$c_2\cdot \bs_2^2=\frac{(x-b_2)(x-b_3)}{x-b_1}\cdot \om^{(2)},$$
where
$$c_1=\frac{1}{\wp(\a_1)\wp(\a_2)\wp(\a_3)}+O(q^2)=\frac{4}{g_3(\tau_1)}+O(q^2), \ \ c_2=q+O(q^3).$$
\end{lemma}

\begin{proof}
We can look for $\bs_1$ in the form $\bs_1=(F_1(x_1)\th_1,F_2(x_2)\th_2)$,
where $F_1(x_1)$ is a rational function on $C_1-q_1$ with the only pole of order $1$ at $x_1=\a_1$, $F_2(x_2)$ is regular on $C_2-q_2$,
and there exists power series $\phi_1(x_1)$, $\phi_2(x_2)$ satisfying \eqref{odd-functions-eq}.
We have 
$$F_1=a_0+a_1h_{\a_1}+\sum_{n\ge 2} a_n f_n(x_1), \ \ F_2=b_0+\sum_{n\ge 2} b_n f_n(x_2),$$
where as before, $(1, f_n(x_i))$ is a basis of functions on $C_i-q_i$ with $f_n(x_i)=x_i^{-n}+O(1)$.

As before, we can construct $(F_1,F_2,\phi_1,\phi_2)$ by solving \eqref{odd-functions-eq} recursively modulo $t^n$. We can start with the following solution modulo $t^2$: 
$$a_0=0, \ a_1=t, \ a_{\ge 2}=0, \ b_0=1, \ b_{\ge 2}=0, \ \phi_1=a_0+a_1h_{\a_1}(x_1)_{\ge 0}, \ \phi_2=b_0.$$
Using the expansion $h_{\a_1}(x_1)_{\ge 0}=\wp(\a_1)x_1+O(x_1^2)$, we deduce the claimed formula for $F_1,F_2$ modulo $t^5$.
One can also check that $F_1(x_1)$ is an odd function of $x_1$, i.e., $a_{2k}=0$. This implies that $F_1$ vanishes at $\a_2$ and $\a_3$.

The fact that the identity involving $\bs_1^2$ holds up to a constant, follows from the fact that both sides give a section of $\om_{C/B}$ with the pole of order $2$ at $\a_1$ and
double zeros at $\a_2$ and $\a_3$. To compute the constant, we
recall that $a_i\equiv -\wp(\a_i)^{-1} \mod q^4$, so we get
\begin{align*}
&\frac{(x-a_2)(x-a_3)}{x-a_1}\cdot\om^{(2)}=\\
&\frac{\wp(\a_1)}{\wp(\a_2)\wp(\a_3)}\cdot(-q\cdot 
\frac{(\wp(x_1)-\wp(\a_2))(\wp(x_1)-\wp(\a_3))}{\wp(x_1)-\wp(\a_1)}\cdot dx_1,dx_2) \mod (t^4).
\end{align*}
Now the formula for the constant follows from the identity
$$h_{\a_1}(x_1,\tau_1)^2=\frac{(\wp(x_1)-\wp(\a_2))(\wp(x_1)-\wp(\a_3))}{\wp(x_1)-\wp(\a_1)}$$
(see Appendix \ref{ell-sec}).

The formulas involving $\bs_2$ are proved similarly, using $b_i\equiv -q^2\wp(\b_i) \mod q^6$.
\end{proof}

We can use sections $\bs_1$ and $\bs_2$ to get a trivialization of $\Ber_1|_{B}\simeq \det(\pi_*\om_{C/B})\ot \det^{-1}R\pi_*(L)$.
Namely, the resolution
$$R\pi_*L\to [\pi_*L(\a_1+\b_1)]\to L(\a_1)|_{\a_1}\oplus L(\b_1)|_{\b_1}]$$
gives an isomorphism
$${\det}^{-1} R\pi_*L\simeq L(\a_1)|_{\a_1}\ot L(\b_1)|_{\b_1}\ot {\det}^{-1} \pi_*L(\a_1+\b_1).$$
Now $(\bs_1, \bs_2)$ is a basis of $\pi_*L(\a_1+\b_1)$, 
so $\frac{\th_1}{(x_1-\a_1)}\cdot \frac{\th_2}{(x_2-\b_1)}\cdot (\bs_1\we \bs_2)^{-1}$ gives a trivialization of $\det^{-1} R\pi_*L$, and 
$${\bf b(-,-)}:=(\om^{(1)}\we \om^{(2)})\cdot \frac{\th_1}{(x_1-\a_1)}\cdot \frac{\th_2}{(x_2-\b_1)}\cdot (\bs_1\we \bs_2)^{-1}$$
is a trivialization of $\Ber_1|_B$ near the $(-,-)$ separating node divisor.

Note that on the formal punctured neighborhood of the $(-,-)$ boundary divisor we have $R\pi_*(L)=0$, so we have the trivialization $\om^{(1)}\we \om^{(2)}$
of $\Ber_1|_B$ on this punctured neighborhood.

\begin{cor} We have 
$$1=-t^2(1+O(t^2)) \frac{\th_1}{(x_1-\a_1)}\cdot \frac{\th_2}{(x_2-\b_1)}\cdot (\bs_1\we \bs_2)^{-1}$$
in $\det^{-1} R\pi_*L$, so
\begin{equation}\label{(-,-)-om1-om2-eq}
\om^{(1)}\we \om^{(2)}=-t^2(1+O(t^2)) {\bf b(-,-)}.
\end{equation}
\end{cor}

\begin{proof}
This follows from the computation of the map $\pi_*(L(\a_1+\b_1))\to L(\a_1)|_{\a_1}\oplus L(\b_1)|_{\b_1}$:  
$$\bs_1\mapsto -t(1+O(t^2))\cdot \frac{\th_1}{(x_1-\a_1)}, \ \ \bs_2\mapsto t(1+O(t^2))\cdot \frac{\th_2}{(x_2-\b_1)}.$$
\end{proof}

Note that
$$\bs:=\wp(\a_1)c_1^{1/2}c^{1/2}\bs_1, \ \ \bs':=c_2^{1/2}c^{1/2}\bs_2$$
satisfy
$$\bs^2=\frac{(x-a_2)(x-a_3)}{x-a_1}\cdot\frac{dx}{y}, \ \ (\bs')^2=\frac{(x-b_2)(x-b_3)}{x-b_1}\cdot\frac{dx}{y}.$$
Hence, we have
\begin{equation}
\label{chi1-(--)-eq}
\begin{array}{l}
\chi_1=(x-a_1)\cdot \bs\cdot \frac{dx}{y}=c_1^{1/2}c^{3/2}\cdot \wp(\a_1)\cdot (x-a_1)\cdot \bs_1\cdot \om^{(2)}=\\
c_1^{1/2}c^{3/2}\cdot (-\frac{1}{2}t^3(\wp'(x_1,\tau_1)+O(q^2)) [dx_1|d\th_1]^3,(1+O(q^2))[dx_2|d\th_2]^3).
\end{array}
\end{equation}
\begin{equation}
\label{chi2-(--)-eq}
\begin{array}{l}
\chi_2=(x-b_1)\cdot \bs'\cdot \frac{dx}{y}=\\
c^{3/2}(\pm i)t^3((1+O(q^2))[dx_1|d\th_1]^3,\frac{1}{2}t^3(\wp'(x_2,\tau_2)+O(q^2))[dx_2|d\th_2]^3).
\end{array}
\end{equation}
Recall that $c^{3/2}=-t^{-3}((-g_3(\tau_1))^{3/4}+O(t^4))$,
so $t^3\chi_1$ and $\chi_2$ are regular near $t=0$.

\subsection{Mumford form near the  $(+,+)$ separating node divisor}

Recall that we denote by
$\Psi\in \Ber_1^{-5}\ot \om_{\SS_2}$ the (everywhere nonvanishing) Mumford form corresponding to the isomorphism \eqref{hol-Mumf-isom-eq}.

We would like to understand $\Psi$ in terms of the ``gluing coordinates" $\tau_1,\tau_2,t,\eta_1,\eta_2$ near the separating node divisor.
In terms of these coordinates we can write
$$\Psi=[\frac{f_0(\tau_1,\tau_2,t)}{t^2}+\frac{f_1(\tau_1,\tau_2,t)}{t^2}\eta_1\eta_2]s^{-5}\ot [d\tau_1d\tau_2dt|d\eta_1d\eta_2],$$
where 
$$s=\om^{(1)}\we \om^{(2)}$$ 
is a trivializing section of $\Ber_1$.
To find $f_0$, it is enough to look at the restriction $\Psi|_{\SS_{2,\bos}}$ to the reduced space, while to find $f_1$ one has to use a projection.
We will use Witten's identification of $\Psi|_{\SS_{2,\bos}}$ and of $\pi^{\can}_*\Psi$, where $\pi^{\can}:\SS_2\to \SS_{2,\bos}$ is the canonical projection.
We also use the expressions for the coordinates pulled back under $\pi^{\can}$:
$$t'=t+\eta_1\eta_2/2+O(t^4), \ \ \tau'_1=\tau_1-2\pi i t^3\wp(u_2)\eta_1\eta_2+O(t^5),
$$ $$\tau'_2=\tau_2-2\pi i t^3\wp(u_1)\eta_1\eta_2+O(t^5)$$
(see Sec\ \ref{periods-calculation-subsubsec}).

First, we need to recall the relevant line bundles on $\SS_{2,\bos}$. Let $p:C\to \SS_{2,\bos}$ be the universal curve and $L$ the universal spin structure on $C$.
The conormal bundle to the embedding $\SS_{2,\bos}\to \SS_2$ is $\Pi p_*(L\ot \om_{C/\SS_{2,\bos}})$, so
$$\om_{\SS_2}|_{\SS_{2,\bos}}\simeq \om_{\SS_{2,\bos}}\ot {\det}^{-1} p_*(L\ot \om_{C/\SS_{2,\bos}}).$$
Also, 
$$\Ber_1|_{\SS_{2,\bos}}\simeq \la:=\det(p_*\om_{C/\SS_{2,\bos}}).$$
Thus, $\Psi|_{\SS_{2,\bos}}$ is a section of $\la^{-5}\ot \det^{-1} p_*(L\ot \om_{C/\SS_{2,\bos}})\ot \om_{\SS_{2,\bos}}$.


On the other hand, by Proposition \ref{Ber1-can-proj-prop}, we have a natural isomorphism
$$\Ber_1\simeq (\pi^{\can})^*\la.$$
Hence, $\pi^{\can}_*\Psi$ is a well defined section of $\la^{-5}\ot \om_{\SS_{2,\bos}}$.


Recall that Witten uses the identification of $\SS_{2,\bos}$ with the quotient of the configuration of space of $(u_1,u_2,u_3,v_1,v_2,v_3)$ 
by the action of $\SL_2$ and by some finite group (see Section \ref{hyperell-basics-section}). 
There is a natural trivialization of $\om_{\SS_{2,\bos}}$ given by
$\vol^{-1}du_1du_2du_3dv_1dv_2dv_3$,
where $\vol^{-1}$ denotes the contraction with a generator of $\bigwedge^3\ssl_2$.
If we denote the coordinates of ramification points by $e_1,\ldots,e_6$ then one has (see \cite[Eq.(2.10)]{Witten})
\begin{align*}
&\vol^{-1}de_1\ldots de_6=\\
&\sum_{a<b<c}(-1)^{a+b+c}(e_a-e_b)(e_b-e_c)(e_c-e_a)de_1\ldots\widehat{de_a}\ldots\widehat{de_b}\ldots\widehat{de_c}\ldots de_6.
\end{align*}

\begin{prop}\label{p-Witten}
One has 
\begin{equation}\label{Witten-for1}
\Psi|_{\SS_{2,\bos}}=\frac{\vol^{-1}du_1du_2du_3dv_1dv_2dv_3}{\prod_{i<j}((u_i-u_j)(v_i-v_j))\prod_{k,l}(u_k-v_l)^2\cdot\chi_1\we\chi_2\cdot(dx/y\we xdx/y)^5},
\end{equation}
\begin{equation}\label{Witten-for2}
\pi^{\can}_*\Psi=c\cdot\frac{Q(u,v)\cdot \vol^{-1}du_1du_2du_3dv_1dv_2dv_3}{\prod_{i<j}((u_i-u_j)(v_i-v_j))\prod_{k,l}(u_k-v_l)^2\cdot(dx/y\we xdx/y)^5}
\end{equation}
where $c$ is a nonzero constant,
$dx/y\we xdx/y$ is the trivialization of $\la$ corresponding from the hyperelliptic model for the curve,
$\chi_1,\chi_2$ is a basis of $p_*(L\ot \om_{C/\SS_{2,\bos}})$ given by \eqref{chi1-chi2-eq},
$$Q(u,v)=3c_3(u)-c_2(u)c_1(v)+c_1(u)c_2(v)-3c_3(v),$$
where $c_i(x_1,x_2,x_3)$ are elementary symmetric polynomials.
\end{prop}

\begin{proof} These formulas appear as Eq.(3.27) and (3.28) in \cite{Witten}.
Witten's derivation of these two formulas is based on some properties of $\Psi$ and of the projection.
We claim that these properties indeed hold for the canonical projection $\pi^{\can}$. 

For the derivation of \eqref{Witten-for1}, one needs to know that $\Psi_{\SS_{2,\bos}}$ is nonvanishing on the locus $u_i\neq u_j$, $u_i\neq v_j$, and one needs to know the order
of poles along the divisors $(u_i-u_j)$, $(v_i-v_j)$ and $(u_i-v_j)$.
Let $\de_{NS,\bos}$ and $\de_{R,\bos}$ denote the NS and Ramond node components of the boundary
divisor in $\ov{\SS}_{2,\bos}$. We proved in \cite{FKP-supercurves} that $\Psi$ has a simple pole along $\de_R$ and a double pole along $\de_{NS}$. 

Recall that when two of the ramification points merge, we get a point on a non-separating node boundary component. Furthermore, we showed in Section \ref{hyperell-basics-section}
that $(u_i-u_j)^2$ and $(v_i-v_j)^2$ (resp., $(u_i-v_j)$) are the pullbacks of the components $\de_{R,\bos}$ (resp., $\de_{NS,\bos}$). Let $t_R=0$ (resp., $t_{NS}=0$)
be local equations of $\de_{R,\bos}$ (resp., $\de_{NS,\bos}$) near non-separating node components. 
Then, as $u_i-u_j\to 0$, the behavior of $\Psi|_{\SS_{2,\bos}}$ is of the form $dt_R/t_R$, where $t_R=(u_i-u_j)^2$, which explains why $(u_i-u_j)$ appears in the denominator of
the right-hand side of \eqref{Witten-for1}. Similar analysis works for $v_i-v_j$. On the other hand, as $t_{NS}=u_i-v_j\to 0$, the behavior of $\Psi|_{\SS_{2,\bos}}$ is
of the form $dt_{NS}/t_{NS}^2$, which explains the factor $(u_i-v_j)^2$ in the denominator of \eqref{Witten-for1}.

For the derivation of \eqref{Witten-for2} we need to know that $\pi^{\can}_*\Psi$ has at most a simple pole along $\de_{R,\bos}$ and at most a double pole along the nonseparating NS node
component of $\de_{NS,\bos}$. But this follows from the similar behavior of $\Psi$ established in \cite{FKP-supercurves}
and from the fact that near each nonseparating node component $\pi^{\can}$ is regular and
satisfies $(\pi^{\can})^*\de_{\bos}=\de$ (see Theorem \ref{projection-nonsep-prop}).
\end{proof}

We would like to rewrite these formulas in terms of the gluing variables $\tau_1,\tau_2,t$ (up to some power of $t$) near a separating node. 
Note that modulo $t^4$ the canonical projection corresponds
to the variables $\tau_1,\tau_2,t+\eta_1\eta_2/2$. Formulas \eqref{ramification-points-gluing-formulas-eq} 
give an expression of $u_i$, $v_i$ in terms of these variables.
The basis of $\lambda$ used in \eqref{Witten-for1}, \eqref{Witten-for2} is related to
$s=\om^{(1)}\we\om^{(2)}$ via \eqref{lambda-basis-gluing-eq}.


\begin{cor}\label{push-forward-Psi-exp-cor} 
One has 
$$\pi^{\can}_*\Psi=q^{-2}(a+O(q^2))\cdot s^{-5} d\tau_1 d\tau_2dq,$$
where $q=-t^2$,
and $a$ depends only on $\tau_1,\tau_2$ (and a choice of nontrivial points of order $2$).
\end{cor}

\begin{proof} 
Recall that the ramification points have coordinates
$$a_i=-\wp(\a_i,\tau_1)^{-1}+O(q^4), \
b_i=-q^2(\wp(\b_i,\tau_2)+O(q^4)),\ i=1,2,3,$$
where $\a_1,\a_2,\a_3$ (resp., $\b_1,\b_2,\b_3$) are nontrivial points of order $2$ on
the elliptic curve $C_1=\C/(\Z+\Z\tau_1)$ (resp., $C_2=\C/(\Z+\Z\tau_2)$),
and the partition determining the spin structure is $(a_1,b_2,b_3)$, $(b_1,a_2,a_3)$ (see Corollary \ref{glued-spin-partition-cor}).
Now we compute the ingredients of the Witten's formula using the correspondence \eqref{glued-partition-corr-eq}.
We have
$$da_i=-\frac{\partial(\wp^{-1})}{\partial\tau}(\a_i,\tau_1)d\tau_1+O(q^3), \ 
db_i=-2q\wp(\b_i,\tau_2)dq-q^2\frac{\partial\wp}{\partial\tau}(\b_i,\tau_2)d\tau_2+O(q^5)$$
Let us write for brevity 
$$e_i=\wp(\a_i,\tau_1), \ i=1,2,3.$$
Note that $e_1e_2e_3=g_3(\tau_1)/4$.
Then we have
\begin{align*}
&\vol^{-1}da_1da_2da_3db_1db_2db_3=
                 \sum_{\{i<j\}=[1,3]\setminus k, \{i'<j'\}=[1,3]\setminus k'}(-1)^{i+j+k'}
                 \times
  \\
  &(a_i-a_j)(a_i-b_{k'})(a_j-b_{k'})\times\\
&(\frac{\partial(\wp^{-1})}{\partial\tau}(u_k,\tau_1)+O(q^3))\cdot
2q^3({\begin{vmatrix} \wp(\b_{i'},\tau_2) & \frac{\partial\wp}{\partial\tau}(\b_{i'},\tau_2)\\
\wp(\b_{j'},\tau_2) & \frac{\partial\wp}{\partial\tau}(\b_{j'},\tau_2)\end{vmatrix}}+O(q^3))\cdot d\tau_1d\tau_2dq=\\
&q^3((g_3(\tau_1)/4)^{-2}A+O(q^2))d\tau_1d\tau_2dq,
\end{align*}
where
\begin{align*}
A&=
   2\sum_{\{i<j\}=[1,3]\setminus k}(-1)^{i+j}(e_j-e_i)\cdot\frac{\partial\wp}{\partial\tau}(\a_k,\tau_1)\times
   \\ &\quad
\sum_{\{i'<j'\}=[1,3]\setminus k'}(-1)^{k'}
{\begin{vmatrix} \wp(\b_{i'},\tau_2) & \frac{\partial\wp}{\partial\tau}(\b_{i'},\tau_2)\\
\wp(\b_{j'},\tau_2) & \frac{\partial\wp}{\partial\tau}(\b_{j'},\tau_2)\end{vmatrix}}.
\end{align*}
Next, we have
$$(a_1-b_2)(a_1-b_3)(b_2-b_3)(a_2-a_3)(a_2-b_1)(a_3-b_1)=q^2((g_3(\tau_1)/4)^{-2}B+O(q^2)),$$
where
$$B=(e_2-e_3)(\wp(\b_3,\tau_2)-\wp(\b_2,\tau_2)),$$
while
$$(a_1-b_1)^2\prod_{i=2}^3(a_1-a_i)^2(b_1-b_i)^2\cdot\prod_{2\le i,j\le 3}(a_i-b_j)^2=
q^8((g_3(\tau_1)/4)^{-6}C+O(q^2)),$$
where
$$C=(e_1-e_2)^2(e_1-e_3)^2(\wp(\b_1,\tau_2)-\wp(\b_2,\tau_2))^2(\wp(\b_1,\tau_2)-\wp(\b_3,\tau_2))^2.$$

Finally,
$$Q(a_1,b_2,b_3;a_2,a_3,b_1)=-\wp(\a_1)^{-1}\wp(\a_2)^{-1}\wp(\a_3)^{-1}+O(q^2)=-\frac{4}{g_3(\tau_1)}+O(q^2).$$
Thus, taking into account \eqref{lambda-basis-gluing-eq}, we get
$$\pi^{\can}_*\Psi=\const\cdot q^{-2}(\frac{A}{BC}+O(q^2))\cdot s^{-5},$$
which has the required form.
\end{proof}

\begin{remark} The fact that $\pi^{\can}_*\Psi$ has polar behavior $dt/t^3$ near $t=0$ is due to the discrepancy between the divisor $t=0$ and the pull-back
of the NS boundary divisor (given by $t=0$ on the bosonization) under the canonical projection near the separating node component (see Theorem \ref{projection-nonsep-prop}).
\end{remark}

\begin{theorem}\label{Mumford-form-thm}
One has 
$$\Psi=(\frac{c}{t^2}+O(1))\cdot s^{-5}[d\tau_1d\tau_2dt/d\eta_1d\eta_2],$$
where $c$ depends only on $\tau_1,\tau_2$ (and a choice of nontrivial points of order $2$).
\end{theorem}

\begin{proof}
We can write $\Psi$ in the form
$$\Psi=[\frac{a_0}{t^2}+\frac{a_1}{t}+\frac{b_0}{t^2}\eta_1\eta_2+\frac{b_1}{t}\eta_1\eta_2+O(1)]\cdot s^{-5}[d\tau_1d\tau_2dt/d\eta_1d\eta_2],$$
where $a_i$ and $b_i$ depend only on $\tau_1,\tau_2$.

Let us consider the involution $t\mapsto -t$, $\eta_1\mapsto -\eta_1$, $\eta_2\mapsto \eta_2$, $\tau_i\mapsto \tau_i$, which is a part of the $\Z_2$-action
\eqref{main-Z2-action-eq}.
It is easy to check that this involution preserves $s$. Hence, from the invariance of $\Psi$ under this involution we get
$$a_1=b_0=0.$$

Next, using the formula for the pull-back of $(t,\tau_1,\tau_2)$ under $\pi^{\can}$ we can calculate $\pi^{\can}_*\Psi$ in terms of the above expansion:
$$\pi^{\can}_*\Psi=[\frac{a_0}{t^3}+\frac{b_1}{t}+O(1)]\cdot s^{-5}d\tau_1d\tau_2dt.$$
Comparing this with the formula of Corollary \ref{push-forward-Psi-exp-cor} we deduce that $b_1=0$.
\end{proof}

\begin{remark}
In Sec.\ \ref{g2-g1-polar-term-sec} below we will interpret the function $c$ appearing in the formula of Theorem \ref{Mumford-form-thm} in terms
of genus $1$ data. In particular, this will show that it is invertible everywhere outside the deeper strata of the $(+,+)$ separating node divisor.
\end{remark}

Combining Theorem \ref{Mumford-form-thm} with Proposition \ref{period-det-prop}(ii), we derive the following result.

\begin{cor}\label{measure-polar-term-cor}
Let us write $\Psi=\Psi_0+\Psi_1t\eta_1\eta_2$, where $\Psi_0=\frac{c}{t^2}+O(1)$ and $\Psi_1$ is regular.
Then the superstring measure $\mu=\Psi\cdot \ov{\Psi}\cdot h^5$ has form
$$\mu=[\mu_0+\mu_1\eta_1\eta_2+\ov{\mu}_1\ov{\eta}_1\ov{\eta}_2+\mu_{11}\eta_1\eta_2\ov{\eta}_1\ov{\eta}_2]\cdot 
[d\tau_1d\tau_2dt/d\eta_1d\eta_2]\ot [d\ov{\tau}_1d\ov{\tau}_2d\ov{t}/d\ov{\eta}_1d\ov{\eta}_2],$$ 
with
$$\mu_{11}=-40\pi^2\cdot h_0^6\cdot\frac{c\cdot \ov{c}}{t\cdot \ov{t}}+\frac{O(\ov{t})}{t}+\frac{O(t)}{\ov{t}},$$
where $h|_{t=\ov{t}=0}=h_0\cdot s\cdot \ov{s}|_{t=\ov{t}=0}$, and
$c$ is from Theorem \ref{Mumford-form-thm}.
\end{cor}

\begin{proof}
Using the notation from Proposition \ref{period-det-prop}(ii), we have
\begin{align*}
&\mu_{11}=h_0^5t\ov{t}\Psi_1\ov{\Psi}_1+5h_0^4h_1\ov{t}\Psi_0\ov{\Psi}_1+5h_0^4\ov{h}_1t\ov{\Psi}_0\Psi_1+\\
&(5h_0^4h_{11}+20h_0^3h_1\ov{h}_1)\Psi_0\ov{\Psi}_0.
\end{align*}
By Proposition \ref{period-det-prop}(ii), only term $5h_0^4h_{11}\Psi_0\ov{\Psi}_0$ will contribute to the polar part.
It remains to use Corollary \ref{h0-h11-cor} and Theorem \ref{Mumford-form-thm} to get the assertion.
\end{proof}

\section{Genus $2$ superstring measure near the separating node divisor}\label{Section-6}

\subsection{Mumford isomorphism on $\SS_{1,1}$}

Below we use the notations of Sec.\ \ref{Section-6.1}, so $\pi:X\to \SS_{1,1}$ is the universal supercurve, $P\sub X$ the universal NS puncture, $\rho:\SS_{1,1}\to\SS_{1,1,\bos}$
the projection, etc.
The Mumford isomorphism on $\SS_{1,1}$ has form
$$\Psi_{1,1}:\Ber_1^5\ot \LL\rTo{\sim} \om_{\SS_{1,1}},$$
where $\LL=P^*\om_{X/\SS_{1,1}}$ (in \cite{FKP-supercurves} this line bundle is called $\Psi_1$).

It is easy to see that $\LL$ gets identified with the pull back of $\Pi L|_p$ on $\SS_{1,1,\bos}$:
$$\LL\simeq \Pi \rho^* L|_p.$$
Indeed, we have to construct an isomorphism $\om_{X/\SS_{1,1}}\ot_{\OO_X} \OO_{\SS_{1,1}}\simeq \Pi L|_p$, where we use the homomorphism $\ev_P:\OO_X\to \OO_{\SS_{1,1}}$.
This isomorphism is induced by the natural map
$$\om_{X/\SS_{1,1}}=\om_{C/\SS_{1,1}}\oplus L\rTo{(s\cdot \ev_p,\ev_p)} L|_p$$ 
Hence, we have a tautological even section 
$$\eta\in H^0(\SS_{1,1},\LL^{-1}).$$

On the other hand, we claim that there is a canonical isomorphism 
$$\Ber_1\simeq \LL^2$$
on $\SS_{1,1}$. Indeed, this follows easily from the isomorphism
$$\Ber_1\simeq \rho^*\pi_*\om_{X/\SS_{1,1}}$$
together with the fact that $\pi_*\om_{C/\SS_{1,1}}\to \om_{C/\SS_{1,1}}|_p$ is an isomorphism.

Finally, we observe that we have an exact sequence
$$0\to \rho^*\om_{\SS_{1,1,\bos}}\to \Om^1_{\SS_{1,1}}\to \LL\to 0,$$
so passing to the Berezinians (and remembering that $\LL$ is odd) we get an isomorphism
$$\om_{\SS_{1,1}}\simeq \rho^*\om_{\SS_{1,1,\bos}}\ot \LL^{-1}.$$
Thus, we can view the Mumford isomorphism as an isomorphism
$$\Psi_{1,1}:\Ber_1^6\simeq \Ber_1^5\ot \LL^2\rTo{\sim}\rho^*\om_{\SS_{1,1,\bos}}.$$
Since $\Ber_1\simeq \rho^*\la$, where $\la$ is the Hodge bundle on $\SS_{1,1,\bos}$, and
since the Mumford isomorphism is even, it follows that we can identify $\Psi_{1,1}$ 
with the pullback of an isomorphism on $\SS_{1,1,\bos}$,
$$\Psi_1:\la^6\rTo{\sim} \om_{\SS_{1,1,\bos}}.$$

We claim that in fact $\Psi_1$ automatically comes from a similar isomorphism on $\SS_{1,\bos}$, the moduli stack of
even spin curves of genus $1$. Indeed, this follows from the fact that the projection 
$$pr:\SS_{1,1,\bos}\to \SS_{1,\bos}$$
satisfies $pr_*\OO=\OO$, and both line bundles $\la$ and $\om_{\SS_{1,1,\bos}}$ descend to $\SS_{1,\bos}$.

Thus, $\Psi_1$ is what Witten calls $\Psi_{1,+}$ and we would like to justify Witten's formula \cite[(3.22)]{Witten} for it, which uses
the hyperelliptic model. We need to recall this formula and to justify it.

Let us think of the elliptic curve $E$ as a double cover of $\P^1$ given by
$$y^2=(x-u_1)(x-u_2)(x-v_1)(x-v_2),$$
where the even spin-structure on $E$ corresponds to the division of the branch points into two sets of $2$,
$(u_1,u_2)$ and $(v_1,v_2)$. Then the Mumford form can be written as
$$\Psi_1=F(u_1,u_2,v_1,v_2)\cdot (dx/y)^{-6}\cdot \vol^{-1}\cdot du_1du_2dv_1dv_2,$$
where $F$ is invertible as long as the $u_i$ and $v_j$ are distinct.
Furthermore, the $\SL_2$-invariance and the invariance under permutations preserving the division into two (unordered)
subsets show that
$$F=(u_1-u_2)^a(v_1-v_2)^a\cdot \prod_{i,j=1}^2 (u_i-v_j)^b,$$
for some integers $a$, $b$ such that $a$ is odd and $2a+4b=-10$. (We can get rid of a universal constant factor in $F$ by rescaling the volume form on $\ssl_2$.)
So to determine $F$ completely one needs to find either $a$ or $b$.

\begin{lemma}\label{g1-Mumf-lem} 
$F$ has pole of order $1$ near $v_1-v_2=0$, so $a=-1$, $b=-2$ and
$$F=(u_1-u_2)^{-1}(v_1-v_2)^{-1}\cdot \prod_{i,j=1}^2 (u_i-v_j)^{-2},$$
\end{lemma}

\begin{proof}
We will deduce this from the fact that $\Psi_1$ has pole of order $1$ near the Ramond boundary divisor
by \cite[Thm.\ B]{FKP-supercurves}. 

Let us consider the map from the upper half-plane to $\SS_{1,1,\bos}$ sending $\tau$ to the elliptic curve $\C/(\Z+\Z\tau)$,
equipped with the hyperelliptic covering is given by $\wp(z)^{-1}$, so the ramification points $u_1,u_2,v_1,v_2$ are given by 
$$u_1=0, \ u_2=e_1^{-1}=\wp(1/2,\tau)^{-1},$$
$$v_1=e_2^{-1}=\wp(\tau/2,\tau)^{-1}, \ v_2=e_3^{-1}=\wp((\tau+1)/2,\tau)^{-1},$$
and the even spin structure corresponding to the grouping $(u_1,u_2), (v_1,v_2)$ of the ramification points.


Note that $e_2$ and $e_3$ are functions of $q^{1/2}:=\exp(\pi i\tau)$, such that $e_3(q^{1/2})=e_2(-q^{1/2})$, while $e_1$ is a function of $q$.
We need to check that as $q^{1/2}\to 0$, $v_1-v_2$ agrees with the Ramond boundary divisor, while $\vol^{-1}\cdot du_1du_2dv_1dv_2$ has a nonzero limit.

Now, as $q^{1/2}$ goes to zero, we get 
$$e_1\to \frac{2\pi^2}{3}, \ \ e_2\to -\frac{\pi^2}{3}, \ \ e_3\to -\frac{\pi^2}{3}.$$
Furthemore, the well-known expansion
$$\la(\tau)=\frac{e_3-e_2}{e_1-e_2}=16q^{1/2}+O(q)$$
implies that $e_3-e_2=16\pi^2 q^{1/2}+O(q)$. Hence, 
$$v_1-v_2=e_2^{-1}-e_3^{-1}=\frac{16\cdot 9}{\pi^2} q^{1/2}+O(q),$$
so $(v_1-v_2)=(q^{1/2})$. On the other hand, the stack $\ov{\MM}_{1,1}$ near the cusp is modeled by the quotient of a small neighborhood of $0$ in
the coordinate $x=q^{1/2}$ by the action of $\Z/2$. Since the projection $\ov{\SS}_{1,1,\bos}\to \ov{\MM}_{1,1}$ is unramified near the Ramond boundary divisor,
the pull-back of the Ramond boundary divisor also corresponds to the ideal $(q^{1/2})$.

Note that from the expansion of $e_3-e_2$ we deduce that 
$$e_2=-\frac{\pi^2}{3}-8\pi^2 q^{1/2}+O(q), \ \ e_3=-\frac{\pi^2}{3}+8\pi^2 q^{1/2}+O(q).$$

We have
$$\vol^{-1}du_1du_2dv_1dv_2=a\cdot d\tau=\frac{a}{\pi i q^{1/2}}dq^{1/2},$$
where 
$$a=-e_2^{-1}e_3^{-1}(e_2^{-1}-e_3^{-1})\frac{\partial e_1^{-1}}{\partial \tau}+
e_1^{-1}e_3^{-1}(e_1^{-1}-e_3^{-1})\frac{\partial e_2^{-1}}{\partial \tau}-e_1^{-1}e_2^{-1}(e_1^{-1}-e_2^{-1})\frac{\partial e_3^{-1}}{\partial \tau}.$$
Now the fact that $a(q^{1/2})/q^{1/2}$ has nonzero limit as $q^{1/2}\to 0$ is easily deduced from the above expansions of $e_2$ and $e_3$.
\end{proof}

\subsection{Polar term near the $(+,+)$ separating node divisor in terms of genus $1$ data}\label{g2-g1-polar-term-sec}

We would like to interpret the formula of Corollary \ref{measure-polar-term-cor} for
the polar term of the projection of the superstring measure
$$\mu_{11}[d\tau_1d\tau_2d\ov{\tau}_1d\ov{\tau}_2dtd\ov{t}]=
-40\pi^2\cdot h_0^6\cdot\frac{c\cdot \ov{c}}{t\cdot \ov{t}}[d\tau_1d\tau_2d\ov{\tau}_1d\ov{\tau}_2dtd\ov{t}]+\frac{O(\ov{t})}{t}+\frac{O(t)}{\ov{t}}$$
in more invariant terms.

First, we observe that for the family of supercurves given by the canonical coordinates $(t,\tau_1,\tau_2,\eta_1,\eta_2)$
near the boundary divisor we have trivializations $\th_i\in \LL_i$, $i=1,2$. Furthermore, under the canonical
identification of the normal bundle
$$N_{\De_{NS}}\simeq \LL_1^{-1}\ot \LL_2^{-1}$$
the trivialization $t^{-1}$ of the normal bundle corresponds to $\th_1^{-1}\ot \th_2^{-1}$ (see \cite[Lem.\ 10.5]{FKP-supercurves}).

On the other hand, the trivialization $s$ of $\Ber_1$ restricts to the trivialization $s_1\ot s_2$ of $\Ber_1\boxtimes \Ber_1$
on $\De_{NS}$, where $s_i$ corresponds to the basis $dz_i$ of the Hodge bundle.

Thus, we can identify the polar term of $\Psi$ as
$$\frac{c}{t^2}\cdot s^{-5}[d\tau_1d\tau_2dt/d\eta_1d\eta_2]|_{\De_{NS}}=
c_1\th_1^{-1}s_1^{-5}[d\tau_1/d\eta_1]\cdot c_2\th_2^{-1}s_2^{-5}[d\tau_2/d\eta_2],$$
where 
$$c=c_1\cdot c_2$$ 
and
$$\Psi_{1,1}=c_i\th_i^{-1}s_i^{-5}[d\tau_i/d\eta_i]$$
is the Mumford isomorphism on the $i$-th copy of $\SS_{1,1}$.
The trivialization given by the section
$\th_i^{-1}s_i^{-5}[d\tau_i/d\eta_i]$ of $\LL^{-1}\ot \Ber_1^{-5}\ot \om_{\SS_{1,1}}$
corresponds to the trivialization $s_i^{-6}[d\tau_i]$ of the pull-back of $\Ber_1^{-6}\ot \om_{\SS_{1,\bos}}$. 
Thus, using the identification of $\Psi_{1,1}$ with the pull-back of $\Psi_1$, we get
$$\Psi_1=c_i s_i^{-6}[d\tau_i].$$
Note that this implies in particular, that the function $c$ used in Theorem \ref{Mumford-form-thm} and Corollary \ref{measure-polar-term-cor},
is invertible away from the deeper strata of the $(+,+)$ separating node divisor.

On the other hand, we have 
$$h_0\cdot s\cdot \ov{s}|_{t=\ov{t}=0}=h^{(1)}\cdot h^{(2)},$$
where $h^{(i)}$ is the hermitian form of $\Ber_1\boxtimes\ov{\Ber}_1$ on the $i$th factor of $\SS_{1,1}\times\ov{\SS}_{1,1}$.
Thus, we can rewrite
\begin{align*}
&h_0^6\cdot c\cdot \ov{c}[d\tau_1 d\tau_2d\ov{\tau}_1d\ov{\tau}_2]=
(h_0\cdot s\cdot\ov{s})^6\cdot (c\cdot s^{-6})\cdot [d\tau_1 d\tau_2]\cdot (\ov{c}\cdot \ov{s}^{-6})\cdot [d\ov{\tau}_1d\ov{\tau}_2]=\\
&(h^{(1)})^6\cdot (h^{(2)})^6\cdot (c_1\cdot s_1^{-6}[d\tau_1]\cdot c_2\cdot s_2^{-6}[d\tau_2])
\cdot (\ov{c}_1\cdot \ov{s}_1^{-6}[d\ov{\tau}_1]\cdot \ov{c}_2\cdot \ov{s}_2^{-6}[d\ov{\tau}_2]).
\end{align*}
Thus, we arrive to the following invariant reformulation of Corollary \ref{measure-polar-term-cor}.

\begin{prop} \label{Proposition-6.2} One has 
$$t\cdot\ov{t}\cdot\mu_{11}[d\tau_1d\tau_2d\ov{\tau}_1d\ov{\tau}_2dtd\ov{t}]|_{t=\ov{t}=0}=
- 40\pi^2\cdot (h^{(1)})^6\cdot (h^{(2)})^6\cdot (\Psi_1\times \Psi_1)\cdot (\ov{\Psi}_1\times\ov{\Psi}_1)$$
\end{prop}

Note that the trivialization $s_i=dz_i$ coincides with $dx/y$ up to sign, since 
$x=\wp(z_i)^{-1}$ and $y=\wp'(z_i)/\wp(z_i)^2$.
Hence, comparing the two formulas for $\Psi_1$ we deduce
$$\omega_i:=c_i[d\tau_i]=F\cdot \vol^{-1}du_1du_2dv_1dv_2.$$
In particular, viewing $\omega_i$ as a $1$-form on $\SS_{1,1,\bos}$ 
and using the formula for $F$ together with the identity
$$(u_1-u_2)(v_1-v_2)+(u_1-v_1)(v_2-u_2)+(u_1-v_2)(u_2-v_1)=0,$$
we deduce that the sum of
$\omega_i$ over choices of even spin-structures is zero.
Since $s_i$ does not depend on a spin structure, the same is true for $\Psi_1$:
its sum over choices of even spin-structures is zero, i.e.,
\begin{equation}\label{sum-over-spin-str-eq}
\pi_*\Psi_1=0,
\end{equation}
where $\pi:\SS_{1,1,\bos}\to \MM_{1,1}$ is the projection.



%


\subsection{Behavior near the $(-,-)$ separating node divisor}\label{Section-7.2}

Recall that the superstring measure has form $\mu=\Psi\cdot \wt{\Psi}\cdot h^5$, where $\Psi$ and $\wt{\Psi}$ are Mumford forms, and $h$ is a certain section of
$\Ber_1\boxtimes \wt{\Ber}_1$, defined near quasidiagonal.
We start with a general statement on the behavior of $h^5$ near the $(-,-)$ separating node divisor $D^{-,-}$ for $g\le 11$ (note that $D^{-,-}$ may have several components).

\begin{prop}\label{mu-reg-prop}
For $g\le 11$, the section $h^5$ is regular on a neighborhood of the quasidiagonal near $D^{-,-}$. Furthermore, for a local choice
 of a Largangian subbundle 
$\La\sub R^1\pi_*\C_{X/\SS_g}$ near a point $s$ on $D^{-,-}$ corresponding to a union of two smooth curves,
such that $\La_s$ is transversal to $H^0(\om_{C_s})\sub H^1(C_s,\C)$,
the section $h^5$ vanishes to the order $\ge 11-g$ along the divisor given by $f_{\La}\wt{f}_{\La}$, where $f_{\La}$ is a local regular function vanishing
on $D^{-,-}_{\bos}$ associated with $\La$ (see Remark \ref{theta-null-rem}).
\end{prop}

\begin{proof}
The proof uses the same method as in \cite[Sec.\ 5.4]{FKP-per}.
The key observation is that we still have the local system $R^1\pi_*(\C_{X/S})$ near this component, where $\pi:X\to S$ is the universal supercurve,
and for the underlying usual family of stable curves, $\pi:C\to S_0$, we still have an embedding of a subbundle
$$\pi_*\om_{C/S_0}\sub R^1\pi_*(\C_{C/S_0})=R^1\pi_*(\C_{X/S_0}).$$

To study the behavior of the superperiods near this boundary component we have to modify the Lagrangian setup of \cite[Sec.\ 4.6]{FKP-per}, by replacing
one marked point with two marked points (one on each component of the reducible curve).  

Namely, we work with a family of stable supercurves $\pi:X\to S$, equipped with a collection of (disjoint) NS-punctures $P_1,\ldots,P_r$ (below we will only need the case $r=2$),
such that there is at least one puncture on each component of any geometric fiber $X_s$. We denote by $D_i\sub X$ the corresponding Cartier divisors 
(see \cite[Sec.\ 2.5]{FKP-supercurves}), and set $D=\sum_i D_i$.

We consider the vector bundle 
$$\wt{\VV}=\pi_*(\OO_X(nD)/\OO_X(-ND))$$
for $N\gg n$ and $n$ sufficiently large so that $R^1\pi_*(\OO_X(nD))=\pi_*(\OO_X(-nD))=0$.
We denote by $\VV$ the quotient of $\wt{\VV}$ by the kernel of the skew-symmetric form $B(f,g)=\sum_i\Res_{D_i}(f\de g)$.
Here $\de:\OO_X\to \om_{X/S}$ is the canonical derivation, and $\Res_{D_i}$ is the canonical functional with values in $\OO_S$ on sections of $\om_{X/S}(mD_i)$ in 
a formal neighborhood of $D_i$ (see \cite[Sec.\ 2.8]{FKP-per}). As in \cite[Sec.\ 4.6]{FKP-per}, we define two Lagrangian subbundles in $\VV$, the canonical one
$$L_{\can}:=\pi_*(\OO_X/\OO_X(-(n+1)D))/\OO_S\sub \VV,$$
and the Lagrangian subbundle $L_\La\sub \VV$ defined as the preimage of a Lagrangian subbundle $\La\sub R^1\pi_*\C_{X/S}$ in
$$\wt{L}_{\can}:=\pi_*(\OO_X(nD)/\C_{X/S})\sub \VV.$$

As in \cite[Sec.\ 4.7, Sec.\ 5.4]{FKP-per}, locally we choose the subbundle $\La=W\sub R^1\pi_*\C_{X/S}$, transversal to  
$$\pi_*\om_{C/S_{\bos}}\sub R^1\pi_*\C_{C/S_{\bos}}=R^1\pi_*\C_{X/S}|_{S_{\bos}}.$$
Then, applying \cite[Thm.\ 4.14]{FKP-per} 
we get a regular even function $f=f_\La$ such that $f^2$ differs by a unit from the rational function
$\th_\La^{-1}$ with respect to a local trivialization of $\Ber_1$ (see Remark \ref{theta-null-rem}), and 
\begin{equation}\label{h-f2-a-eq}
h=f^2\cdot \wt{f}^2\cdot a, 
\end{equation}
where $a$ lies an $\OO_{\SS_g\times\SS_g^c}$-subalgebra generated by functions in $\NN^2/f$ and $\wt{\NN}^2/\wt{f}$. Then the result is deduced as in \cite[Sec.\ 5.4]{FKP-per}, raising \eqref{h-f2-a-eq} to the $5$th power and using the fact that the number of odd directions on $\ov{\SS}_g$ is $2g-2$. 
\end{proof}

Now let us specialize to the case $g=2$.

\begin{theorem}\label{mu-g2-van-thm}
For $g=2$, $h^5$ vanishes to the order $\ge 9$
on $(D^{-,-})\times \ov{\SS}_2^c\cup \ov{\SS}_2\times(D^{-,-})^c$.
Hence, the superstring measure $\mu$  is regular on a neighborhood of the quasidiagonal near $D^{-,-}$ and vanishes to the order of $\ge 7$
on $(D^{-,-})\times \ov{\SS}_2^c\cup \ov{\SS}_2\times(D^{-,-})^c$.
\end{theorem}

\begin{proof}
Note that the second assertion follows from the first since $\Psi$ has a pole of order $2$ along $D^{-,-}$.
Taking into account Proposition \ref{mu-reg-prop}, it is enough to check that near a point of $D^{-,-}$ corresponding to a curve $C=C_1\cup C_2$,
where $C_i$ are smooth curves of genus $1$, with the spin structure $L=\OO_{C_1}\oplus \OO_{C_2}$, the local function $f_{\La}$ (see Remark \ref{theta-null-rem})
for some choice of $\La$ has form $g\cdot t$, where $g$ is invertible and $t=0$ is the equation of the boundary divisor.

Equivalently, we have to prove that $\th_{\La}^{-1}$ has form $t^2\cdot \bt$ for some local trivialization $\bt$ of $\Ber_1$.
We will use the model $X\to S$ for the universal supercurve in a formal neighborhood of $D^{-,-}$ described in Sec.\ \ref{--sep-sec},
so $S$ has even coordinates $\tau_1$, $\tau_2$, $t$ and odd coordinates $\eta_1$, $\eta_2$.
Since $\pi_*\om_{X/S}$ is not locally trivial near $(C,L)$, we will pick nontrivial points of order two $u_1\in C_1$, $u_2\in C_2$ (and extend them to nearby curves), and
use the resolution of $R\pi_*\om_{X/S}$ of the form 
$[B\to C]$, where
$$B=\pi_*\om_{X/S}(D_{u_1}+D_{u_2}), \ \ C=\pi_*(\om_{X/S}(D_{u_1})|_{D_{u_1}})\oplus \pi_*(\om_{X/S}(D_{u_2})|_{D_{u_2}}),$$
where $D_{u_1}$ and $D_{u_2}$ are the relative divisors in $X$ corresponding to the NS punctures $u_1$ and $u_2$ (so $D_{u_i}$ is given by the ideal $(x_i-u_i)$).
Moreover, we have a natural surjective map $r:C\to \OO_S$ given by the sum of residues, so that if replace $C$ by $C_0=\ker(r)\sub C$,
we still have an identification of $\Ber[B\to C_0]$ with $\Ber_1$.
Note that over the punctured neighborhood of $D^{-,-}$, we also have the bundle $A:=\pi_*\om_{X/S}$ and a quasi-isomorphism
\begin{equation}\label{A-B-C-eq}
A\to [B\to C_0].
\end{equation}
Unraveling the definition of $\th_{\La}^{-1}$ for $\La$ associated with a choice of symplectic bases in $H_1(C_1,\Z)$ and $H_1(C_2,\Z)$, we obtain
$$\th_{\La}^{-1}=\ber(\La\rTo{\a^{-1}} \pi_*\om_{X/S}\to [B\to C_0]),$$
where $\a:\pi_*\om_{X/S}\to \La\simeq \OO_S^2$ is given by the evaluation on the cycles $\a_1\in H_1(C_1,\Z)$ and $\a_2\in H_1(C_2,\Z)$.
Equivalently, $\th_{\La}$ is the image of $\om^{(1)}\we \om^{(2)}$, where $(\om^{(1)},\om^{(2)})$ are normalized differentials (defined for on the punctured neighborhood)
under the isomorphism
$$\Ber(A)\to \Ber[B\to C_0]$$
induced by the quasi-isomorphism \eqref{A-B-C-eq}.
To calculate the latter isomorphism we will compute the maps $i:A\to B$ and $p:B\to C_0$ explicitly in terms of some bases (note that $\rk A=2|0$, $\rk B=3|2$ and $\rk C=1|2$).

We already have the basis $(\om^{(1)},\om^{(2)})$ of $A$ (see \eqref{--norm-diff-eq}), and there is a natural basis of $C_0$ induced by the elements
$$(\frac{s_i}{x_i-u_i})_{i=1,2}, \  \frac{s_1\th_1}{x_1-u_1}-\frac{s_2\th_2}{x_2-u_2}.$$
So the main part of the calculation is finding the basis of $B=\pi_*\om_{X/S}(D_{u_1}+D_{u_2})$.
We use the notation of Sec.\ \ref{--sep-sec}, and describe sections of $\pi_*\om_{X/S}(D_u+D_v)$ by pairs $(\om_1,\om_2)$, where
$\om_i\in \om_{X_i/S}(D_{u_i})(X\setminus\{q_i\})$, $i=1,2$, such that there exist regular series $\phi_1(x_1,\th_1)$, $\phi_2(x_2,\th_2)$ and a constant $\phi_0\in \OO_S$ 
satisfying \eqref{om-phi-super-relation-eq}.

First, we observe that to get a basis of $\om_{X_i/S}(D_{u_i})(X\setminus\{q_i\})$ for $i=1,2$, we have to add to the basis \eqref{--diff-basis-eq} two additional elements
$$H_{u_i}(x_i,\th_i,\tau_i):=h_{u_i}(x_i,\tau_i)+\th_i\eta_i\dot{h}_{u_i}(x_i,\tau_i), \ \ \psi_1(x_i,\th_i)\cdot h_{u_i}(x_i,\tau_i),$$
where $h_{u_i}(x_i,\tau_i)$ is given by \eqref{h-u-eq}.
Then solving the relations \eqref{om-phi-super-relation-eq} iteratively as in Prop.\ \ref{--even-diff-prop}, we obtain formulas expressing 
$(\om_1,\om_2)\in \pi_*\om_{X/S}(D_u+D_v)$ in terms of five functions on $\OO_S$: 

\noindent
1) the coefficient of $s_1\psi_1(x_1,\th_1)$ in $\om_1$;

\noindent
2) the coefficient of $s_2\psi_1(x_2,\th_2)$ in $\om_2$;

\noindent
3) $\phi_0$;

\noindent
4) the coefficient of $s_2$ in $\om_2$;

\noindent
5) the coefficient of $s_1$ in $\om_1$.

From this we get five basis elements for $B$ with the following expansions modulo $t^3$:
%
\begin{align*}
\bv_1&:=(s_1\cdot(-H_{u_1}(x_1,\th_1)\cdot \frac{\eta_1}{2\pi i}+\psi_1(x_1,\th_1)), s_2\cdot 2\pi i t^2\psi_2(x_2,\th_2)),\\
\bv_2&:=(s_1\cdot 2\pi i t^2\psi_2(x_1,\th_1), s_2\cdot(-H_{u_2}(x_2,\th_2)\cdot\frac{\eta_2}{2\pi i}+\psi_1(x_2,\th_2)),\\
  \bv&:=\Biggl(s_1\cdot\biggl(H_{u_1}(x_1,\th_1)\cdot t(\wp(u_2,\tau_2)+(2\pi i)^2\frac{E_2(q_2)}{12})\frac{\eta_2}{2\pi i}
  \\
  &\hspace{1.8cm}-\psi_1(x_1,\th_1)h_{u_1}(x_1)+f_2(x_1,\th_1)\frac{\eta_1}{2\pi i}\biggr),\\
     &\hspace{1cm} s_2\cdot\biggl(H_{u_2}(x_2,\th_2)\cdot t(\wp(u_1,\tau_1)+(2\pi i)^2\frac{E_2(q_1)}{12})\frac{\eta_1}{2\pi i}
  \\
  &\hspace{1.8cm}+\psi_1(x_2,\th_2)h_{u_2}(x_2)-f_2(x_2,\th_2)\frac{\eta_2}{2\pi i}\biggr)\Biggr),\\
\bphi_1&:=(s_1\cdot t H_{u_1}(x_1,\th_1),s_2),  \\
\bphi_2&:=(s_1, -s_2\cdot t H_{u_2}(x_2,\th_2)),
\end{align*}
so that the map $i:A\to B$ is given by 
$$\om^{(1)}\mapsto \bv_1+\bphi_1\cdot \frac{1}{t}\cdot \frac{\eta_1}{2\pi i}, \ \
\om^{(2)}\mapsto \bv_2-\bphi_2\cdot \frac{1}{t}\cdot \frac{\eta_2}{2\pi i},$$
while the map $p:B\to C_0$ is given by 
\begin{align*}
&\bv_i\mapsto \frac{s_1}{x_1-u_1}\cdot \frac{\eta_i}{2\pi i}, \ i=1,2,\\
&\bv\mapsto -\frac{s_1}{x_1-u_1}\cdot [\zeta_1(u_1,\tau_1)\eta_1+t(\wp(u_2,\tau_2)+(2\pi i)^2\frac{E_2(q_2)}{12})\frac{\eta_2}{2\pi i}]+\\
&\frac{s_2}{x_2-u_2}\cdot [\zeta_1(u_2,\tau-2)\eta_2-t(\wp(u_1,\tau_1)+
(2\pi i)^2\frac{E_2(q_1)}{12})\frac{\eta_1}{2\pi i}]+\frac{s_1\th_1}{x_1-u_1}-\frac{s_2\th_2}{x_2-u_2},\\
&\bphi_1\mapsto -\frac{s_1}{x_1-u_1}\cdot t,\\
&\bphi_2\mapsto \frac{s_2}{x_2-u_2}\cdot t.
\end{align*}
We can choose a section $\si:C_0\to B$ of $p:B\to C_0$ such that
$$\si(\frac{s_1}{x_1-u_1})=-\frac{1}{t} \bphi_1, \ \ \si(\frac{s_2}{x_2-u_2})=\frac{1}{t} \bphi_2,$$
and the coefficient of $\bv$ in $\si(\frac{s_1\th_1}{x_1-u_1}-\frac{s_2\th_2}{x_2-u_2})$ is $1$.
Then $\th_{\La}^{-1}$ can be identified with the Berezinian of the map $(i,\si):A\oplus C_0\to B$. Thus, modulo $t^3$ we get
$$\th_{\La}^{-1}=\Ber\left[\begin{matrix} 1 & 0 & * & 0 & 0 \\ 0 &1 & * & 0 & 0 \\ 0 & 0 & 1 & 0 & 0 \\ \frac{\eta_1}{2\pi i t} & 0  & * & -\frac{1}{t} & 0 \\ 
0 & -\frac{\eta_2}{2\pi i t} & * & 0 & \frac{1}{t} \end{matrix}\right]=-t^2$$
(note that in this matrix the first three columns and rows correspond to even generators, while the last two columns and rows correspond to odd generators).
This implies the assertion we want.
\end{proof}

\begin{remark}
The meromorphic differentials $(\bphi_1,\bphi_2)$ restrict to meromorphic sections $(\bs_1,\bs_2)$ of the spin structure $L$ on the bosonization of the family,
constructed in Lemma \ref{--spin-lem}.  
\end{remark}

\appendix
\section{Elliptic functions}\label{ell-sec}


For $\tau$ in the upper half-plane we denote by $\zeta(z,\tau)$ and $\wp(z,\tau)$ the Weierstrass zeta and $\wp$-functions associated with the lattice $\Z+\Z\tau$.
We will often omit $\tau$ from the notation.

Recall that 
$$\zeta(z+1,\tau)=\zeta(z,\tau)+\eta_1(\tau), \ \ \zeta(z+\tau,\tau)=\zeta(z,\tau)+\eta_2(\tau),$$
where 
$$\tau\eta_1(\tau)-\eta_2(\tau)=2\pi i,$$
$$\eta_1(\tau)=-(2\pi i)^2\cdot \frac{E_2(q)}{12},$$
for $q=\exp(2\pi i\tau)$, where 
$$E_2(q):=1-24\sum_{n\ge 1}\frac{nq^n}{1-q^n},
$$
As in \cite{P-superell}, it is convenient to consider the function
$$\zeta_1(z)=\zeta_1(z,\tau)=(-2\pi i)^{-1}(\zeta(z,\tau)-\eta_1(\tau)z)$$
satisfying $\zeta_1(z+1,\tau)=\zeta_1(z,\tau)$, $\zeta_1(z+\tau,\tau)=\zeta_1(z,\tau)+1$.

For each $u\in\C\setminus \Z+\Z\tau$, we set
\begin{equation}\label{h-u-eq}
h_{u}(z,\tau):=\zeta(z,\tau)-\zeta(z-u,\tau)-\zeta(u,\tau)=-\frac{\wp'(z,\tau)}{2(\wp(z,\tau)-\wp(u,\tau))}.
\end{equation}
Note that $h_u(z,\tau)$ is $\Z+\Z\tau$-periodic in both $z$ and $u$, with poles of order $1$ at $0$ and $u$ and the residues of $h_u(z)dz$  at these points are
$1$ and $-1$.
One has
\begin{equation}\label{h-u-odd-eq}
h_{-u}(-z,\tau)=-h_u(z,\tau).
\end{equation}

We will often work with points of order $2$ on the elliptic curve $\C/(\Z+\Z\tau)$.
Recall that if $(\a_1,\a_2,\a_3)$ are nontrivial points of order $2$ then $\wp'(\a_i)=0$, and so
$\wp(\a_1)$, $\wp(\a_2)$, $\wp(\a_3)$ are the roots of the cubic polynomial $4x^3-g_2(\tau)x-g_3(\tau)$.
In particular,
\begin{equation}
4\wp(\a_1)\wp(\a_2)\wp(\a_3)=g_3(\tau).
\end{equation}

The identity \eqref{h-u-odd-eq} shows that the function $h_{\a_1}(z)$ is odd, hence, vanishes at $\a_2$ and $\a_3$. This can be used to prove the following identity:
$$h_{\a_1}(z)^2=\frac{(\wp(z)-\wp(\a_2))(\wp(z)-\wp(\a_2))}{\wp(z)-\wp(\a_1)}.$$

\end{document}